\documentclass[a4paper,11pt]{amsart}

\usepackage{amsmath, amssymb, amsthm}
\usepackage{thmtools}
\usepackage[margin=1in]{geometry}
\usepackage{url}
\usepackage{bm}
\usepackage{dsfont}
\usepackage{multicol}
\usepackage{ytableau}
\usepackage{tabularray}
\usepackage[parfill]{parskip}
\usepackage{hyperref}
\usepackage{listings}[2013/08/05]
\usepackage{collectbox}

\makeatletter
\newcommand{\mybox}{%
    \collectbox{%
        \setlength{\fboxsep}{1pt}%
        \fbox{\BOXCONTENT}%
    }%
}
\makeatother

\hypersetup{
     colorlinks=true,
     linkcolor=blue,
     filecolor=magenta,
     urlcolor=cyan,
     citecolor=red }

\usepackage{cleveref}

\theoremstyle{plain}
\newtheorem{thm}{Theorem}
\numberwithin{thm}{section}

\newtheorem{lem}[thm]{Lemma}
\newtheorem{prop}[thm]{Proposition}
\newtheorem{cor}[thm]{Corollary}
\newtheorem{rmk}[thm]{Remark}

\theoremstyle{definition}

\newtheorem{defns}[thm]{Definitions}
\newtheorem{conj}[thm]{Conjecture}
\newtheorem{ex}[thm]{Example}

\begin{document}

\title{Consecutive patterns, Kostant's problem and type $A_{6}$}

\author{Samuel Creedon and Volodymyr Mazorchuk}

\date{}

\begin{abstract}
For a permutation $w$ in the symmetric group $\mathfrak{S}_{n}$, let $L(w)$ denote the simple highest weight module in the principal block of the BGG category $\mathcal{O}$ for the Lie algebra $\mathfrak{sl}_{n}(\mathbb{C})$. We first prove that $L(w)$ is Kostant negative whenever $w$ consecutively contains certain patterns. We then provide a complete answer to Kostant's problem in type $A_{6}$ and show that the indecomposability conjecture also 
holds in type $A_{6}$, that is, applying an indecomposable projective functor to a simple module outputs either
an indecomposable module or zero.
\end{abstract}

\maketitle

{\bf Keywords:} Kostant's Problem, Consecutive Patterns, Type $A$. \\
{\bf Mathematics Subject Classification 2010:} Primary 17B10; Secondary 05E10

\section{Introduction}

Elements of the symmetric group 
$\mathfrak{S}_{n}$ index the isomorphism classes  of simple modules
in the principal block of  the Bernstein-Gelfand-Gelfand category $\mathcal{O}$ for the Lie algebra
$\mathfrak{sl}_n$, see \cite{BGG76,Hu08}. For each of these simple  modules one can ask a classical
question, called {\em Kostant's problem}, see \cite{Jo80}, namely, whether the universal enveloping algebra
of $\mathfrak{sl}_n$ surjects onto the algebra of  adjointly finite linear endomorphisms of
this module.  The answer is not known in general, however, a number of special cases are settled,
see \cite{Ma23} for an overview. The problem is important as it provides essential  information about
the action of the monoidal category of projective endofunctors of $\mathcal{O}$ on the simple
module in question, see \cite{KMM23}. Recently it was shown in \cite{CM24} that, asymptotically, the answer 
is negative for almost all modules.

For some small rank cases, the problem is solved. Case $n=2$ is trivial and case $n=3$
is an easy exercise. Cases $n=4$ and $n=5$ are not trivial  but were settled in \cite{KM10} 
(using, in particular, some ideas from
\cite{MS08b}). Case $n=6$ was recently settled in \cite{KMM23}. The latter step required both,
a significant theoretical component developed in \cite{KMM23}, and extensive explicit computations
in the Hecke algebra of 
$\mathfrak{S}_6$ (using SageMath \cite{St25} and the CHEVIE package in GAP3 \cite{Mi15}). The starting point of the present paper
was our attempt at $n=7$, which we completely resolve here. As
in the $n=6$ case, this new step required development of a significant theoretical component 
as well as extensive explicit computations in the Hecke algebra.

Kostant's problem is closely related to some  other open questions about simple highest weight
modules. One of them is the so-called {\em K{\aa}hrstr{\"o}m's Conjecture} which provides
a (conjectural!) combinatorial  reformulation of  Kostant's problem  in terms of the
combinatorics  of the right Kazhdan-Lusztig order on the symmetric group, see \cite{KMM23}. 
Like Kostant's problem, K{\aa}hrstr{\"o}m's conjecture is wide open. One important 
insight of this conjecture is  the relation of Kostant's problem to the 
right Kazhdan-Lusztig order. This reveals a hidden difficulty of Kostant's problem: no combinatorial interpretation of the right Kazhdan-Lusztig order
is known.

Another related open question is the Indecomposability Conjecture from \cite{KiM16} which
asserts that applying an indecomposable projective functor to a simple  highest weight
module outputs either zero or an indecomposable module. The main result  of \cite{KMM23}
connects the three conjectures together. The Indecomposability conjecture is
established, for  $n\leq 6$, in \cite{CMZ19}. In the present paper,  we prove the Indecomposability
Conjecture for $n=7$ and give numerous new examples which satisfy K{\aa}hrstr{\"o}m's Conjecture.

Our arguments are based on a wise preselection before  we dive into a finite brute force
case-by-case analysis. There exist powerful reduction techniques for Kostant's problem based 
on parabolic induction, developed in \cite{Ma05,MS08,K10}. One very important simplification
is that the answer to Kostant's  problem is an invariant of a Kazhdan-Lusztig left cell,
see \cite{MS08}. This implies that it is enough to answer Kostant's problem for
involutions in $\mathfrak{S}_n$, as each left cell contains a unique involution.
There are $232$ involution in $\mathfrak{S}_7$
(see the sequence A000085 in \cite{OEIS}). Parabolic induction (combined with
the known complete solution for all smaller rank cases) deals with $161$ of these 
involution. The recent paper \cite{MMM24} provides a complete solution to Kostant's
problem  for fully commutative elements, which helps with another $29$ cases.
The total of $42$ cases remains.

At the next step we employ a new general theoretical idea.  We observe that consecutive
containment of certain patterns in a permutation necessarily results in the negative 
answer to Kostant's  problem for the corresponding simple  highest weight module.  
In \cite{CM24} we used this observation for the pattern $\mathtt{2143}$ (as it turns out,
the simplest one of those that we discovered) to show that, asymptotically, the answer 
to Kostant's  problem is negative for almost all involutions and for almost all permutations. 
In this paper we prove a
similar  result for other patterns, namely,  for the patterns 
$\mathtt{3142}$, $\mathtt{14325}$, $\mathtt{15324}$, $\mathtt{25314}$ and $\mathtt{24315}$.
This deals with $25$ more cases, leaving $17$ remaining. In fact, up to the natural
symmetry of the root system, we only have $11$ remaining cases, which is fairly manageable.
For these $11$ cases, we do a brute force case-by-cases analysis involving GAP3 
computations.

The utility of computations is justified by K{\aa}hrstr{\"o}m's Conjecture
and its connection to Kostant's problem established in \cite{KMM23}.
Essentially, we need to solve some equations in the Hecke algebra
(or show that these equations do not have any solutions). The equations 
are ``small'', but  they have two parameters and involve both the
Kazhdan-Lusztig and the dual  Kazhdan-Lusztig bases of the Hecke algebra
of   $\mathfrak{S}_7$ (and $\mathfrak{S}_7$ itself  is not  really small).
At the end, it turned out that the answer to Kostant's problem is positive
for exactly $125$ of the $232$ involutions. This result, in itself,  
erased our (not very strong) original hope that Kostant positive
involutions are enumerated by Motzkin numbers (as was suggested
by the known answers  for $n\leq 6$).

The Indecomposability conjecture  is automatic for 
all Kostant positive involutions but it is unclear for Kostant negative 
involutions. The paper \cite{CMZ19} proposes a reduction based on the 
support (the set of simple reflections involved) for the element 
indexing the indecomposable projective functor in question. Using
this reduction and the list of Kostant negative involution, we
reduce the Indecomposability conjecture to $18$ involutions 
indexing the indecomposable projective functors and $22$ involutions
indexing the simple highest weight modules. Most outcomes 
here are zero, however, for each simple highest weight module on
the list, there were a few indecomposable projective functors
which resulted in a non-zero module. In most of these cases
the GAP3 computations combined with a result from \cite{KMM23}
imply that the endomorphism algebra of the resulting module
is positively graded, hence local. So, the module in question
is indecomposable. The few remaining cases are dealt with 
using various ad hoc tricks.

The paper  is organized as follows: All necessary preliminaries are collected
in \Cref{s2}. In \Cref{Sec:KostantsProblemConsecutivePatterns} 
we collected our theoretical results on the role some consecutive patterns
play in Kostant's problem. \Cref{Sec:KostantsProblemA6} contains the
resolution of Kostant's problem for $\mathfrak{S}_7$, while in 
\Cref{Sec:IndecompConjA6} we prove the Indecomposability conjecture for 
$\mathfrak{S}_7$. Lastly, in the appendix we illustrate the use of GAP3 to carry out the
computations reported here.

\subsection*{Acknowledgements}

The first author is partially supported by Vergstiftelsen.
The second author is partially supported by the Swedish Research Council.
All computations were done using the CHEVIE package in GAP3 \cite{Mi15}, with some auxiliary computations done in SageMath \cite{St25}.

\section{Preliminaries}\label{s2}

\subsection{Symmetric Group}

For $n\in\mathbb{Z}_{\geq 0}$, let $\mathfrak{S}_{n}$ be the symmetric group of degree $n$. As a Coxeter group, it is generated by the set of simple reflections $S_{n}:=\{s_{1},\dots, s_{n-1}\}$ subject to the type $A_{n-1}$ braid relations. Let $\mathfrak{I}_{n}:=\{x\in\mathfrak{S}_{n} \ | \ x^{2}=e\}$ denote the subset of \emph{involutions} in $\mathfrak{S}_{n}$.

A tuple/string $(s_{i_{k}},\dots,s_{i_{1}})$ of simple reflections is called a \emph{reduced word of} $x\in\mathfrak{S}_{n}$ if $x=s_{i_{k}}\cdots s_{i_{1}}$ and $k\in\mathbb{Z}_{\geq 0}$ minimal, in which case $k$ is called the \emph{length} of $x$ and we write $\ell(x):=k$. Let $\leq$ denote the \emph{Bruhat order} of $\mathfrak{S}_{n}$, where $x\leq y$ if there exists a substring of some (or every) reduced word of $y$ which gives a reduced word for $x$. The \emph{left and right descent sets} associated to $x$ are 
\[ D_{L}(x):=\{s\in S_{n} \ | \ sx<x\} \hspace{2mm} \text{ and } \hspace{2mm} D_{R}(x):=\{s\in S_{n} \ | \ xs<x\}, \]
respectively. Since $D_{L}(x)=D_{R}(x^{-1})$, when $x\in\mathfrak{I}_{n}$ we will write $D(x):=D_{L}(x)=D_{R}(x)$. The \emph{support} of $x$ is the subset $\mathsf{Sup}(x)\subset S_{n}$ consisting of all simple reflections appearing in a (or every) reduced word of $x$. The (strong right) \emph{Bruhat graph} associated to $\mathfrak{S}_{n}$ is the graph whose vertices are the elements of $\mathfrak{S}_{n}$ and whose edges are the pairs $(x,y)$ where there exists an $s\in S_{n}$ such that $xs=y$. We call such a pair of elements $x$ and $y$ (strong right) \emph{Bruhat neighbours}, and we call a path $(w_{1},\dots,w_{k})\in\mathfrak{S}_{n}^{\times k}$ in this graph a (strong right) \emph{Bruhat walk}. 

We also view $\mathfrak{S}_{n}$ as the group of bijections/permutations of the set $[n]:=\{1,2,\dots,n\}$. We will make use of one-line notation for permutations: for $x\in\mathfrak{S}_{n}$ we write $x=\mathtt{i}_{1}\mathtt{i}_{2}\cdots\mathtt{i}_{n}$ with the latter expression being a string whose $k$-th entry is $\mathtt{i}_{k}:=x(k)$. With such an $x$ and any $y\in\mathfrak{S}_{n}$, we have $xy=\mathtt{i}_{y(1)}\mathtt{i}_{y(2)}\cdots\mathtt{i}_{y(n)}$. Therefore, the right regular action of $\mathfrak{S}_{n}$ corresponds in one-line notation to permuting the positions of the entries $\mathtt{i}_{k}$. One can determine the right descent set of a permutation from its one-line notation by the following well-known equivalence:
\begin{equation}\label{Eq:DescentSetDescription}
xs_{k}<x \iff x(k)=\mathtt{i}_{k}>\mathtt{i}_{k+1}=x(k+1).
\end{equation}

\subsection{The Robinson-Schensted correspondence}

For $n\geq 0$, we let $\Lambda_{n}$ denote the set of all \emph{integer partitions of $n$}, or equivalently the set of \emph{Young diagrams of size $n$}. For any $\lambda\in\Lambda_{n}$, we let $\mathsf{SYT}_{n}(\lambda)$ denote the set of \emph{standard Young tableaux} of shape $\lambda$. Then we denote by
\[ \mathtt{RS}:\mathfrak{S}_{n}\rightarrow\bigsqcup_{\lambda\in\Lambda_{n}}\mathsf{SYT}_{n}(\lambda)\times\mathsf{SYT}_{n}(\lambda) \] 
the \emph{Robinson-Schensted correspondence} (see \cite{Sa01}). This map is a bijection defined by \emph{Schensted's insertion algorithm} from \cite{Sc61}. For $w\in\mathfrak{S}_{n}$ we set $\mathtt{RS}(w):=(\mathtt{P}_{w},\mathtt{Q}_{w})$. It is known that $s_{i}\in D_{R}(w)$ if and only if the row containing $i$ in $\mathtt{Q}_{w}$ is higher than the row containing $i+1$ in $\mathtt{Q}_{w}$. An analogous statement holds for $D_{L}(w)$ and $\mathtt{P}_{w}$. We let $\mathsf{sh}(w)$ be the underlying Young diagram/integer partition of $\mathtt{Q}_{w}$ (or equivalently $\mathtt{P}_{w}$), and let $\preceq$ be the \emph{dominance order} on $\Lambda_{n}$.
\begin{ex}
Consider $w=\mathtt{1524376}\in\mathfrak{S}_{7}$. Then $\mathtt{RS}(w)=(\mathtt{P}_{w},\mathtt{Q}_{w})$ where
\[ \mathtt{P}_{w}={\begin{ytableau}1&2&3&6\\4&7\\5\end{ytableau}}, \hspace{1mm} \text{ and } \hspace{1mm} \mathtt{Q}_{w}={\begin{ytableau}1&2&4&6\\3&7\\5\end{ytableau}}. \]
We have $D_{R}(w)=\{s_{2},s_{4},s_{6}\}$ where, for example, $s_{2}$ belongs to $D_{R}(w)$ since $2$ appears in the first row of $\mathtt{Q}_{w}$ while $3$ appears in the second row. We have that $\mathsf{sh}(w)=(4,2,1)$, and as an example of the dominance order, we have $(3,2,2)\preceq\mathsf{sh}(w)\preceq(5,2)$.
\end{ex} 

\subsection{Hecke Algebra}
\label{SubSec:HeckeAlgebra}

We denote by $\mathcal{H}_{n}$ the \emph{Hecke algebra} associated to $\mathfrak{S}_{n}$ over the ring $\mathbb{Z}[v,v^{-1}]$. We have the standard basis $\{H_{w} \ | \ w\in\mathfrak{S}_{w}\}$ for $\mathcal{H}_{n}$ which satisfies the relations
\begin{equation}\label{Eq:HeckePresentation}
H_{x}H_{y}=H_{xy}, \hspace{4mm} H_{w}H_{s}=(v^{-1}-v)H_{w}+H_{ws}
\end{equation}
for all $x,y\in\mathfrak{S}_{n}$ such that $\ell(xy)=\ell(x)+\ell(y)$, and all $w\in\mathfrak{S}_{n}$ and $s\in D_{R}(w)$. The standard basis and above relations give a presentation for $\mathcal{H}_{n}$ as a $\mathbb{Z}[v,v^{-1}]$-algebra. Consider the group ring $\mathbb{Z}\mathfrak{S}_{n}$ associated to $\mathfrak{S}_{n}$. By \Cref{Eq:HeckePresentation}, we have a ring epimorphism $\mathsf{ev}:\mathcal{H}_{n}\rightarrow\mathbb{Z}\mathfrak{S}_{n}$ given on the standard basis by $H_{w}\mapsto w$ and on the vairable $v\mapsto1$. 

Recalling from \cite{KL79}, we also have the \emph{Kazhdan-Lusztig basis} $\{\underline{H}_{w} \ | \ w\in\mathfrak{S}_{n}\}$. However, we use the normalization in \cite{So07}. The transition matrix between this and the standard basis is unitriangular with respect to (any total refinement of) the Bruhat order, and the non-diagonal coefficients belong to $v\mathbb{Z}[v]$. That is to say, for any $w\in\mathfrak{S}_{n}$, we have the equality
\[ \underline{H}_{w}=H_{w}+\sum_{x<w}p_{x,w}H_{x}, \text{ where } p_{x,y}\in v\mathbb{Z}[v], \]
The \emph{Kazhdan-Lusztig $\mu$-function} is $\mu(x,w)=\mu(w,x):=[v]p_{x,w}$ (the coefficient of $v$ in $p_{x,w}$). For any $w,v\in\mathfrak{S}_{n}$ which are Bruhat neighbours, it is known that $\mu(w,v)=1$. Lastly, we have a bilinear form $(-,-):\mathcal{H}_{n}\times\mathcal{H}_{n}\rightarrow\mathbb{Z}[v,v^{-1}]$ given on any pair $X,Y\in\mathcal{H}_{n}$ by
\[ (X,Y):=[H_{e}]XY, \] 
with $[H_{e}]XY$ denoting the coefficient of $H_{e}$ in $XY$ when expressed in terms of the standard basis. The \emph{dual Kazhdan-Lusztig basis} $\{\underline{\hat{H}}_{w} \ | \ w\in\mathfrak{S}_{n}\}$ of $\mathcal{H}_{n}$ is uniquely defined by $(\underline{H}_{x},\underline{\hat{H}}_{y})=\delta_{x,y^{-1}}$ where $\delta$ is the Kronecker delta. For any permutation $w\in\mathfrak{S}_{n}$ and simple reflection $s\in S_{n}$, we have the well-known equality (see for example \cite[Proposition 46]{CMZ19})
\begin{equation}\label{Eq:DualKLBRightMultHs}
\underline{\hat{H}}_{w}\underline{H}_{s}=
\begin{cases}
\displaystyle (v+v^{-1})\underline{\hat{H}}_{w}+\underline{\hat{H}}_{ws}+\sum_{\substack{x>w \\ xs>x}}\mu(x,w)\underline{\hat{H}}_{x}, & s\in D_{R}(w) \\
0, & s\not\in D_{R}(w)
\end{cases}
\end{equation}
\begin{lem}\label{Rmk:NonNegCoeffPolys}
Given any $w\in\mathfrak{S}_{n}$ and $s_{i_{1}},\dots,s_{i_{k}}\in S_{n}$, we have that
\[ \underline{\hat{H}}_{w}\underline{H}_{s_{i_{1}}}\underline{H}_{s_{i_{2}}}\cdots\underline{H}_{s_{i_{k}}}\in\mathbb{Z}_{\geq 0}[v+v^{-1}]\{\underline{\hat{H}}_{x} \ | \ x\in\mathfrak{S}_{n}\}. \]
\end{lem} 

\begin{proof}
This follows by induction on $k$ and employing \Cref{Eq:DualKLBRightMultHs}.
\end{proof}

\subsection{The Kazhdan-Lusztig Orders}

We let $\leq_{L}$ and $\leq_{R}$ denote the \emph{left} and \emph{right Kazhdan-Lusztig orders} on $\mathfrak{S}_{n}$. Recall, this means $x\leq_{L}y$ if there exists $X\in\mathcal{H}_{n}$ such that $\underline{H}_{y}$ appears with non-zero coefficient in the product $X\underline{H}_{x}$ when expressed in terms of the Kazhdan-Lusztig basis. Here $\leq_{R}$ is defined with $X$ acting on the right instead. It is well-known that $x\leq_{L}y$ if and only if $x^{-1}\leq_{R}y^{-1}$ for all $x,y\in\mathfrak{S}_{n}$, and that for $w\in\mathfrak{S}_{n}$ and $s\in S_{n}$, then
\begin{equation}\label{Eq:RDescentROrder}
ws<w \implies ws\leq_{R}w,
\end{equation}
Also, for any $x,y\in\mathfrak{S}_{n}$, we have by \cite[Theorem 5.1]{Ge06} and \cite[Proposition 2.4]{KL79},
\begin{equation}\label{Eq:ROrderDomDes}
x\leq_{R}y \implies \mathsf{sh}(y)\preceq\mathsf{sh}(x) \text{ and } D_{L}(x)\subset D_{L}(y).
\end{equation}
Let $\sim_{L}$ and $\sim_{R}$ denote the left and right equivalence relations induced from the orders $\leq_{L}$ and $\leq_{R}$ respectively. The equivalence classes associated to these relations are called the \emph{left} and \emph{right cells} respectively. These cells can be combintorially described by the Robinson-Schensted correspondence, in particular,  by \cite[Section 5]{KL79} (see also \cite[Corollary 5.6]{Ge06}) we have
\begin{equation}\label{Eq:CellsRS}
x\sim_{L}y \iff \mathtt{Q}_{x}=\mathtt{Q}_{y} \hspace{1mm} \text{ and }  \hspace{1mm} x\sim_{R}y \iff \mathtt{P}_{x}=\mathtt{P}_{y}.
\end{equation}
It is known that $(\mathtt{P}_{x^{-1}},\mathtt{Q}_{x^{-1}})=(\mathtt{Q}_{x},\mathtt{P}_{x})$, and thus $\mathtt{P}_{x}=\mathtt{Q}_{x}$ if and only if $x$ is an involution. Hence, from \Cref{Eq:CellsRS} above, each left and right cell contains a unique involution.

\subsection{Category $\mathcal{O}$}

Let $\mathfrak{sl}_{n}:=\mathfrak{sl}_{n}(\mathbb{C})$ be the \emph{complex special linear Lie algebra} of all traceless complex $n\times n$ matrices, and consider the \emph{standard triangular decomposition} 
\[ \mathfrak{sl}_{n}=\mathfrak{n}^{-}\oplus\mathfrak{h}\oplus\mathfrak{n}^{+}. \] 
The Weyl group associated to $\mathfrak{sl}_{n}$ is the symmetric group $\mathfrak{S}_{n}$, which acts naturally on the dual space $\mathfrak{h}^{*}$. We let $\rho\in\mathfrak{h}^{*}$ denote the half sum of all positive roots. For $w\in\mathfrak{S}_{n}$ and $\lambda\in\mathfrak{h}^{*}$ we have the \emph{dot-action} $w\cdot\lambda:=w(\lambda+\rho)-\rho$. We write $w_{0}:=\mathsf{n(n-1)\cdots 1}$ to denote the \emph{longest element} of $\mathfrak{S}_{n}$, and we write $\mathcal{U}(\mathfrak{sl}_{n})$ to denote the universal enveloping algebra of $\mathfrak{sl}_{n}$.

Let $\mathcal{O}:=\mathcal{O}(\mathfrak{sl}_{n})$ be the \emph{BGG category} associated to the triangular decomposition of $\mathfrak{sl}_{n}$ above, see \cite{BGG76,Hu08}. The simple objects of $\mathcal{O}$ are (up to isomorphism) the \emph{simple highest weight modules} $L(\lambda)$ for each $\lambda\in\mathfrak{h}^{*}$. The module $L(\lambda)$ is the simple top of the \emph{Verma module} $\Delta(\lambda)$. Consider the \emph{principal block} $\mathcal{O}_{0}:=\mathcal{O}_{0}(\mathfrak{sl}_{n})$, being the indecomposable summand of $\mathcal{O}$ containing the trivial $\mathfrak{sl}_{n}$-module. The simple objects of $\mathcal{O}_{0}$ are in bijection with $\mathfrak{S}_{n}$, given by $L(w):=L(w\cdot 0)$ for $w\in\mathfrak{S}_{n}$. We denote the indecomposable project cover of $L(w)$ in $\mathcal{O}_{0}$ by $P(w)$.

The principal block $\mathcal{O}_{0}$ is equivalent to the left module category for some finite-dimensional basic associative algebra $A$ (defined uniquely, up to isomorphism). This algebra is Koszul by \cite{So90}, and so it admits a Koszul $\mathbb{Z}$-grading. We denote by $\mathcal{O}_{0}^{\mathbb{Z}}$ the corresponding $\mathbb{Z}$-graded version of $\mathcal{O}_{0}$ (see \cite{St03}). For $w\in\mathfrak{S}_{n}$, the modules $L(w)$, $\Delta(w)$, and $P(w)$ admit graded lifts. We use the same notation to denote them. Fix the standard graded lift of $L(w)$ in degree $0$ and the standard graded lifts of $\Delta(w)$ and $P(w)$ so their simple tops are in degree $0$. Let $\langle 1\rangle$ denote the shift functor which sends degree $0$ modules to degree $-1$.

\subsection{Grothendieck Groups $\mathsf{Gr}(\mathcal{O}_{0})$ and $\mathsf{Gr}(\mathcal{O}_{0}^{\mathbb{Z}})$}

Let $\mathsf{Gr}(\mathcal{O}_{0}^{\mathbb{Z}})$ and $\mathsf{Gr}(\mathcal{O}_{0})$ denote the Grothendieck groups of $\mathcal{O}_{0}^{\mathbb{Z}}$ and $\mathcal{O}_{0}$ respectively, with the former being viewed as a $\mathbb{Z}[v,v^{-1}]$-module where $v$ acts by the shift $\langle 1\rangle$. For a module $M$ belonging to either $\mathcal{O}_{0}^{\mathbb{Z}}$ or $\mathcal{O}_{0}$, we will let $[M]$ denote its corresponding image in the respected Grothendieck group. By \cite{BB,BK,So92}, we have a $\mathbb{Z}[v,v^{-1}]$-module isomorphism $\mathsf{Gr}(\mathcal{O}_{0}^{\mathbb{Z}})\xrightarrow{\sim}\mathcal{H}_{n}$ given by
\[ [\Delta(w)]\mapsto H_{w}, \hspace{1mm} [P(w)]\mapsto\underline{H}_{w}, \text{ and } [L(w)]\mapsto\underline{\hat{H}}_{w}. \]
We then obtain a $\mathbb{Z}$-module isomorphism $\mathsf{Gr}(\mathcal{O}_{0})\xrightarrow{\sim}\mathbb{Z}\mathfrak{S}_{n}$ by precomposing the above isomorphism with the ring epimorphism $\mathsf{ev}:\mathcal{H}_{n}\rightarrow\mathfrak{S}_{n}$. As such, the forgetful functor $\mathcal{O}_{0}^{\mathbb{Z}}\rightarrow\mathcal{O}_{0}$, which simply forgets the grading, decategorifies to this ring epimorphism $\mathsf{ev}$.

\subsection{Projective Functors}

A \emph{projective functor} of $\mathcal{O}_{0}$ is a direct summand of $(-\otimes V)$ with $V$ some finite dimensional $\mathfrak{sl}_{n}$-module. Let $\mathcal{P}_{0}:=\mathcal{P}_{0}(\mathfrak{sl}_{n})$ be the monoidal category of such projective functors. By \cite[Theorem 3.3]{BG80}, the isomorphism classes of the indecomposable projective functors are in bijection with the elements of $\mathfrak{S}_{n}$. For each $w\in\mathfrak{S}_{n}$, we denote by $\theta_{w}\in\mathcal{P}_{0}$ the unique (up to isomorphism) indecomposable projective functor normalized so that $\theta_{w}P(e)\cong P(w)$. By \cite[Theorem 8.2]{St03}, indecomposable projective functors admit graded lifts which act on $\mathcal{O}_{0}^{\mathbb{Z}}$, and we use the same notation to denote them. Let $\mathcal{P}_{0}^{\mathbb{Z}}$ denote the corresponding $\mathbb{Z}$-graded version of $\mathcal{P}_{0}$. 

We summarise various properties of projective functors which will come into play later in this paper. Firstly, given $x,y\in\mathfrak{S}_{n}$, by \cite[Lemma~12]{MM11}, we have the equivalence
\begin{equation}\label{Eq:ThetaLNonZero}
\theta_{x}L(y)\neq 0 \iff x\leq_{R}y^{-1}.
\end{equation}
Combining \cite[Theorem~2.2]{KMM23} with Conjectures 14.2 P8 and Section 15 in \cite{Lu03}, we have
\begin{equation}\label{Eq:ThetaxL=ThetayLImpliesxly}
\theta_{x}L(w)\cong\theta_{y}L(w)\neq 0 \implies x\sim_{L}y.
\end{equation}
Lastly, for the following result, see for example \cite[Proposition 46]{CMZ19}:
\begin{prop}\label{Prop:JantsenMiddleStructure}
Let $w\in\mathfrak{S}_{n}$ and $s\in S_{n}$. Then $\theta_{s}L(w)\neq0$ if and only if $ws<w$. In this case $\theta_{s}L(w)$ is indecomposable, of graded length three, has simple top $L_{w}$ in degree $-1$, simple socle $L_{w}$ in degree $1$, and semi-simple module $\mathsf{J}_{s}(w)$ in degree zero \text{(}called the \emph{Jantzen middle}\text{)} where 
\[ \mathsf{J}_{s}(w)\cong L(ws)\oplus\bigoplus_{\substack{x>w \\ xs>x}}L(x)^{\oplus \mu(w,x)}. \]
\end{prop}

\subsection{Grothendieck Rings $\mathsf{Gr}(\mathcal{P}_{0})$ and $\mathsf{Gr}(\mathcal{P}_{0}^{\mathbb{Z}})$}

Let $\mathsf{Gr}(\mathcal{P}_{0}^{\mathbb{Z}})$ and $\mathsf{Gr}(\mathcal{P}_{0})$ denote the split Grothendieck rings of $\mathcal{P}_{0}^{\mathbb{Z}}$ and $\mathcal{P}_{0}$ respectively, with the former being viewed as a $\mathbb{Z}[v,v^{-1}]$-algebra where $v$ acts by the shift $\langle 1\rangle$. For a projective functor $\theta$ belonging to either $\mathcal{P}_{0}^{\mathbb{Z}}$ or $\mathcal{P}_{0}$, we will let $[\theta]$ denote the corresponding image in the respected split Grothendieck ring. By, for example \cite{So92, Ma12}, we have a $\mathbb{Z}[v,v^{-1}]$-algebra isomorphism
\[ \mathsf{Gr}(\mathcal{P}_{0}^{\mathbb{Z}})\xrightarrow{\sim}\mathcal{H}_{n}^{\text{op}} \hspace{1mm} \text{ defined by } \hspace{1mm} [\theta_{w}]\mapsto\underline{H}_{w}. \]
Precomposing with the epimorphism $\mathsf{ev}:\mathcal{H}_{n}\rightarrow\mathfrak{S}_{n}$ gives a $\mathbb{Z}$-algebra isomorphism $\mathsf{Gr}(\mathcal{P}_{0})\xrightarrow{\sim}\mathbb{Z}\mathfrak{S}_{n}^{\text{op}}$. In particular, the forgetful functor $\mathcal{P}_{0}^{\mathbb{Z}}\rightarrow\mathcal{P}_{0}$ decategorifies to $\mathsf{ev}$. Furthermore, the natural action of $\mathcal{P}_{0}^{\mathbb{Z}}$ on $\mathcal{O}_{0}^{\mathbb{Z}}$, and $\mathcal{P}_{0}$ on $\mathcal{O}_{0}$, decategorify to the regular right action of $\mathcal{H}_{n}$ and $\mathbb{Z}\mathfrak{S}_{n}$ respectively.

\subsection{Parabolic Counterparts}
\label{SubSec:CategoriesOl}

Everything covered in the previous subsections have direct analogs for semi-simple Levi factors of parabolic subalgebras, as we summarise here: A fixed subset $I\subset S_{n}$ of simple reflections naturally determines a parabolic subalgebra $\mathfrak{p}\subset\mathfrak{sl}_{n}$. Let $\mathfrak{l}$ denote the semi-simple Levi factor of $\mathfrak{p}$. Let $I=I_{1}\sqcup\cdots\sqcup I_{k}$ be the decomposition where each $I_{i}\neq\emptyset$ consists of consecutive simple reflections, $k$ is minimal, and the indices of the simple reflections in $I_{i}$ are all less than those in $I_{j}$ if $i<j$. For each $1\leq i\leq k$, set $n_{i}:=|I_{i}|+1$. Then we have a natural isomorphism
\[ \mathfrak{l}\cong\mathfrak{sl}_{n_{1}}\times\cdots\times\mathfrak{sl}_{n_{k}} \hspace{2mm} \text{(where $n_{i}\geq 2$ and $n_{1}+\cdots+n_{k}=|I|+k$)}. \]
The Weyl group of $\mathfrak{l}$ is the subgroup $\mathfrak{S}_{n}(I):=\langle I \rangle\subset\mathfrak{S}_{n}$ generated by $I$. Let $w_{0}^{I}$ denote the longest element in $\mathfrak{S}_{n}(I)$, and $\mathsf{X}(I)$ the set of minimal coset representatives of the cosets $\mathfrak{S}_{n}/\mathfrak{S}_{n}(I)$. The above isomorphism induces one for the Weyl groups $\phi:\mathfrak{S}_{n}(I)\xrightarrow{\sim}\mathfrak{S}_{n_{1}}\times\cdots\times\mathfrak{S}_{n_{k}}$. For $w\in\mathfrak{S}_{n}(I)$, set $\phi(w)=:(w_{1},\dots,w_{k})\in\mathfrak{S}_{n_{1}}\times\cdots\times\mathfrak{S}_{n_{k}}$. The isomorphism $\phi$ sends the simple reflections in $I_{i}$ to those in the $i$-th factor $\mathfrak{S}_{n_{i}}$ such that the order of the indices is preserved. 

Let $\mathcal{O}(\mathfrak{l})$ be the BGG category associated to $\mathfrak{l}$ and its induced standard triangular decomposition. Let $\mathcal{O}_{0}(\mathfrak{l})$ be the principal block of $\mathcal{O}(\mathfrak{l})$ and $\mathcal{O}_{0}^{\mathbb{Z}}(\mathfrak{l})$ its $\mathbb{Z}$-graded counterpart. We denote the simple objects of $\mathcal{O}_{0}(\mathfrak{l})$ by $L_{\mathfrak{l}}(w)$ for each $w\in\mathfrak{S}_{n}(I)$, and we use the same notation to denote their standard graded lifts in $\mathcal{O}_{0}^{\mathbb{Z}}(\mathfrak{l})$. From the above isomorphism, we have an equivalence of categories
\[ \mathcal{F}:\mathcal{O}_{0}^{\mathbb{Z}}(\mathfrak{l})\xrightarrow{\sim}\mathcal{O}_{0}^{\mathbb{Z}}(\mathfrak{sl}_{n_{1}})\times\cdots\times\mathcal{O}_{0}^{\mathbb{Z}}(\mathfrak{sl}_{n_{k}}) \]
where $\mathcal{F}(L_{\mathfrak{l}}(w))=L(w_{1})\boxtimes\cdots\boxtimes L(w_{k})$. We let $\mathcal{P}_{0}(\mathfrak{l})$ denote the monoidal category of projective functors of $\mathcal{O}_{0}(\mathfrak{l})$, and $\mathcal{P}_{0}^{\mathbb{Z}}(\mathfrak{l})$ its $\mathbb{Z}$-graded counterpart. For each $w\in\mathfrak{S}_{n}(I)$, let $\theta_{w}\in\mathcal{P}_{0}^{\mathbb{Z}}(\mathfrak{l})$ be the corresponding indecomposable projective endofunctor of $\mathcal{O}_{0}^{\mathbb{Z}}(\mathfrak{l})$. We have an equivalence
\[ \mathcal{G}:\mathcal{P}_{0}^{\mathbb{Z}}(\mathfrak{l})\xrightarrow{\sim}\mathcal{P}_{0}^{\mathbb{Z}}(\mathfrak{sl}_{n_{1}})\times\cdots\times\mathcal{P}_{0}^{\mathbb{Z}}(\mathfrak{sl}_{n_{k}}) \]
where $\mathcal{G}(\theta_{w})=\theta_{w_{1}}\boxtimes\cdots\boxtimes\theta_{w_{k}}$. Moreover, the functors $\mathcal{F}$ and $\mathcal{G}$ are compatible in the sense that, for any $x,w\in\mathfrak{S}_{n}(I)$, we have $\mathcal{F}(\theta_{x}L_{\mathfrak{l}}(w))\cong\mathcal{G}(\theta_{x})\mathcal{F}(L_{\mathfrak{l}}(w))$. 

Let $\mathcal{H}_{n}(I)$ denote the Hecke algebra associated to $\mathfrak{S}_{n}(I)$, which is the subalgebra of $\mathcal{H}_{n}$ generated by $H_{w}$ for $w\in\mathfrak{S}_{n}(I)$. Induced from above, we have an isomorphism of $\mathbb{Z}[v,v^{-1}]$-algebras
\[ \mathcal{H}_{n}(I)\xrightarrow{\sim}\mathcal{H}_{n_{1}}\otimes\cdots\otimes\mathcal{H}_{n_{k}} \]
given by $H_{w}\mapsto H_{w_{1}}\otimes\cdots\otimes H_{w_{k}}$, with analogous images for the Kazhdan-Lusztig and dual Kazhdan-Lusztig bases. The Grothendieck group $\mathsf{Gr}(\mathcal{O}_{0}^{\mathbb{Z}}(\mathfrak{l}))$ is isomorphic to $\mathcal{H}_{n}(I)$ as a $\mathbb{Z}[v,v^{-1}]$-module, while the split Grothendieck ring $\mathsf{Gr}(\mathcal{P}_{0}^{\mathbb{Z}}(\mathfrak{l}))$ is isomorphic to $\mathcal{H}_{n}(I)^{\mathsf{op}}$ as a $\mathbb{Z}[v,v^{-1}]$-algebra. The natural action of $\mathcal{P}_{0}^{\mathbb{Z}}(\mathfrak{l})$ on $\mathcal{O}_{0}^{\mathbb{Z}}(\mathfrak{l})$ decategorifies to the right regular representation of $\mathcal{H}_{n}(I)$. 

\subsection{Indecomposability Conjecture}

Given $x,y\in\mathfrak{S}_{n}$, let $\mathsf{KM}(x,y)\in\{\mathtt{true}, \mathtt{false}\}$ denote the \emph{truth value} of the statement ``the module  $\theta_{x}L(y)$ viewed within $\mathcal{O}_{0}^{\mathbb{Z}}$ (or equivalently within $\mathcal{O}_{0}$) is either zero or indecomposable''. Furthermore, let $\mathsf{KM}(x,\star)$ denote the conjunction of $\mathsf{KM}(x,y)$ for all $y\in\mathfrak{S}_{n}$, and define $\mathsf{KM}(\star,y)$ similarly. Then the \emph{Indecomposability Conjecture} is as follows:
\vspace{-8mm}
\begin{conj}(\cite[Conjecture 2]{KiM16})\label{Conj:IndecompConj}
\vspace{-3mm}
We have $\mathsf{KM}(\star,y)=\mathtt{true}$ for all $y\in\mathfrak{S}_{n}$.
\end{conj}
The Indecomposability conjecture was studied extensively in \cite{CMZ19}. We summarise some of their results here. Firstly, they proved $\mathsf{KM}(x,y)=\mathtt{true}$ for all $x,y\in\mathfrak{S}_{n}$ when $1\leq n\leq6$.
\begin{prop}\emph{(}\cite[Equation 12, Proposition 2]{CMZ19}\emph{)}\label{Prop:IndecompConjLRCellInv}
The following hold:
\begin{itemize}
\item[(a)] For all $x,x'\in\mathfrak{S}_{n}$ such that $x\sim_{R}x'$, we have $\mathsf{KM}(x,\star)=\mathsf{KM}(x',\star)$.
\item[(b)] For all $y,y'\in\mathfrak{S}_{n}$ such that $y\sim_{L}y'$, we have $\mathsf{KM}(\star,y)=\mathsf{KM}(\star,y')$.
\item[(c)] For all $x,y\in\mathfrak{S}_{n}$, we have $\mathsf{KM}(x,y)=\mathsf{KM}(y^{-1}w_{0},w_{0}x^{-1})$.
\end{itemize}
\end{prop}

The following result is \cite[Corollary 43]{CMZ19}, which has been specialised to our setting:

\begin{thm}\label{Thm:IndecomposabilityFromParabolic}
Let $I\subset S_{n}$ and $\mathfrak{l}\subset\mathfrak{sl}_{n}$, as in 
\Cref{SubSec:CategoriesOl}. For $x,y\in\mathfrak{S}_{n}(I)$ and $z\in\mathsf{X}(I)$, then $\theta_{x}L_{\mathfrak{l}}(y)$ in $\mathcal{O}_{0}^{\mathbb{Z}}(\mathfrak{l})$ is indecomposable if and only if $\theta_{x}L(zy)$ in $\mathcal{O}_{0}^{\mathbb{Z}}$ is indecomposable.
\end{thm}

\begin{rmk}
In \cite{CMZ19} they proved the result above by showing that the endomorphism algebras of $\theta_{x}L_{\mathfrak{l}}(y)$ in $\mathcal{O}_{0}^{\mathbb{Z}}(\mathfrak{l})$ and $\theta_{x}L(zy)$ in $\mathcal{O}_{0}^{\mathbb{Z}}$ are isomorphic. As such, $\theta_{x}L_{\mathfrak{l}}(y)\neq0$ if and only if $\theta_{x}L(zy)\neq0$.
\end{rmk}

\begin{cor}\label{Cor:KMxSmallSup}
Let $x\in\mathfrak{S}_{n}$ be such that $|\mathsf{Sup}(x)|\leq5$, then $\mathsf{KM}(x,\star)=\mathtt{true}$.
\end{cor}

\begin{proof}
Let $w\in\mathfrak{S}_{n}$ and assume $\theta_{x}L(w)\neq 0$. We seek to prove that $\theta_{x}L(w)$ is indecomposable. Let $I=\mathsf{Sup}(x)$ and recall the notation given in 
\Cref{SubSec:CategoriesOl}. Then $w=zy$ for a unique $y\in\mathfrak{S}_{n}(I)$ and $z\in\mathsf{X}(I)$. By \Cref{Thm:IndecomposabilityFromParabolic}, the module $\theta_{x}L(w)=\theta_{x}L(zy)$ is indecomposable if and only if the module $\theta_{x}L_{\mathfrak{l}}(y)\neq0$ is indecomposable. Now $\theta_{x}L_{\mathfrak{l}}(y)$ is indecomposable if and only if
\[ \mathcal{F}(\theta_{x}L_{\mathfrak{l}}(y))=\theta_{x_{1}}L(y_{1})\boxtimes\cdots\boxtimes\theta_{x_{k}}L(y_{k}) \] 
is indecomposable, which is the case if and only if $\theta_{x_{i}}L(y_{i})$ in $\mathcal{O}_{0}^{\mathbb{Z}}(\mathfrak{sl}_{n_{i}})$ is indecomposable for each $i\in[k]$. The indecomposability conjecture holds for $n\in[6]$, so $\theta_{x_{i}}L(y_{i})$ in $\mathcal{O}_{0}^{\mathbb{Z}}(\mathfrak{sl}_{n_{i}})$ is either zero or indecomposable as $n_{i}=|\mathsf{Sup}(x_{i})|+1\leq |\mathsf{Sup}(x)|+1=6$, and they are non-zero as $\theta_{x}L_{\mathfrak{l}}(y)\neq 0$.
\end{proof}

\subsection{Kostant's Problem}

For $M$ an $\mathfrak{sl}_{n}$-module, the space of linear maps $\text{Hom}_{\mathbb{C}}(M,M)$ admits a $\mathcal{U}(\mathfrak{sl}_{n})$-bimodule structure, and so a $\mathfrak{sl}_{n}$-module structure via the adjoint action. The $\mathfrak{sl}_{n}$-submodule $\mathcal{L}(M,M)\subset\text{Hom}_{\mathbb{C}}(M,M)$ of locally finite maps is preserved under the adjoint $\mathfrak{sl}_{n}$-action. As $\mathcal{U}(\mathfrak{sl}_{n})$ itself is locally finite under the adjoint $\mathfrak{sl}_{n}$-action, we have a ring homomorphism $\mathcal{U}(\mathfrak{sl}_{n})\rightarrow\mathcal{L}(M,M)$. In \cite{Jo80} the following question was posed under the name \emph{Kostant problem}:
\begin{center}
\emph{When is the ring homomorphism} $\mathcal{U}(\mathfrak{sl}_{n})\rightarrow\mathcal{L}(M,M)$ \emph{surjective?}
\end{center}
For $w\in\mathfrak{S}_{n}$, let $\mathsf{K}(w)\in\{\mathtt{true}, \mathtt{false}\}$ be the truth value of the statement ``the ring homomorphism $\mathcal{U}(\mathfrak{sl}_{n})\rightarrow\mathcal{L}(L(w),L(w))$ is surjective''. We call $w$ \emph{Kostant positive} if $\mathsf{K}(w)=\mathtt{true}$, and \emph{Kostant negative} otherwise. In \cite[Section 4]{KM10}, Kostant's problem has been answered in full for $\mathfrak{sl}_{n}$ with $1\leq n\leq5$. A partial answer for $\mathfrak{sl}_{6}$ is given in \cite{KM10,K10}, and later completed in \cite{KMM23}. 

The results of this paper focus on the following equivalent reformulation of $\mathsf{K}(w)$: Firstly, denote by $\mathsf{Kh}(w)\in\{\mathtt{true}, \mathtt{false}\}$ the truth value of the statement ``For any distinct $x,y\in\mathfrak{S}_{n}$ such that $\theta_{x}L(w), \theta_{y}L(w)\neq0$, then $\theta_{x}L(w)\ncong\theta_{y}L(w)$ in $\mathcal{O}_{0}^{\mathbb{Z}}$''. Then we have the following:
\begin{thm}\emph{(}\cite[Theorem 8.16]{KMM23}\emph{)}\label{Thm:KiffKhKM}
For any $w\in\mathfrak{S}_{n}$, we have the equality of truth values
\begin{equation}\label{Eq:KiffKhKM}
\mathsf{K}(w)=\mathsf{Kh}(w)\wedge\mathsf{KM}(\star,w),
\end{equation}
where $\wedge$ denotes the conjunction operator on truth values.
\end{thm}
Therefore, to investigate Kostant's problem for a given permutation $w\in\mathfrak{S}_{n}$, it suffices to confirm that $\theta_{x}L(w)$ is zero or indecomposable, and to compare whether any such modules are isomorphic or not as $x$ varies over $\mathfrak{S}_{n}$. From (b) of \Cref{Prop:IndecompConjLRCellInv}, we know that $\mathsf{KM}(\star,w)$ is left cell invariant. It was shown in \cite{MS08} that the same holds true for $\mathsf{K}(w)$:
\begin{thm}\emph{(}\cite[Theorem 61]{MS08}\emph{)}\label{Thm:KLeftCellInv}
For $w,w'\in\mathfrak{S}_{n}$ such that $w\sim_{L}w'$, then $\mathsf{K}(w)=\mathsf{K}(w')$.
\end{thm}
Lastly, for $I\subset S_{n}$ and $\mathfrak{l}\subset\mathfrak{sl}_{n}$ as in \Cref{SubSec:CategoriesOl}, let $\mathsf{K}_{\mathfrak{l}}$, $\mathsf{Kh}_{\mathfrak{l}}$, and $\mathsf{KM}_{\mathfrak{l}}$ denote the $\mathfrak{l}$-counterparts of $\mathsf{K}$, $\mathsf{Kh}$, and $\mathsf{KM}$ respectively. So we also have the equality of true values $\mathsf{K}_{\mathfrak{l}}(w)=\mathsf{Kh}_{\mathfrak{l}}(w)\wedge\mathsf{KM}_{\mathfrak{l}}(\star,w)$. Then for all $w\in\mathfrak{S}_{n}(I)$, the following equality of true values was proved in \cite[Theorem~1.1]{K10}:
\begin{equation}\label{Thm:ParaLiftK}
\mathsf{K}_{\mathfrak{l}}(w)=\mathsf{K}(ww_{0}^{I}w_{0}).
\end{equation}

\subsection{K{\aa}hrstr\"om's Conjecture}

We now recall \emph{K{\aa}hrstr\"om's Conjecture}, which was first published in \cite[Conjecture 1.2]{KMM23}. This conjecture equates $\mathsf{K}$ with $\mathsf{Kh}$ and decategorified versions, when restricted to involutions. In particular, it suggests that Kostant's problem can be entirely reformulated with combinatorics associated to the Hecke algebra.
\begin{conj}\label{Conj:KahrstromsConjecture}
For an involution $d\in\mathfrak{I}_{n}$, the following assertions are equivalent:
\begin{itemize}
\item[(a)] The involution $d$ is Kostant positive.
\item[(b)] For all $x\neq y\in\mathfrak{S}_{n}$ with $\theta_{x}L(d),\theta_{y}L(d)\neq0$, then $\theta_{x}L(d)\not\cong\theta_{y}L(d)$ in $\mathcal{O}_{0}^{\mathbb{Z}}$.
\item[(c)] For all $x\neq y\in\mathfrak{S}_{n}$ with $\underline{\hat{H}}_{d}\underline{H}_{x},\underline{\hat{H}}_{d}\underline{H}_{y}\neq0$, then $\underline{\hat{H}}_{d}\underline{H}_{x}\neq\underline{\hat{H}}_{d}\underline{H}_{y}$ in $\mathcal{H}_{n}$.
\item[(d)] For all $x\neq y\in\mathfrak{S}_{n}$ with $\mathsf{ev}(\underline{\hat{H}}_{d}\underline{H}_{x}),\mathsf{ev}(\underline{\hat{H}}_{d}\underline{H}_{y})\neq0$, then $\mathsf{ev}(\underline{\hat{H}}_{d}\underline{H}_{x})\neq\mathsf{ev}(\underline{\hat{H}}_{d}\underline{H}_{y})$ in $\mathbb{Z}\mathfrak{S}_{n}$.
\end{itemize} 
\end{conj}  
Statement (b) above is precisely $\mathsf{Kh}(d)$. For $w\in\mathfrak{S}_{n}$, let $[\mathsf{Kh}](w)$ denote the truth value of statement (c) above and let $[\mathsf{Kh}^{\mathsf{ev}}](w)$ denote the truth value of statement (d) above (replacing $d$ with $w$). As an equality of true values, K{\aa}hrstr\"om's conjecture is the claim that for all $d\in\mathfrak{I}_{n}$, we have
\begin{equation}\label{Eq:KhConj}
\mathsf{K}(d)=\mathsf{Kh}(d)=[\mathsf{Kh}](d)=[\mathsf{Kh}^{\mathsf{ev}}](d).
\end{equation}
This conjecture was a significant driving force behind most of the results for this paper. As $[\mathsf{Kh}](d)$ and $[\mathsf{Kh}^{\mathsf{ev}}](d)$ strictly concern Hecke algebra and symmetric group combinatorics, it can thus be checked computationally for small cases. Such computations motivated the results of \Cref{Sec:KostantsProblemConsecutivePatterns} and played a significant role 
in \Cref{Sec:KostantsProblemA6}. In particular, we often prove $\mathsf{K}(d)=\mathtt{false}$ by proving $\mathsf{Kh}(d)=\mathtt{false}$, which comes from knowing that $[\mathsf{Kh}](d)=\mathtt{false}$ via computations.

\section{Kostant's Problem and Consecutive Patterns}
\label{Sec:KostantsProblemConsecutivePatterns}

Let $m\leq n$, $w=\mathtt{i}_{1}\cdots\mathtt{i}_{n}\in\mathfrak{S}_{n}$, and $p=\mathtt{j}_{1}\cdots\mathtt{j}_{m}\in\mathfrak{S}_{m}$. We say $w$ \emph{contains} $p$ \emph{as a pattern} if there exists a subsequence $\mathtt{i}_{a_{1}}\cdots\mathtt{i}_{a_{m}}$ (so $a_{1}<\cdots<a_{m}$ with $a_{i}\in[n]$) which has the same relative order as $\mathtt{j}_{1}\cdots\mathtt{j}_{m}$. As a special case, $w$ \emph{consecutively contains} $p$ \emph{as a pattern} if the entries in the subsequence $\mathtt{i}_{a_{1}}\cdots\mathtt{i}_{a_{m}}$ occur in consecutive positions, that is, if $a_{k+1}=a_{k}+1$ for $1\leq k<m$. 

\begin{ex}\label{Ex:PatternContainment}
Consider $w=\mathtt{1524376}\in\mathfrak{S}_{7}$ and $p=\mathtt{2143}\in\mathfrak{S}_{4}$. Then $w$ contains $p$ as a pattern in four different ways, which we list here by underlying the respective subsequence:
\[ \mathtt{1\underline{5}\underline{2}43\underline{76}}, \hspace{3mm} \mathtt{1\underline{5}2\underline{4}3\underline{76}}, \hspace{3mm} \mathtt{1\underline{5}24\underline{376}}, \hspace{3mm} \mathtt{152\underline{4376}}. \]
The latter subsequence $\mathtt{4376}$ demonstrates a consecutive containment.
\end{ex}

\begin{defns}\label{Defns:SpecialRDCompatibleSequence}
Let $\underline{x}:=(s_{i_{k}},\dots,s_{i_{1}})$ be a reduced word. We say a Bruhat walk $(w_{1},\dots,w_{k})$ is:
\begin{itemize}
\item[(a)] \emph{weakly compatible} with $\underline{x}$ if $w_{j}s_{i_{j}}<w_{j}$ and $w_{j}s_{i_{j-1}}>w_{j}$ for all valid $j$.
\item[(b)] \emph{compatible} with $\underline{x}$ if $w_{j}s_{i_{j}}<w_{j}$ and $w_{j}s_{i_{j\pm1}}>w_{j}$ for all valid $j$.
\end{itemize}
\end{defns}

\begin{rmk}
In the above definition, we stress that the reduced word $\underline{x}=(s_{i_{k}},\dots,s_{i_{1}})$ is indexed from right to left, while the Bruhat walk $(w_{1},\dots,w_{k})$ is indexed left to right.
\end{rmk}

\begin{ex}
\label{Ex:SpecialRDCompatibleSequence}
In the following two examples, we focus on employing \Cref{Eq:DescentSetDescription} to deduce the Bruhat relations. We do this to help with the readability of 
\Cref{Thm:PatternsKostantNegative} below.
\begin{itemize}
\item[(a)] Let $w=\mathtt{14325}\in\mathfrak{S}_{5}$. For the reduced word $\underline{x}=(s_{i_{4}},s_{i_{3}},s_{i_{2}},s_{i_{1}}):=(s_{2},s_{3},s_{1},s_{2})$ and Bruhat walk $(w_{1},w_{2},w_{3},w_{4}):=(w,ws_{1},w,ws_{3})$, we show $w_{j}s_{i_{j}}<w_{j}$ for all $1\leq j\leq 4$:
\begin{align*}
w_{1}s_{i_{1}}=ws_{2}<w=w_{1} \hspace{2mm} &\text{(since $w(2)=4>3=w(3)$)}, \\
w_{2}s_{i_{2}}=(ws_{1})s_{1}<ws_{1}=w_{2} \hspace{2mm} &\text{(since $(ws_{1})(1)=4>1=(ws_{1})(2)$)}, \\
w_{3}s_{i_{3}}=ws_{3}<w=w_{3} \hspace{2mm} &\text{(since $w(3)=3>2=w(4)$)}, \\
w_{4}s_{i_{4}}=(ws_{3})s_{2}<ws_{3}=w_{4} \hspace{2mm} &\text{(since $(ws_{3})(2)=4>2=(ws_{3})(3)$)}.
\end{align*}
We now show that $w_{j}s_{i_{j-1}}>w_{j}$ for all $1<j\leq 4$:
\begin{align*}
w_{2}s_{i_{1}}=(ws_{1})s_{2}>ws_{1}=w_{2} \hspace{2mm} &\text{(since $(ws_{1})(2)=1<3=(ws_{1})(3)$)}, \\
w_{3}s_{i_{2}}=ws_{1}>w=w_{3} \hspace{2mm} &\text{(since $w(1)=1<4=w(2)$)}, \\
w_{4}s_{i_{3}}=(ws_{3})s_{3}>ws_{3}=w_{4} \hspace{2mm} &\text{(since $(ws_{3})(3)=2<3=(ws_{3})(4)$)}.
\end{align*}
Therefore $(w,ws_{1},w,ws_{3})$ is weakly compatible with $\underline{x}=(s_{2},s_{3},s_{1},s_{2})$.
\item[(b)] Let $w=\mathtt{3142}\in\mathfrak{S}_{4}$, and consider the reduced word $\underline{x}=(s_{i_{3}},s_{i_{2}},s_{i_{1}}):=(s_{1},s_{2},s_{3})$ and Bruhat walk $(w_{1},w_{2},w_{3}):=(w,ws_{2},w)$. We first show that $w_{j}s_{i_{j}}<w_{j}$ for all $1\leq j\leq3$:
\begin{align*}
w_{1}s_{i_{1}}=ws_{3}<w=w_{1} \hspace{2mm} &\text{(since $w(3)=4>2=w(4)$)}, \\
w_{2}s_{i_{2}}=(ws_{2})s_{2}<ws_{2}=w_{2} \hspace{2mm} &\text{(since $(ws_{2})(2)=4>1=(ws_{2})(3)$)}, \\
w_{3}s_{i_{3}}=ws_{1}<w=w_{3} \hspace{2mm} &\text{(since $w(1)=3>1=w(2)$)}.
\end{align*}
We now show that $w_{j}s_{i_{j-1}}>w_{j}$ for all $1<j\leq 3$ and $w_{j}s_{i_{j+1}}>w_{j}$ for all $1\leq j<3$:
\begin{align*}
w_{2}s_{i_{1}}=(ws_{2})s_{3}>ws_{2}=w_{2} \hspace{2mm} &\text{(since $(ws_{2})(3)=1<2=(ws_{2})(4)$)}, \\
w_{3}s_{i_{2}}=ws_{2}>w=w_{3} \hspace{2mm} &\text{(since $w(2)=1<4=w(3)$)}, \\
w_{1}s_{i_{2}}=ws_{2}>w=w_{1} \hspace{2mm} &\text{(same as the previous case)}, \\
w_{2}s_{i_{3}}=(ws_{2})s_{1}>ws_{2}=w_{2} \hspace{2mm} &\text{(since $(ws_{2})(1)=3<4=(ws_{2})(2)$)}.
\end{align*}
Therefore $(w,ws_{2},w)$ is compatible with $\underline{x}=(s_{1},s_{2},s_{3})$.
\end{itemize} 
\end{ex}

\begin{lem}\label{Lem:SpecialRDCompatibleSequence}
Let $(w_{1},\dots,w_{k})$ be a Bruhat walk and $\underline{x}=(s_{i_{k}},\dots,s_{i_{1}})$ a reduced word.
\begin{itemize}
\item[(i)] If $(w_{1},\dots,w_{k})$ is weakly compatible with $\underline{x}$ then $\theta_{s_{i_{k}}}\hspace{-2mm}\cdots\hspace{1mm}\theta_{s_{i_{1}}}L(w_{1})\neq 0$.
\item[(ii)] If $(w_{1},\dots,w_{k})$ is compatible with $\underline{x}$, then $\theta_{s_{i_{k}}}\hspace{-2mm}\cdots\hspace{1mm}\theta_{s_{i_{1}}}L(w_{1})\cong\theta_{s_{i_{k}}}L(w_{k})\neq 0$.
\end{itemize}
\end{lem}

\begin{proof}
Item (i): It suffices to prove that $[\theta_{s_{i_{k}}}\hspace{-2mm}\cdots\hspace{1mm}\theta_{s_{i_{1}}}L(w_{1})]=\underline{\hat{H}}_{w_{1}}\underline{H}_{s_{i_{1}}}\hspace{-2mm}\cdots\underline{H}_{s_{i_{k}}}\neq 0$. For $1\leq j<k$, since $w_{j}s_{i_{j}}<w$, then by \Cref{Eq:DualKLBRightMultHs} we have that
\[ \underline{\hat{H}}_{w_{j}}\underline{H}_{s_{i_{j}}}=(v+v^{-1})\underline{\hat{H}}_{w_{j}}+\underline{\hat{H}}_{w_{j}s_{i_{j}}}+\sum_{\substack{w>w_{j} \\ ws_{i_{j}}>w}}\mu(w_{j},w)\underline{\hat{H}}_{w}. \]
We know $w_{j+1}s_{i_{j}}>w_{j+1}$ since $(w_{1},\dots,w_{k})$ is weakly compatible with $\underline{x}$, and that $\mu(w_{j},w_{j+1})=1$ since $w_{j}$ and $w_{j+1}$ are Bruhat neighbours. Hence from the above equality we see that
\begin{equation}
\label{Eq:WeaklyCompatibleDKLBasisAppearence1}
[\underline{\hat{H}}_{w_{j+1}}]\underline{\hat{H}}_{w_{j}}\underline{H}_{s_{i_{j}}}\neq0.
\end{equation}
In other words, $\underline{\hat{H}}_{w_{j+1}}$ appears with a non-zero coefficient in the product $\underline{\hat{H}}_{w_{j}}\underline{H}_{s_{i_{j}}}$ when expressed in terms of the dual Kazhdan-Lusztig basis. From \Cref{Rmk:NonNegCoeffPolys}, applying \Cref{Eq:WeaklyCompatibleDKLBasisAppearence1} inductively tells us that $\underline{\hat{H}}_{w_{k}}$ appears with non-zero coefficient in the product $\underline{\hat{H}}_{w_{1}}\underline{H}_{s_{i_{1}}}\cdots\underline{H}_{s_{i_{k-1}}}$ when expressed in terms of the dual Kazhdan-Lusztig basis. Therefore, we must have that
\[ (\underline{\hat{H}}_{w_{1}}\underline{H}_{s_{i_{1}}}\cdots\underline{H}_{s_{i_{k-1}}})\underline{H}_{s_{i_{k}}}=(\underline{\hat{H}}_{w_{k}}+X)\underline{H}_{s_{i_{k}}}=\underline{\hat{H}}_{w_{k}}\underline{H}_{s_{i_{k}}}+X\underline{H}_{s_{i_{k}}}\in\mathbb{Z}_{\geq 0}[v+v^{-1}]\{\underline{\hat{H}}_{w} \ | \ w\in W\} \]
for some $X\in\mathbb{Z}_{\geq 0}[v+v^{-1}]\{\underline{\hat{H}}_{w} \ | \ w\in W\}$. Since we are dealing with coefficients in $\mathbb{Z}_{\geq0}[v+v^{-1}]$,
\[ \underline{\hat{H}}_{w_{k}}\underline{H}_{s_{i_{k}}}+X\underline{H}_{s_{i_{k}}}=0 \iff  \underline{\hat{H}}_{w_{k}}\underline{H}_{s_{i_{k}}}=0 \hspace{1mm} \text{ and } \hspace{1mm} X\underline{H}_{s_{i_{k}}}=0.  \]
However, $\underline{\hat{H}}_{w_{k}}\underline{H}_{s_{i_{k}}}\neq 0$ since $w_{k}s_{i_{k}}<w_{k}$, thus $\underline{\hat{H}}_{w_{k}}\underline{H}_{s_{i_{k}}}+X\underline{H}_{s_{i_{k}}}\neq 0$ which proves (i).

Item (ii): Compatibility implies weakly compatibility, thus by (i) we have the non-zero condition. As such, we only need to show that we have an isomorphism $\theta_{s_{i_{k}}}\hspace{-2mm}\cdots\hspace{1mm}\theta_{s_{i_{1}}}L(w_{1})\cong\theta_{s_{i_{k}}}L(w_{k})$. We prove this by induction on $k$, with the base case being immediate. By induction we have
\[ \theta_{s_{i_{k}}}\left(\theta_{s_{i_{k-1}}}\hspace{-2mm}\cdots\hspace{1mm}\theta_{s_{i_{1}}}L(w_{1})\right)\cong\theta_{s_{i_{k}}}\theta_{s_{i_{k-1}}}L(w_{k-1}), \]
noting $(w_{1},\dots,w_{k-1})$ is compatible with $(s_{i_{k-1}},\dots,s_{i_{1}})$. As $w_{k-1}s_{i_{k-1}}<w_{k-1}$, by 
\Cref{Prop:JantsenMiddleStructure} the module $\theta_{s_{i_{k-1}}}L(w_{k-1})$ is indecomposable, of graded length three, has simple top and socle isomorphic to $L(w_{k-1})$, and has semi-simple Jantzen middle $\mathsf{J}_{s_{i_{k-1}}}(w_{k-1})$. Since $w_{k-1}s_{i_{k}}>w_{k-1}$, then $\theta_{s_{i_{k}}}L(w_{k-1})=0$. This implies that we have an isomorphism
\[ \theta_{s_{i_{k}}}\theta_{s_{i_{k-1}}}L(w_{k-1})\cong\theta_{s_{i_{k}}}\mathsf{J}_{s_{i_{k-1}}}(w_{k-1}). \]
Furthermore, the module $L(w_{k})$ appears as a summand of $\mathsf{J}_{s_{i_{k-1}}}(w_{k-1})$ with multiplicity $1$ since both $w_{k}s_{i_{k-1}}>w_{k}$ and $\mu(w_{k-1},w_{k})=1$. Therefore, we have the isomorphism
\[ \theta_{s_{i_{k}}}\theta_{s_{i_{k-1}}}L(w_{k-1})\cong\theta_{s_{i_{k}}}L(w_{k})\oplus\theta_{s_{i_{k}}}M, \]
where $M$ is such that $\mathsf{J}_{s_{i_{k-1}}}(w_{k-1})=L(w_{k})\oplus M$. As $|\mathsf{Sup}(s_{i_{k-1}}s_{i_{k}})|=2$, \Cref{Cor:KMxSmallSup} implies that $\theta_{s_{i_{k-1}}s_{i_{k}}}L(w_{k-1})=\theta_{s_{i_{k}}}\theta_{s_{i_{k-1}}}L(w_{k-1})$ is either zero or indecomposable. But $\theta_{s_{i_{k}}}L(w_{k})\neq0$ since $w_{k}s_{i_{k}}<w_{k}$, thus $\theta_{s_{i_{k-1}}s_{i_{k}}}L(w_{k-1})$ is indecomposable. Therefore $\theta_{s_{i_{k}}}M=0$ and hence we have that $\theta_{s_{i_{k}}}\theta_{s_{i_{k-1}}}L(w_{k-1})\cong\theta_{s_{i_{k}}}L(w_{k})$, which completes the proof of (ii).
\end{proof}

\begin{thm}\label{Thm:PatternsKostantNegative}
Let $w\in\mathfrak{S}_{n}$. Then $w$ is Kostant negative whenever there exists an element in the same left cell as $w$ which consecutively contains any of the following patterns:
\[ \mathtt{2143}, \hspace{2mm} \mathtt{3142}, \hspace{2mm} \mathtt{14325}, \hspace{2mm} \mathtt{15324}, \hspace{2mm} \mathtt{25314}, \hspace{2mm} \mathtt{24315}. \]
\end{thm}

\begin{proof}
By \Cref{Thm:KLeftCellInv}, $\mathsf{K}(w)$ is left cell invariant, thus we may assume that $w$ itself consecutively contains any of the above patterns. For the pattern $\mathtt{2143}$, this was proven in \cite[Proposition~5]{CM24}. We prove this theorem for patterns $\mathtt{3142}$ and $\mathtt{14325}$, and explain how the remaining patterns follow in a completely analogous manner to that of $\mathtt{14325}$.

First assume $w=\mathtt{i}_{1}\cdots\mathtt{i}_{n}\in S_{n}$ consecutively contains $\mathtt{3142}$. So there exists $m\in[n-3]$ such that $\mathtt{i}_{m}\mathtt{i}_{m+1}\mathtt{i}_{m+2}\mathtt{i}_{m+3}$ has the same relative order as the sequence $\mathtt{3142}$. In other words,
\begin{equation}\label{Eq:2143PatternInequailities}
\mathtt{i}_{m+2}>\mathtt{i}_{m}>\mathtt{i}_{m+3}>\mathtt{i}_{m+1}.
\end{equation}
Consider the Bruhat walk $(w,ws_{m+1},w)$ and reduced word $\underline{x}=(s_{m},s_{m+1},s_{m+2})$. By the chain of inequalities \eqref{Eq:2143PatternInequailities} (see also (ii) of \Cref{Ex:SpecialRDCompatibleSequence} as the computations are analogous), one can deduce that $(w,ws_{m+1},w)$ is compatible with  $\underline{x}$. Then, by (ii) of \Cref{Lem:SpecialRDCompatibleSequence},
\[ \theta_{s_{m+2}s_{m+1}s_{m}}L(w)\cong\theta_{s_{m}}\theta_{s_{m+1}}\theta_{s_{m+2}}L(w)\cong\theta_{s_{m}}L(w)\neq 0. \]
Therefore $\mathsf{Kh}(w)=\mathtt{false}$, and hence by \Cref{Eq:KiffKhKM} the element $w$ is Kostant negative.

Now assume $w=\mathtt{i}_{1}\cdots\mathtt{i}_{n}\in S_{n}$ consecutively contains $\mathtt{14325}$. So there exists $m\in[n-4]$ where $\mathtt{i}_{m}\mathtt{i}_{m+1}\mathtt{i}_{m+2}\mathtt{i}_{m+3}\mathtt{i}_{m+4}$ has the same relative order as $\mathtt{14325}$. In other words,
\begin{equation}\label{Eq:14325PatternInequailities}
\mathtt{i}_{m+4}>\mathtt{i}_{m+1}>\mathtt{i}_{m+2}>\mathtt{i}_{m+3}>\mathtt{i}_{m}.
\end{equation}
Let $x:=s_{m+1}s_{m+2}s_{m}s_{m+1}$ and $y:=s_{m+2}s_{m+3}x$. We seek to prove that $\theta_{x}L(w)\cong\theta_{y}L(w)\neq 0$, meaning $\mathsf{Kh}(w)=\mathtt{false}$ and thus $\mathsf{K}(w)=\mathtt{false}$ by \Cref{Eq:KiffKhKM}. We first prove $\theta_{x}L(w)\neq 0$. Consider the Bruhat walk $(w,ws_{m},w,ws_{m+2})$ and reduced word $\underline{x}=(s_{m+1},s_{m+2},s_{m},s_{m+1})$ of $x$. From \Cref{Eq:14325PatternInequailities} (see also (i) from 
\Cref{Ex:SpecialRDCompatibleSequence} as the computations are analogous), one can deduce that $(w,ws_{m},w,ws_{m+2})$ is weakly compatible with $\underline{x}$. Thus, by (i) of  
\Cref{Lem:SpecialRDCompatibleSequence},
\[ \theta_{x}L(w)\cong\theta_{s_{m+1}}\theta_{s_{m+2}}\theta_{s_{m}}\theta_{s_{m+1}}L(w)\neq 0, \] 
where $\theta_{x}\cong\theta_{s_{m+1}}\theta_{s_{m+2}}\theta_{s_{m}}\theta_{s_{m+1}}$ can be checked directly. So we have the non-zero condition. It remains to prove $\theta_{x}L(w)\cong\theta_{y}L(w)$. To help show this, we first prove that $\theta_{x}L(ws_{m+3})=0$. Consider the reduced word $(s_{m},s_{m+1})$ and the Bruhat walk $(ws_{m+3},ws_{m+3}s_{m})$. By \Cref{Eq:14325PatternInequailities}, one can check that  $(ws_{m+3},ws_{m+3}s_{m})$ is compatible with $(s_{m},s_{m+1})$, and hence by (ii) of \Cref{Lem:SpecialRDCompatibleSequence}, the module $\theta_{x}L(ws_{m+3})\cong\theta_{s_{m+1}}\theta_{s_{m+2}}(\theta_{s_{m}}\theta_{s_{m+1}}L(ws_{m+3}))$ is isomorphic to
\[ \theta_{s_{m+1}}\theta_{s_{m+2}}\theta_{s_{m}}L(ws_{m+3}s_{m})=\theta_{s_{m+1}}\theta_{s_{m}}\theta_{s_{m+2}}L(ws_{m+3}s_{m})=0, \]
as $\theta_{s_{m+2}}$ commutes with $\theta_{s_{m}}$ and $\theta_{s_{m+2}}L(ws_{m+3}s_{m})=0$ as $(ws_{m+3}s_{m})s_{m+2}>ws_{m+3}s_{m}$. Also, one can directly check that $\theta_{y}\cong\theta_{x}\theta_{s_{m+2}s_{m+3}}\cong\theta_{x}\theta_{s_{m+3}}\theta_{s_{m+2}}$. Now consider the reduced word $(s_{m+3},s_{m+2})$ and Bruhat walk $(w,ws_{m+3})$. By \Cref{Eq:14325PatternInequailities}, one can deduce that $(w,ws_{m+3})$ is compatible with $(s_{m+3},s_{m+2})$, and so by (ii) of 
\Cref{Lem:SpecialRDCompatibleSequence} we have
\[ \theta_{y}L(w)=\theta_{x}(\theta_{s_{m+3}}\theta_{s_{m+2}}L(w))\cong\theta_{x}\theta_{s_{m+3}}L(ws_{m+3}). \]
Since $(ws_{m+3})s_{m+3}<ws_{m+3}$, by \Cref{Prop:JantsenMiddleStructure} the module $\theta_{s_{m+3}}L(ws_{m+3})$ is indecomposable, of graded length three, has simple top and socle isomorphic to $L(ws_{m+3})$, and semi-simple Jantzen middle $\mathsf{J}_{s_{m+3}}(ws_{m+3})$. We know $\theta_{x}L(ws_{m+3})=0$, and we know that $L(w)$ appears as a summand of $\mathsf{J}_{s_{m+3}}(ws_{m+3})$ since $ws_{m+3}>w$ and $\mu(ws_{m+3},w)=1$. Therefore, we have
\[ \theta_{y}L(w)\cong\theta_{x}\theta_{s_{m+3}}L(ws_{m+3})\cong\theta_{x}\mathsf{J}_{s_{m+3}}(ws_{m+3})=\theta_{x}L(w)\oplus\theta_{x}M, \]
where $M$ is such that $\mathsf{J}_{s_{m+3}}(ws_{m+3})=L(w)\oplus M$. Lastly, since $|\mathsf{Sup}(y)|=4$, by \Cref{Cor:KMxSmallSup}, the module $\theta_{y}L(w)$ is zero or indecomposable. We know $\theta_{x}L(w)\neq 0$, and therefore
\[ \theta_{y}L(w)\cong\theta_{x}L(w)\oplus \theta_{x}M\neq 0. \]
Hence $\theta_{y}L(w)$ is indecomposable, which implies $\theta_{x}M=0$ and so $\theta_{y}L(w)\cong\theta_{x}L(w)$ as desired.

Lastly, suppose $w$ consecutively contains one of the patterns $\mathtt{15324}$, $\mathtt{25314}$, or $\mathtt{24315}$, at positions $m$ to $m+4$. Then, for the same $x:=s_{m+1}s_{m+2}s_{m}s_{m+1}$ and $y:=s_{m+2}s_{m+3}x$, one can prove
\[ \theta_{y}L(w)\cong\theta_{x}L(w)\neq 0\] 
in exactly the same manner as above. That is to say, all the Bruhat relations which allowed one to form the arguments above similarly hold for this $w$ too. Thus $w$ is also Kostant negative.
\end{proof}

We note that, for $1\leq n\leq4$, all Kostant negative $w\in \mathfrak{S}_{n}$ are accounted for by \Cref{Thm:PatternsKostantNegative}. 

\begin{cor}\label{Cor:PatternsKahrstromsConjecture}
Any involution $d\in\mathfrak{I}_{n}$ satisfies K{\aa}hrstr\"om's Conjecture (\Cref{Conj:KahrstromsConjecture}) whenever it consecutively contains any of the following as patterns:
\[ \mathtt{2143}, \hspace{2mm} \mathtt{3142}, \hspace{2mm} \mathtt{14325}, \hspace{2mm} \mathtt{15324}, \hspace{2mm} \mathtt{25314}, \hspace{2mm} \mathtt{24315}. \]
\end{cor}

\begin{proof}
Given such a $d$, it was shown in the proof of \Cref{Thm:PatternsKostantNegative} that $\mathsf{K}(d)=\mathsf{Kh}(d)=\mathtt{false}$,  and the equality $\mathsf{Kh}(d)=\mathtt{false}$ naturally implies $[\mathsf{Kh}](d)=\mathtt{false}$ and $[\mathsf{Kh}^{\mathsf{ev}}](d)=\mathtt{false}$.
\end{proof}

\section{Kostant's Problem for $A_{6}$}
\label{Sec:KostantsProblemA6}

In this section we answer Kostant's problem for $\mathfrak{sl}_{7}$. By \Cref{Thm:KLeftCellInv}, and the fact that each left cell contains a unique involution, it suffices to answer Kostant's problem for $\mathfrak{I}_{7}$. We have $|\mathfrak{I}_{7}|=232$ involutions in $\mathfrak{S}_{7}$, and thus cases to solve. We begin below by first recalling Kostant's problem for the smaller cases. We then employ the results of \cite{K10} to lift these smaller cases to $\mathfrak{sl}_{7}$. Doing this accounts for $161$ cases. Next, we recall the results of \cite{MMM24} which answers Kostant's problem for fully commutative elements. This accounts for an additional $29$ cases. We then use \Cref{Thm:PatternsKostantNegative} to obtain a further $25$ cases. Lastly, we have $17$ remaining cases ($11$ up to symmetry), which are treated with a case-by-case analysis.     

From here on, we use the notational short-hand $\mathsf{i}:=s_{i}\in S_{n}$. Hopefully this should causes no confusion with one-line notation as we are using different fonts (we will also stress when one-line notation is being used). Also, for any $1\leq j<i\leq n$, we will write $\mathsf{i}_{j}:=\mathsf{i(i-1)}\cdots\mathsf{j}\in\mathfrak{S}_{n}$, as this helps with conserving space. So for example, in $\mathfrak{S}_{7}$ we have
\[ s_{2}s_{3}s_{2}s_{4}s_{3}s_{2}s_{1}s_{5}s_{4}=\mathsf{232432154}=\mathsf{23_{2}4_{1}5_{4}}. \]
Lastly, when expressing an element of $\mathfrak{S}_{n}$ as a product of simple reflections in $S_{n}$, we will use the unique reduced expression which is minimal in the natural lexicographic ordering. 

\subsection{Kostant's Problem for Smaller Cases}
\label{SubSec:Kostant'sProblemforSmallerCases}

From \cite[Section 4]{KM10} we have the following:

\emph{Cases} $n=1,2,3$: $\mathsf{K}(w)=\mathtt{true}$ for all $w\in\mathfrak{S}_{n}$.

\emph{Case} $n=4$: $\mathsf{K}(w)=\mathtt{false}$ for $w\in\mathfrak{S}_{4}$ if and only if it is in the left cell of the involution $\mathsf{13}$.

\emph{Case} $n=5$: $\mathsf{K}(w)=\mathtt{false}$ for $w\in\mathfrak{S}_{5}$ if and only if it is in the left cell of one of the involutions
\[ \mathsf{13}, \hspace{3mm} \mathsf{24}, \hspace{3mm} \mathsf{23_{2}}, \hspace{3mm} \mathsf{12_{1}4}, \text{ or } \mathsf{134_{3}}. \]
From \cite[Section 10.1]{KMM23} we have the following:

\emph{Case} $n=6$: $\mathsf{K}(w)=\mathtt{false}$ for $w\in\mathfrak{S}_{6}$ if and only if it is in the left cell of any of the following involutions:
\vspace{-2mm}
\begin{table}[h]
\begin{tabular}{ccccc}
 $\mathsf{13}$ & $\mathsf{135}$ & $\mathsf{14_{3}5_{4}}$ & $\mathsf{123_{1}5}$ & $\mathsf{23_{2}45_{2}}$ \\
 $\mathsf{24}$ & $\mathsf{134_{3}}$ & $\mathsf{2_{1}3_{2}5}$ & $\mathsf{134_{3}5_{3}}$ & $\mathsf{13_{1}45_{3}}$  \\
 $\mathsf{35}$ & $\mathsf{12_{1}4}$ & $\mathsf{23_{2}4_{2}}$ & $\mathsf{2_{1}4_{2}5_{4}}$ & $\mathsf{2_{1}3_{1}45_{2}}$  \\
 $\mathsf{23_{2}}$ & $\mathsf{245_{4}}$ & $\mathsf{1345_{3}}$ & $\mathsf{12_{1}3_{1}5}$ & $\mathsf{13_{2}4_{1}5_{3}}$  \\
 $\mathsf{34_{3}}$ & $\mathsf{23_{2}5}$ & $\mathsf{12_{1}45_{4}}$ & $\mathsf{123_{2}4_{1}}$ & $\mathsf{12_{1}34_{3}5_{1}}$
\end{tabular}
\end{table}
\vspace{-7mm}

\subsection{Parabolic Lifts of Smaller Cases}

Recall the set-up from \Cref{SubSec:CategoriesOl}, letting $I\subset S_{n}$ and $\mathfrak{l}\subset\mathfrak{sl}_{n}$ the corresponding semi-simple Levi factor of the induced parabolic.

\begin{lem}\label{Lem:KForAProducts}
For any $w\in\mathfrak{S}_{n}(I)$, we have that
\[ \mathsf{K}_{\mathfrak{l}}(w)=\mathsf{K}(w_{1})\wedge\cdots\wedge\mathsf{K}(w_{k}). \]
\end{lem}

\begin{proof}
Clearly $\theta_{x}L_{\mathfrak{l}}(w)$ is indecomposable if and only if $\mathcal{F}(\theta_{x}L_{\mathfrak{l}}(w))=\theta_{x_{1}}L(w_{1})\boxtimes\cdots\boxtimes\theta_{x_{k}}L(w_{k})$ is indecomposable if and only if $\theta_{x_{i}}L(w_{i})$ is indecomposable for each $1\leq i\leq k$. Thus, it suffices to show that $\mathsf{Kh}_{\mathfrak{l}}(w)=\mathsf{Kh}(w_{1})\wedge\cdots\wedge\mathsf{Kh}(w_{k})$. Firstly, suppose $\mathsf{Kh}(w_{i})=\mathtt{false}$ for some $i\in[k]$. Then there exists distinct elements $x,y\in\mathfrak{S}_{n_{i}}$ where $\theta_{x}L(w_{i}), \theta_{y}L(w_{i})\neq 0$ and $\theta_{x}L(w_{i})\cong\theta_{y}L(w_{i})\neq0$. Therefore, we must have that
\begin{align*}
\mathcal{F}(\theta_{\phi^{-1}(x)}L_{\mathfrak{l}}(w))&=L(w_{1})\boxtimes\cdots\boxtimes L(w_{i-1})\boxtimes\theta_{x}L(w_{i})\boxtimes L(w_{i+1})\boxtimes\cdots\boxtimes L(w_{k}) \\
&\cong L(w_{1})\boxtimes\cdots\boxtimes L(w_{i-1})\boxtimes\theta_{y}L(w_{i})\boxtimes L(w_{i+1})\boxtimes\cdots\boxtimes L(w_{k}) \\
&=\mathcal{F}(\theta_{\phi^{-1}(y)}L_{\mathfrak{l}}(w)).
\end{align*}
So $\theta_{\phi^{-1}(x)}L_{\mathfrak{l}}(w)\cong\theta_{\phi^{-1}(y)}L_{\mathfrak{l}}(w)\neq0$ which gives $\mathsf{Kh}_{\mathfrak{l}}(w)=\mathtt{false}$. Now suppose $\mathsf{Kh}_{\mathfrak{l}}(w)=\mathtt{false}$, so there exists distinct elements $x,y\in\mathfrak{S}_{n}(I)$ where $\theta_{x}L_{\mathfrak{l}}(w), \theta_{y}L_{\mathfrak{l}}(w)\neq 0$  and $\theta_{x}L_{\mathfrak{l}}(w)\cong\theta_{y}L_{\mathfrak{l}}(w)$. Hence $\mathcal{F}(\theta_{x}L_{\mathfrak{l}}(w))\cong\mathcal{F}(\theta_{y}L_{\mathfrak{l}}(w))\neq 0$ which implies $\theta_{x_{i}}L(w_{i})\cong\theta_{y_{i}}L(w_{i})\neq 0$ for each $1\leq i\leq k$. Since $x$ and $y$ are distinct, some pair $x_{i}$ and $y_{i}$ are distinct, and thus $\mathsf{Kh}(w_{i})=\mathtt{false}$.
\end{proof}

With the above lemma and the results of \Cref{SubSec:Kostant'sProblemforSmallerCases}, we have an answer to Kostant's problem for each $\mathfrak{l}\subset\mathfrak{sl}_{7}$. Recalling \Cref{Thm:ParaLiftK}, we know $\mathsf{K}_{\mathfrak{l}}(w)=\mathsf{K}(ww_{0}^{I}w_{0})$ for any $w\in\mathfrak{S}_{n}(I)$. This allows us to lift the answer of Kostant's problem for all the smaller cases of $\mathfrak{l}\subset\mathfrak{sl}_{7}$ to answers of $\mathfrak{sl}_{7}$ itself. We list all the cases obtained in this manner in \Cref{Table:1} below. Explicitly, we have listed all $d\in\mathfrak{I}_{7}$ such that there exists $I\subset S_{7}$ and $w\in\mathfrak{S}_{7}(I)$ where $d\sim_{L}ww_{0}^{I}w_{0}$, alongside the truth value $\mathsf{K}(d)$, which by \Cref{Thm:ParaLiftK} and \Cref{Thm:KLeftCellInv}, is equal to $\mathsf{K}_{\mathfrak{l}}(w)$. In general there are many such $I\subset S_{7}$ and $w\in\mathfrak{S}_{7}(I)$ where $d\sim_{L}ww_{0}^{I}w_{0}$, but we only give a single example for each $d$.

{\footnotesize
\begin{longtblr}[
  caption = {\normalsize{All $d\in\mathfrak{I}_{7}$ where $d\sim_{L}ww_{0}^{I}w_{0}$ for some $I\subset S_{7}$ and $w\in\mathfrak{S}_{7}(I)$.}},
  label = {Table:1},
]{
  colspec = {|cccc|cccc|},
  rowhead = 1,
  row{even} = {gray9},
  row{1} = {teal9},
}
\hline
$d$ & $I$ & $w$ & $\mathsf{K}_{\mathfrak{l}}(w)$ & $d$ & $I$ & $w$ & $\mathsf{K}_{\mathfrak{l}}(w)$ \\ \hline

$\mathsf{1}$ & $\{1,2,3,4,5\}$ & $\mathsf{e}$ & $\mathtt{true}$ 
& $\mathsf{12_{1}}$ & $\{1,2,3,4\}$ & $\mathsf{e}$ & $\mathtt{true}$ \\ 

$\mathsf{12_{1}3_{1}}$ & $\{1,2,3\}$ & $\mathsf{e}$ & $\mathtt{true}$
& $\mathsf{12_{1}3_{1}4_{1}}$ & $\{1,2\}$ & $\mathsf{e}$ & $\mathtt{true}$ \\

$\mathsf{12_{1}3_{1}4_{1}5_{1}}$ & $\{1\}$ & $\mathsf{e}$ & $\mathtt{true}$
& $\mathsf{12_{1}3_{1}4_{1}5_{1}6_{1}}$ & $\{1\}$ & $\mathsf{1}$ & $\mathtt{true}$ \\

$\mathsf{12_{1}3_{1}4_{1}56_{1}}$ & $\{2\}$ & $\mathsf{e}$ & $\mathtt{true}$
& $\mathsf{12_{1}3_{1}4_{1}6}$ & $\{1,2,3\}$ & $\mathsf{13}$ & $\mathtt{false}$ \\

$\mathsf{12_{1}3_{1}4_{3}5_{1}6_{1}}$ & $\{3\}$ & $\mathsf{e}$ & $\mathtt{true}$
& $\mathsf{12_{1}3_{1}45_{1}}$ & $\{1,2,3\}$ & $\mathsf{2}$ & $\mathtt{true}$ \\

$\mathsf{12_{1}3_{1}45_{4}6_{1}}$ & $\{2,3,4\}$ & $\mathsf{24}$ & $\mathtt{false}$
& $\mathsf{12_{1}3_{1}456_{1}}$ & $\{2,3\}$ & $\mathsf{e}$ & $\mathtt{true}$ \\

$\mathsf{12_{1}3_{1}5}$ & $\{1,2,3,4\}$ & $\mathsf{13}$ & $\mathtt{false}$
& $\mathsf{12_{1}3_{1}5_{1}6_{5}}$ & $\{1,3\}$ & $\mathsf{e}$ & $\mathtt{true}$ \\

$\mathsf{12_{1}3_{1}56_{5}}$ & $\{1,2,3,4,5\}$ & $\mathsf{12_{1}45_{4}}$ & $\mathtt{false}$
& $\mathsf{12_{1}3_{1}6}$ & $\{1,2,3,4\}$ & $\mathsf{14}$ & $\mathtt{true}$ \\

$\mathsf{12_{1}3_{2}4_{1}5_{1}}$ & $\{1,2,3,4\}$ & $\mathsf{23_{2}}$ & $\mathtt{false}$
& $\mathsf{12_{1}3_{2}4_{1}56_{1}}$ & $\{2,3,4\}$ & $\mathsf{3}$ & $\mathtt{true}$ \\

$\mathsf{12_{1}3_{2}4_{2}5_{1}6_{1}}$ & $\{4\}$ & $\mathsf{e}$ & $\mathtt{true}$
& $\mathsf{12_{1}3_{2}4_{3}5_{1}6_{1}}$ & $\{3,4\}$ & $\mathsf{e}$ & $\mathtt{true}$ \\

$\mathsf{12_{1}34_{1}}$ & $\{1,2,3,4\}$ & $\mathsf{2}$ & $\mathtt{true}$
& $\mathsf{12_{1}34_{1}5_{3}6_{1}}$ & $\{2,4\}$ & $\mathsf{e}$ & $\mathtt{true}$ \\

$\mathsf{12_{1}34_{1}6}$ & $\{1,2,3,4\}$ & $\mathsf{24}$ & $\mathtt{false}$
& $\mathsf{12_{1}34_{3}5_{1}}$ & $\{1,2,3,4,5\}$ & $\mathsf{1343}$ & $\mathtt{false}$ \\

$\mathsf{12_{1}34_{3}5_{3}6_{1}}$ & $\{3,4,5\}$ & $\mathsf{35}$ & $\mathtt{false}$
& $\mathsf{12_{1}34_{3}56_{1}}$ & $\{2,3,4,5\}$ & $\mathsf{24}$ & $\mathtt{false}$ \\

$\mathsf{12_{1}345_{1}}$ & $\{1,2,3,4\}$ & $\mathsf{3}$ & $\mathtt{true}$
& $\mathsf{12_{1}345_{4}6_{1}}$ & $\{2,3,4,5\}$ & $\mathsf{25}$ & $\mathtt{true}$ \\

$\mathsf{12_{1}3456_{1}}$ & $\{2,3,4\}$ & $\mathsf{e}$ & $\mathtt{true}$
& $\mathsf{12_{1}35_{1}6_{5}}$ & $\{1,3,4\}$ & $\mathsf{e}$ & $\mathtt{true}$ \\

$\mathsf{12_{1}4}$ & $\{1,2,3,4,5\}$ & $\mathsf{13}$ & $\mathtt{false}$
& $\mathsf{12_{1}4_{1}5_{4}}$ & $\{1,2,4\}$ & $\mathsf{e}$ & $\mathtt{true}$ \\

$\mathsf{12_{1}4_{1}5_{1}6_{4}}$ & $\{1,4\}$ & $\mathsf{e}$ & $\mathtt{true}$
& $\mathsf{12_{1}4_{1}56_{4}}$ & $\{1,2,3,4,5\}$ & $\mathsf{12_{1}45_{3}}$ & $\mathtt{false}$ \\

$\mathsf{12_{1}4_{3}5_{1}6_{4}}$ & $\{1,3,4,5\}$ & $\mathsf{35}$ & $\mathtt{false}$
& $\mathsf{12_{1}45_{4}6_{4}}$ & $\{2,3,4,5,6\}$ & $\mathsf{23_{2}56_{5}}$ & $\mathtt{false}$ \\

$\mathsf{12_{1}46}$ & $\{1,2,3,4,5\}$ & $\mathsf{135}$ & $\mathtt{false}$
& $\mathsf{12_{1}5}$ & $\{1,2,3,4,5\}$ & $\mathsf{14}$ & $\mathtt{true}$ \\

$\mathsf{12_{1}5_{4}6_{5}}$ & $\{1,2,3,4,5\}$ & $\mathsf{135_{4}}$ & $\mathtt{false}$
& $\mathsf{12_{1}6}$ & $\{1,2,3,4,5\}$ & $\mathsf{15}$ & $\mathtt{true}$ \\

$\mathsf{123_{1}}$ & $\{1,2,3,4,5\}$ & $\mathsf{2}$ & $\mathtt{true}$
& $\mathsf{123_{1}4_{1}5_{2}6_{1}}$ & $\{2,5\}$ & $\mathsf{e}$ & $\mathtt{true}$ \\

$\mathsf{123_{1}4_{2}5_{1}}$ & $\{1,2,3,5\}$ & $\mathsf{2}$ & $\mathtt{true}$
& $\mathsf{123_{1}4_{2}56_{1}}$ & $\{2,3,5\}$ & $\mathsf{e}$ & $\mathtt{true}$ \\

$\mathsf{123_{1}4_{3}5_{2}6_{1}}$ & $\{3,5\}$ & $\mathsf{e}$ & $\mathtt{true}$
& $\mathsf{123_{1}5}$ & $\{1,2,3,4,5\}$ & $\mathsf{24}$ & $\mathtt{false}$ \\

$\mathsf{123_{1}6}$ & $\{1,2,3,4,5\}$ & $\mathsf{25}$ & $\mathtt{true}$
& $\mathsf{123_{2}4_{1}}$ & $\{1,2,3,4,5\}$ & $\mathsf{23_{2}}$ & $\mathtt{false}$ \\

$\mathsf{123_{2}4_{1}6}$ & $\{1,2,3,4,5\}$ & $\mathsf{23_{2}5}$ & $\mathtt{false}$
& $\mathsf{123_{2}4_{2}5_{1}}$ & $\{1,2,3,4,5\}$ & $\mathsf{23_{2}4_{2}}$ & $\mathtt{false}$ \\

$\mathsf{123_{2}4_{2}5_{2}6_{1}}$ & $\{5\}$ & $\mathsf{e}$ & $\mathtt{true}$
& $\mathsf{123_{2}4_{2}56_{1}}$ & $\{2,3,4,5\}$ & $\mathsf{34_{3}}$ & $\mathtt{false}$ \\

$\mathsf{123_{2}4_{3}5_{2}6_{1}}$ & $\{3,4,5\}$ & $\mathsf{4}$ & $\mathtt{true}$
& $\mathsf{123_{2}45_{1}}$ & $\{1,2,3,4,5\}$ & $\mathsf{34_{3}2}$ & $\mathtt{true}$ \\

$\mathsf{123_{2}45_{4}6_{1}}$ & $\{2,3,4,5\}$ & $\mathsf{35}$ & $\mathtt{false}$
& $\mathsf{123_{2}456_{1}}$ & $\{2,3,4,5\}$ & $\mathsf{3}$ & $\mathtt{true}$ \\

$\mathsf{123_{2}5_{1}6_{5}}$ & $\{1,3,4,5\}$ & $\mathsf{4}$ & $\mathtt{true}$
& $\mathsf{1234_{1}}$ & $\{1,2,3,4,5\}$ & $\mathsf{3}$ & $\mathtt{true}$ \\

$\mathsf{1234_{1}6}$ & $\{1,2,3,4,5\}$ & $\mathsf{35}$ & $\mathtt{false}$
& $\mathsf{1234_{2}5_{3}6_{1}}$ & $\{2,4,5\}$ & $\mathsf{e}$ & $\mathtt{true}$ \\

$\mathsf{1234_{3}5_{1}}$ & $\{1,2,3,4,5\}$ & $\mathsf{34_{3}}$ & $\mathtt{false}$
& $\mathsf{1234_{3}5_{3}6_{1}}$ & $\{4,5\}$ & $\mathsf{e}$ & $\mathtt{true}$ \\

$\mathsf{1234_{3}56_{1}}$ & $\{2,3,4,5\}$ & $\mathsf{4}$ & $\mathtt{true}$
& $\mathsf{12345_{1}}$ & $\{1,2,3,4,5\}$ & $\mathsf{4}$ & $\mathtt{true}$ \\

$\mathsf{12345_{4}6_{1}}$ & $\{3,4,5\}$ & $\mathsf{e}$ & $\mathtt{true}$
& $\mathsf{123456_{1}}$ & $\{2,3,4,5\}$ & $\mathsf{e}$ & $\mathtt{true}$ \\

$\mathsf{1235_{1}6_{5}}$ & $\{1,3,4,5\}$ & $\mathsf{e}$ & $\mathtt{true}$
& $\mathsf{124_{1}5_{4}}$ & $\{1,2,3,4,5\}$ & $\mathsf{24_{3}}$ & $\mathtt{true}$ \\

$\mathsf{124_{2}5_{1}6_{4}}$ & $\{1,4,5\}$ & $\mathsf{e}$ & $\mathtt{true}$
& $\mathsf{124_{3}5_{1}6_{4}}$ & $\{1,3,4,5,6\}$ & $\mathsf{36}$ & $\mathtt{true}$ \\

$\mathsf{13_{1}4_{3}}$ & $\{1,2,3,5\}$ & $\mathsf{e}$ & $\mathtt{true}$
& $\mathsf{13_{1}4_{1}5_{3}}$ & $\{1,2,5\}$ & $\mathsf{e}$ & $\mathtt{true}$ \\

$\mathsf{13_{1}4_{1}5_{1}6_{3}}$ & $\{1,5\}$ & $\mathsf{e}$ & $\mathtt{true}$
& $\mathsf{13_{1}4_{1}56_{3}}$ & $\{1,2,3,5\}$ & $\mathsf{13}$ & $\mathtt{false}$ \\

$\mathsf{13_{1}4_{3}5_{1}6_{3}}$ & $\{1,3,5\}$ & $\mathsf{e}$ & $\mathtt{true}$
& $\mathsf{13_{1}4_{3}6}$ & $\{1,2,3,4,5\}$ & $\mathsf{13_{2}5}$ & $\mathtt{false}$ \\

$\mathsf{13_{1}45_{4}6_{3}}$ & $\{2,3,4,5,6\}$ & $\mathsf{23_{2}56_{4}}$ & $\mathtt{false}$
& $\mathsf{13_{1}5_{3}6_{5}}$ & $\{1,2,3,4,5\}$ & $\mathsf{13_{2}5_{4}}$ & $\mathtt{false}$ \\

$\mathsf{13_{2}4_{1}5_{3}}$ & $\{1,2,4,5,6\}$ & $\mathsf{46}$ & $\mathtt{false}$
& $\mathsf{13_{2}4_{2}5_{1}6_{3}}$ & $\{1,4,5,6\}$ & $\mathsf{46}$ & $\mathtt{false}$ \\

$\mathsf{13_{2}4_{3}5_{1}6_{3}}$ & $\{1,3,4,5,6\}$ & $\mathsf{35}$ & $\mathtt{false}$
& $\mathsf{134_{1}5_{3}6_{3}}$ & $\{2,4,5,6\}$ & $\mathsf{46}$ & $\mathtt{false}$ \\

$\mathsf{134_{3}5_{3}6_{3}}$ & $\{4,5,6\}$ & $\mathsf{46}$ & $\mathtt{false}$
& $\mathsf{134_{3}56_{3}}$ & $\{2,3,4,5,6\}$ & $\mathsf{245_{4}}$ & $\mathtt{false}$ \\

$\mathsf{1345_{4}6_{3}}$ & $\{3,4,5,6\}$ & $\mathsf{35}$ & $\mathtt{false}$
& $\mathsf{13456_{3}}$ & $\{2,3,4,5,6\}$ & $\mathsf{24}$ & $\mathtt{false}$ \\

$\mathsf{1356_{5}}$ & $\{2,3,4,5,6\}$ & $\mathsf{246}$ & $\mathtt{false}$
& $\mathsf{14_{1}5_{3}6_{4}}$ & $\{1,2,4,5\}$ & $\mathsf{e}$ & $\mathtt{true}$ \\

$\mathsf{14_{3}56_{4}}$ & $\{2,3,4,5,6\}$ & $\mathsf{246_{5}}$ & $\mathtt{false}$
& $\mathsf{145_{4}6_{4}}$ & $\{3,4,5,6\}$ & $\mathsf{36}$ & $\mathtt{true}$ \\

$\mathsf{1456_{4}}$ & $\{2,3,4,5,6\}$ & $\mathsf{25}$ & $\mathtt{true}$
& $\mathsf{156_{5}}$ & $\{2,3,4,5,6\}$ & $\mathsf{26}$ & $\mathtt{true}$ \\

$\mathsf{2_{1}3_{2}}$ & $\{1,2,3,4,6\}$ & $\mathsf{e}$ & $\mathtt{true}$
& $\mathsf{2_{1}3_{1}4_{2}}$ & $\{1,2,3,6\}$ & $\mathsf{e}$ & $\mathtt{true}$ \\

$\mathsf{2_{1}3_{1}4_{1}5_{2}}$ & $\{1,2,6\}$ & $\mathsf{e}$ & $\mathtt{true}$
& $\mathsf{2_{1}3_{1}4_{1}5_{1}6_{2}}$ & $\{1,6\}$ & $\mathsf{e}$ & $\mathtt{true}$ \\

$\mathsf{2_{1}3_{1}4_{1}56_{2}}$ & $\{1,2,3,6\}$ & $\mathsf{13}$ & $\mathtt{false}$
& $\mathsf{2_{1}3_{1}4_{3}5_{1}6_{2}}$ & $\{1,3,6\}$ & $\mathsf{e}$ & $\mathtt{true}$ \\

$\mathsf{2_{1}3_{1}45_{2}}$ & $\{1,2,3,4,6\}$ & $\mathsf{13}$ & $\mathtt{false}$
& $\mathsf{2_{1}3_{1}45_{4}6_{2}}$ & $\{2,3,4,6\}$ & $\mathsf{24}$ & $\mathtt{false}$ \\

$\mathsf{2_{1}3_{1}456_{2}}$ & $\{1,2,3,4,6\}$ & $\mathsf{14}$ & $\mathtt{true}$
& $\mathsf{2_{1}3_{2}4_{1}5_{2}}$ & $\{1,2,4,6\}$ & $\mathsf{e}$ & $\mathtt{true}$ \\

$\mathsf{2_{1}3_{2}4_{1}56_{2}}$ & $\{1,2,3,4,6\}$ & $\mathsf{24}$ & $\mathtt{false}$
& $\mathsf{2_{1}3_{2}4_{2}5_{1}6_{2}}$ & $\{1,4,6\}$ & $\mathsf{e}$ & $\mathtt{true}$ \\

$\mathsf{2_{1}3_{2}4_{3}5_{1}6_{2}}$ & $\{1,3,4,6\}$ & $\mathsf{e}$ & $\mathtt{true}$
& $\mathsf{2_{1}3_{2}56_{5}}$ & $\{2,3,4,5,6\}$ & $\mathsf{24_{3}6}$ & $\mathtt{false}$ \\

$\mathsf{2_{1}34_{2}}$ & $\{1,2,3,4,6\}$ & $\mathsf{2}$ & $\mathtt{true}$
& $\mathsf{2_{1}34_{1}5_{3}6_{2}}$ & $\{2,4,6\}$ & $\mathsf{e}$ & $\mathtt{true}$ \\

$\mathsf{2_{1}34_{3}5_{2}}$ & $\{1,2,3,4,6\}$ & $\mathsf{23_{2}}$ & $\mathtt{false}$
& $\mathsf{2_{1}34_{3}5_{3}6_{2}}$ & $\{4,6\}$ & $\mathsf{e}$ & $\mathtt{true}$ \\

$\mathsf{2_{1}34_{3}56_{2}}$ & $\{2,3,4,6\}$ & $\mathsf{3}$ & $\mathtt{true}$
& $\mathsf{2_{1}345_{2}}$ & $\{1,2,3,4,6\}$ & $\mathsf{3}$ & $\mathtt{true}$ \\

$\mathsf{2_{1}345_{4}6_{2}}$ & $\{3,4,6\}$ & $\mathsf{e}$ & $\mathtt{true}$
& $\mathsf{2_{1}3456_{2}}$ & $\{2,3,4,6\}$ & $\mathsf{e}$ & $\mathtt{true}$ \\

$\mathsf{2_{1}4_{1}5_{2}6_{4}}$ & $\{1,2,4,5,6\}$ & $\mathsf{5}$ & $\mathtt{true}$
& $\mathsf{2_{1}4_{2}56_{4}}$ & $\{2,3,4,5,6\}$ & $\mathsf{24_{3}6_{5}}$ & $\mathtt{false}$ \\

$\mathsf{23_{1}4_{1}5_{2}6_{2}}$ & $\{2,6\}$ & $\mathsf{e}$ & $\mathtt{true}$
& $\mathsf{23_{1}4_{2}5_{2}}$ & $\{1,2,3,6\}$ & $\mathsf{2}$ & $\mathtt{true}$ \\

$\mathsf{23_{1}4_{2}56_{2}}$ & $\{2,3,6\}$ & $\mathsf{e}$ & $\mathtt{true}$
& $\mathsf{23_{1}4_{3}5_{2}6_{2}}$ & $\{3,6\}$ & $\mathsf{e}$ & $\mathtt{true}$ \\

$\mathsf{23_{2}4_{2}5_{2}6_{2}}$ & $\{6\}$ & $\mathsf{e}$ & $\mathtt{true}$
& $\mathsf{23_{2}4_{2}56_{2}}$ & $\{2,3,4,5,6\}$ & $\mathsf{34_{3}5_{3}}$ & $\mathtt{false}$ \\

$\mathsf{23_{2}4_{3}5_{2}6_{2}}$ & $\{3,4,5,6\}$ & $\mathsf{45_{4}}$ & $\mathtt{false}$
& $\mathsf{23_{2}45_{4}6_{2}}$ & $\{2,3,4,5,6\}$ & $\mathsf{34_{3}6}$ & $\mathtt{false}$ \\

$\mathsf{23_{2}456_{2}}$ & $\{2,3,4,5,6\}$ & $\mathsf{34_{3}}$ & $\mathtt{false}$
& $\mathsf{23_{2}5_{2}6_{5}}$ & $\{1,3,4,5,6\}$ & $\mathsf{45_{4}}$ & $\mathtt{false}$ \\

$\mathsf{234_{2}5_{3}6_{2}}$ & $\{2,4,5,6\}$ & $\mathsf{5}$ & $\mathtt{true}$
& $\mathsf{234_{3}5_{3}6_{2}}$ & $\{4,5,6\}$ & $\mathsf{5}$ & $\mathtt{true}$ \\

$\mathsf{234_{3}56_{2}}$ & $\{2,3,4,5,6\}$ & $\mathsf{45_{3}}$ & $\mathtt{true}$
& $\mathsf{2345_{4}6_{2}}$ & $\{3,4,5,6\}$ & $\mathsf{4}$ & $\mathtt{true}$ \\

$\mathsf{23456_{2}}$ & $\{2,3,4,5,6\}$ & $\mathsf{3}$ & $\mathtt{true}$
& $\mathsf{235_{2}6_{5}}$ & $\{1,3,4,5,6\}$ & $\mathsf{4}$ & $\mathtt{true}$ \\

$\mathsf{24_{2}5_{2}6_{4}}$ & $\{1,4,5,6\}$ & $\mathsf{5}$ & $\mathtt{true}$
& $\mathsf{24_{3}5_{2}6_{4}}$ & $\{1,3,4,5,6\}$ & $\mathsf{46}$ & $\mathtt{false}$ \\

$\mathsf{245_{4}6_{4}}$ & $\{3,4,5,6\}$ & $\mathsf{46}$ & $\mathtt{false}$
& $\mathsf{2456_{4}}$ & $\{2,3,4,5,6\}$ & $\mathsf{35}$ & $\mathtt{false}$ \\

$\mathsf{256_{5}}$ & $\{2,3,4,5,6\}$ & $\mathsf{36}$ & $\mathtt{true}$
& $\mathsf{3_{1}4_{1}5_{2}6_{3}}$ & $\{1,2,5,6\}$ & $\mathsf{e}$ & $\mathtt{true}$ \\

$\mathsf{3_{1}4_{2}5_{3}}$ & $\{1,2,3,5,6\}$ & $\mathsf{e}$ & $\mathtt{true}$
& $\mathsf{3_{1}4_{2}56_{3}}$ & $\{2,3,5,6\}$ & $\mathsf{e}$ & $\mathtt{true}$ \\

$\mathsf{3_{1}4_{3}5_{2}6_{3}}$ & $\{1,2,3,5,6\}$ & $\mathsf{2}$ & $\mathtt{true}$
& $\mathsf{3_{2}4_{2}5_{2}6_{3}}$ & $\{1,5,6\}$ & $\mathsf{e}$ & $\mathtt{true}$ \\

$\mathsf{3_{2}4_{2}56_{3}}$ & $\{1,2,3,5,6\}$ & $\mathsf{13}$ & $\mathtt{false}$
& $\mathsf{3_{2}4_{3}5_{2}6_{3}}$ & $\{1,3,5,6\}$ & $\mathsf{e}$ & $\mathtt{true}$ \\

$\mathsf{3_{2}45_{4}6_{3}}$ & $\{3,5,6\}$ & $\mathsf{e}$ & $\mathtt{true}$
& $\mathsf{3_{2}456_{3}}$ & $\{2,3,4,5,6\}$ & $\mathsf{354}$ & $\mathtt{true}$ \\

$\mathsf{34_{2}5_{3}6_{3}}$ & $\{2,5,6\}$ & $\mathsf{e}$ & $\mathtt{true}$
& $\mathsf{34_{3}5_{3}6_{3}}$ & $\{5,6\}$ & $\mathsf{e}$ & $\mathtt{true}$ \\

$\mathsf{34_{3}56_{3}}$ & $\{2,3,4,5,6\}$ & $\mathsf{45_{4}}$ & $\mathtt{false}$
& $\mathsf{345_{4}6_{3}}$ & $\{3,4,5,6\}$ & $\mathsf{5}$ & $\mathtt{true}$ \\

$\mathsf{3456_{3}}$ & $\{2,3,4,5,6\}$ & $\mathsf{4}$ & $\mathtt{true}$
& $\mathsf{35_{3}6_{5}}$ & $\{1,3,4,5,6\}$ & $\mathsf{5}$ & $\mathtt{true}$ \\

$\mathsf{356_{5}}$ & $\{2,3,4,5,6\}$ & $\mathsf{46}$ & $\mathtt{false}$
& $\mathsf{4_{2}5_{3}6_{4}}$ & $\{1,2,4,5,6\}$ & $\mathsf{e}$ & $\mathtt{true}$ \\

$\mathsf{4_{3}5_{3}6_{4}}$ & $\{1,4,5,6\}$ & $\mathsf{e}$ & $\mathtt{true}$
& $\mathsf{4_{3}56_{4}}$ & $\{2,4,5,6\}$ & $\mathsf{e}$ & $\mathtt{true}$ \\

$\mathsf{45_{4}6_{4}}$ & $\{4,5,6\}$ & $\mathsf{e}$ & $\mathtt{true}$
& $\mathsf{456_{4}}$ & $\{2,3,4,5,6\}$ & $\mathsf{5}$ & $\mathtt{true}$ \\

$\mathsf{5_{4}6_{5}}$ & $\{1,3,4,5,6\}$ & $\mathsf{e}$ & $\mathtt{true}$
& $\mathsf{56_{5}}$ & $\{3,4,5,6\}$ & $\mathsf{e}$ & $\mathtt{true}$ \\

$\mathsf{6}$ & $\{2,3,4,5,6\}$ & $\mathsf{e}$ & $\mathtt{true}$
& & & & \\
\hline
\end{longtblr}
}
\vspace{-4mm}
\Cref{Table:1} accounts for $161$ of the $|\mathfrak{I}_{7}|=232$ involutions. Hence we have $71$ cases remaining.

\subsection{Fully Commutative Cases} 

We now recall one of the main results from \cite{MMM24} and apply it to our special case of $\mathfrak{sl}_{7}$. Firstly, recall that a permutation $w\in\mathfrak{S}_{n}$ is called \emph{fully commutative} if and only if the Young diagram $\mathsf{sh}(w)$ has at most two rows. 

For $i\in[n-1]$ and $j\in\{1,2,\dots,\mathsf{min}(i-1,n-i-1)\}$, we let $\Sigma_{n}$ denote the set of elements
\[ \sigma_{i,0}:=s_{i}, \hspace{2mm} \text{ and } \hspace{2mm} \sigma_{i,j}:=\prod_{k=0}^{j}(i-k,i-k+j+1), \]
where we use cycle notation. The elements of $\Sigma_{n}$ are called \emph{special involutions}, and they are fully commutative. Two such elements $\sigma_{i,j}$ and $\sigma_{i',j'}$ are said to be \emph{distinct} provided that 
\[ (i+j+2)\leq (i'-j'-1) \hspace{2mm} \text{ or } \hspace{2mm} (i'+j'+2)\leq (i-j-1). \]
By \cite[Theorem~5]{MMM24}, for a fully commutative involution $d\in\mathfrak{I}_{n}$, then $\mathsf{K}(d)=\mathtt{true}$ if and only if it can be expressed as a product of pairwise distinct special involutions. In $\mathfrak{S}_{7}$, there are $35$ fully commutative involutions, $29$ of which do not belong to \Cref{Table:1}. We record below in \Cref{Table:2} these $29$ fully commutative involutions $d\in\mathfrak{I}_{7}$, their expressions as products of special involutions, and the values $\mathsf{K}(d)$ which depend on whether the special involutions in the products are pairwise distinct.   

{\footnotesize
\begin{longtblr}[
  caption = {The $29$ fully commutative involutions $d\in\mathfrak{I}_{7}$ expressed as products in $\Sigma_{7}$},
  label = {Table:2},
]{
  colspec = {|ccc|ccc|},
  rowhead = 1,
  row{even} = {gray9},
  row{1} = {teal9},
}
\hline
$d$ & As a product in $\Sigma_{7}$ & $\mathsf{K}(d)$ & $d$ & As a product in $\Sigma_{7}$ & $\mathsf{K}(d)$ \\ \hline

$\mathsf{e}$ & $\mathsf{e}$ & $\mathtt{true}$ 
& $\mathsf{13}$ & $\sigma_{1,0}\sigma_{3,0}$ & $\mathtt{false}$  \\ 

$\mathsf{135}$ & $\sigma_{1,0}\sigma_{3,0}\sigma_{5,0}$ & $\mathtt{false}$ 
& $\mathsf{136}$ & $\sigma_{1,0}\sigma_{3,0}\sigma_{6,0}$ & $\mathtt{false}$  \\ 

$\mathsf{14}$ & $\sigma_{1,0}\sigma_{4,0}$ & $\mathtt{true}$ 
& $\mathsf{14_{3}5_{4}}$ & $\sigma_{1,0}\sigma_{4,1}$ & $\mathtt{false}$  \\ 

$\mathsf{146}$ & $\sigma_{1,0}\sigma_{4,0}\sigma_{6,0}$ & $\mathtt{false}$ 
& $\mathsf{15}$ & $\sigma_{1,0}\sigma_{5,0}$ & $\mathtt{true}$  \\ 

$\mathsf{15_{4}6_{5}}$ & $\sigma_{1,0}\sigma_{5,1}$ & $\mathtt{true}$ 
& $\mathsf{16}$ & $\sigma_{1,0}\sigma_{6,0}$ & $\mathtt{true}$  \\ 

$\mathsf{2}$ & $\sigma_{2,0}$ & $\mathtt{true}$ 
& $\mathsf{2_{1}3_{2}5}$ & $\sigma_{2,1}\sigma_{5,0}$ & $\mathtt{false}$  \\ 

$\mathsf{2_{1}3_{2}6}$ & $\sigma_{2,1}\sigma_{6,0}$ & $\mathtt{true}$ 
& $\mathsf{2_{1}4_{2}5_{4}}$ & $\sigma_{3,0}\sigma_{3,2}\sigma_{3,0}$ & $\mathtt{false}$  \\ 

$\mathsf{24}$ & $\sigma_{2,0}\sigma_{4,0}$ & $\mathtt{false}$ 
& $\mathsf{246}$ & $\sigma_{2,0}\sigma_{4,0}\sigma_{6,0}$ & $\mathtt{false}$  \\ 

$\mathsf{25}$ & $\sigma_{2,0}\sigma_{5,0}$ & $\mathtt{true}$ 
& $\mathsf{25_{4}6_{5}}$ & $\sigma_{2,0}\sigma_{5,1}$ & $\mathtt{false}$  \\ 

$\mathsf{26}$ & $\sigma_{2,0}\sigma_{6,0}$ & $\mathtt{true}$ 
& $\mathsf{3}$ & $\sigma_{3,0}$ & $\mathtt{true}$  \\ 

$\mathsf{3_{2}4_{3}}$ & $\sigma_{3,1}$ & $\mathtt{true}$ 
& $\mathsf{3_{2}4_{3}6}$ & $\sigma_{3,1}\sigma_{6,0}$ & $\mathtt{false}$  \\ 

$\mathsf{3_{2}5_{3}6_{5}}$ & $\sigma_{4,0}\sigma_{4,2}\sigma_{4,0}$ & $\mathtt{false}$ 
& $\mathsf{35}$ & $\sigma_{3,0}\sigma_{5,0}$ & $\mathtt{false}$  \\ 

$\mathsf{36}$ & $\sigma_{3,0}\sigma_{6,0}$ & $\mathtt{true}$ 
& $\mathsf{4}$ & $\sigma_{4,0}$ & $\mathtt{true}$  \\ 

$\mathsf{4_{3}5_{4}}$ & $\sigma_{4,1}$ & $\mathtt{true}$ 
& $\mathsf{46}$ & $\sigma_{4,0}\sigma_{6,0}$ & $\mathtt{false}$  \\ 

$\mathsf{5}$ & $\sigma_{5,0}$ & $\mathtt{true}$ 
& & & \\ 
\hline
\end{longtblr}
}

\Cref{Table:1,Table:2} account for $190$ of the $|\mathfrak{I}_{7}|=232$ involutions. We have $42$ remaining cases.

\subsection{Consecutive Containment Cases}

We employ \Cref{Thm:PatternsKostantNegative} to get $25$ Kostant negative cases out of the remaining $42$ cases. These are displayed in \Cref{Table:3} below. Explicitly, we list any $d\in\mathfrak{I}_{7}$ (and their one-line description) out of the remaining $42$ cases which consecutively contain a pattern $p$ from \Cref{Thm:PatternsKostantNegative} (where we have underlined the location in which $p$ appears).

{\footnotesize
\begin{longtblr}[
  caption = {Remaining involutions $d$ which consecutively contain a pattern $p$ from \Cref{Thm:PatternsKostantNegative}},
  label = {Table:3},
]{
  colspec = {|ccc|ccc|},
  rowhead = 1,
  row{even} = {gray9},
  row{1} = {teal9},
}
\hline
$d$ & One-line description of $d$ & Pattern $p$ & $d$ & One-line description of $d$ & Pattern $p$ \\ \hline

$\mathsf{12_{1}45_{4}}$ & $\mathtt{32\underline{16547}}$ & $\mathtt{14325}$ 
& $\mathsf{12_{1}456_{4}}$ & $\mathtt{3\underline{2175}64}$ & $\mathtt{2143}$   \\ 

$\mathsf{123_{1}56_{5}}$ & $\mathtt{42\underline{3176}5}$ & $\mathtt{2143}$ 
& $\mathsf{13_{2}4_{1}56_{3}}$ & $\mathtt{\underline{5274}163}$ & $\mathtt{3142}$   \\ 

$\mathsf{134_{3}}$ & $\mathtt{\underline{2154}367}$ & $\mathtt{2143}$ 
& $\mathsf{134_{3}5_{3}}$ & $\mathtt{\underline{2165}437}$ & $\mathtt{2143}$   \\ 

$\mathsf{134_{3}6}$ & $\mathtt{\underline{2154}376}$ & $\mathtt{2143}$ 
& $\mathsf{1345_{3}}$ & $\mathtt{\underline{2164}537}$ & $\mathtt{2143}$   \\ 

$\mathsf{135_{3}6_{5}}$ & $\mathtt{\underline{2164}735}$ & $\mathtt{2143}$ 
& $\mathsf{145_{4}}$ & $\mathtt{21\underline{36547}}$ & $\mathtt{14325}$   \\ 

$\mathsf{2_{1}3_{1}5_{2}6_{5}}$ & $\mathtt{46\underline{3172}5}$ & $\mathtt{3142}$ 
& $\mathsf{2_{1}34_{2}6}$ & $\mathtt{351\underline{4276}}$ & $\mathtt{2143}$   \\ 

$\mathsf{2_{1}4_{3}5_{2}6_{4}}$ & $\mathtt{3\underline{6175}24}$ & $\mathtt{3142}$ 
& $\mathsf{23_{2}}$ & $\mathtt{\underline{14325}67}$ & $\mathtt{14325}$   \\ 

$\mathsf{23_{2}4_{2}6}$ & $\mathtt{154\underline{3276}}$ & $\mathtt{2143}$ 
& $\mathsf{23_{2}45_{2}}$ & $\mathtt{\underline{16435}27}$ & $\mathtt{15324}$   \\ 

$\mathsf{23_{2}5}$ & $\mathtt{\underline{14326}57}$ & $\mathtt{14325}$ 
& $\mathsf{23_{2}56_{5}}$ & $\mathtt{\underline{14327}65}$ & $\mathtt{14325}$   \\ 

$\mathsf{23_{2}6}$ & $\mathtt{\underline{14325}76}$ & $\mathtt{14325}$ 
& $\mathsf{234_{2}6}$ & $\mathtt{153\underline{4276}}$ & $\mathtt{2143}$   \\ 

$\mathsf{234_{3}5_{2}}$ & $\mathtt{16\underline{35427}}$ & $\mathtt{24315}$ 
& $\mathsf{245_{4}}$ & $\mathtt{1\underline{3265}47}$ & $\mathtt{2143}$   \\ 

$\mathsf{34_{3}}$ & $\mathtt{1\underline{25436}7}$ & $\mathtt{14325}$ 
& $\mathsf{34_{3}6}$ & $\mathtt{125\underline{4376}}$ & $\mathtt{2143}$   \\ 

$\mathsf{45_{4}}$ & $\mathtt{12\underline{36547}}$ & $\mathtt{14325}$ 
& & &  \\
\hline
\end{longtblr}
}

\Cref{Table:1,Table:2,Table:3} account for $215$ of the $|\mathfrak{I}_{7}|=232$ involutions. We have $17$ remaining cases.

\begin{rmk}
By \Cref{Thm:PatternsKostantNegative}, we know that $\mathsf{K}(d)=\mathtt{false}$ if a permutation in the same left cell as $d$ consecutively contains one of six patterns listed therein. However, within \Cref{Table:3} above, it was the involutions themselves which consecutively contained such a pattern. As such, each of the involutions appearing in \Cref{Table:3} satisfy K{\aa}hrstr\"om's Conjecture by \Cref{Cor:PatternsKahrstromsConjecture}.
\end{rmk}

\subsection{Remaining Cases} 
\label{SubSec:RemainingCases}

We have the following $17$ remaining involutions:
\begin{multicols}{3}
\begin{itemize}
\item[(1)]
\begin{itemize}
\item[(a)] $\mathsf{124_{1}56_{4}}$
\item[(b)] $\mathsf{13_{1}456_{3}}$
\end{itemize}
\item[(2)]
\begin{itemize}
\item[(a)] $\mathsf{13_{1}45_{3}}$
\item[(b)] $\mathsf{24_{2}56_{4}}$
\end{itemize}
\item[(3)]
\begin{itemize}
\item[(a)] $\mathsf{14_{3}5_{3}6_{4}}$
\item[(b)] $\mathsf{2_{1}3_{1}4_{2}6}$
\end{itemize}
\end{itemize}
\end{multicols}

\begin{multicols}{3}
\begin{itemize}
\item[(4)]
\begin{itemize}
\item[(a)] $\mathsf{24_{2}5_{4}}$
\item[(b)] $\mathsf{3_{2}45_{3}}$
\end{itemize}
\item[(5)]
\begin{itemize}
\item[(a)] $\mathsf{23_{2}4_{2}}$
\item[(b)] $\mathsf{34_{3}5_{3}}$
\end{itemize}
\item[(6)]
\begin{itemize}
\item[(a)] $\mathsf{234_{2}}$
\item[(b)] $\mathsf{345_{3}}$
\end{itemize}
\end{itemize}
\end{multicols}

\begin{multicols}{5}
\begin{itemize}
\item[(7)] $\mathsf{12_{1}56_{5}}$
\item[(8)] $\mathsf{2_{1}35_{2}6_{5}}$
\item[(9)] $\mathsf{23_{2}4_{2}5_{2}}$
\item[(10)] $\mathsf{2345_{2}}$
\item[(11)] $\mathsf{3_{2}4_{2}5_{3}}$
\end{itemize}
\end{multicols}
These cases have been paired up according to the natural symmetry of the root system, which pairs $w$ with $w_{0}ww_{0}$. Hence we only need to consider one involution for each of the above pairs. We go through each of these case-by-case (which are aided by GAP3 computations):

\textbf{Case} (1)(a): Let $d:=124_{1}56_{4}$. Then given $x:=12456_{2}$ and $y:=124_{2}56$, we seek to show that $\theta_{x}L(d)\cong\theta_{y}L(d)\neq0$. Note that $x=65y$. By computations, we have
\[ \theta_{y}\theta_{65}\cong\theta_{123_{2}6}\oplus\theta_{123_{2}56_{5}}\oplus\theta_{x}. \]
Since $D_{L}(\mathsf{123_{2}6})=\{1,3,6\}$ and $D_{L}(\mathsf{123_{2}56_{5}})=\{1,3,5,6\}$ are not subsets of $D_{L}(d)=\{1,4,6\}$, then by \Cref{Eq:ROrderDomDes} and \Cref{Eq:ThetaLNonZero}, $\theta_{123_{2}6}L(d)=0$ and $\theta_{123_{2}56_{5}}L(d)=0$. Therefore, $\theta_{x}L(d)\cong\theta_{y}\theta_{65}L(d)$. By computations, one can confirm that both
\[ [\theta_{65}L(d)]=[L(d)]+(v+v^{-1})[L(d5)] \]
and $[\theta_{y}L(d5)]=0$. This implies $\theta_{x}L(d)\cong\theta_{y}\theta_{65}L(d)\cong\theta_{y}L(d)$. Lastly, a computation can confirm $[\theta_{y}L(d)]\neq0$ and so $\theta_{y}L(d)\neq0$. Thus we have shown $\mathsf{Kh}(d)=\mathtt{false}$ and so $\mathsf{K}(d)=\mathtt{false}$.

\textbf{Case} (2)(a): Let $d:=13_{1}45_{3}$. Given both $x:=123_{2}45_{3}$ and $y:=3245_{3}$, we seek to show that $\theta_{x}L(d)\cong\theta_{y}L(d)\neq0$. Note that $x=12y$. By computations, we have
\[ \theta_{y}\theta_{12}\cong\theta_{1245_{3}}\oplus\theta_{x}. \]
Note that $D_{L}(\mathsf{1245_{3}})=\{1,4,5\}$ is not a subset of $D_{L}(d)=\{1,3,5\}$, hence by \Cref{Eq:ROrderDomDes} and \Cref{Eq:ThetaLNonZero}, $\theta_{\mathsf{1245_{3}}}L(d)=0$. Therefore $\theta_{x}L(d)\cong\theta_{y}\theta_{12}L(d)$. By computations, we have
\[ [\theta_{12}L(d)]=[L(d)]+(v+v^{-1})[L(d3)] \]
and  $[\theta_{y}L(d3)]=0$. This implies $\theta_{x}L(d)\cong\theta_{y}\theta_{12}L(d)\cong\theta_{y}L(d)$. Lastly, a computation can confirm $[\theta_{y}L(d)]\neq0$ and so $\theta_{y}L(d)\neq0$. Thus we have shown $\mathsf{Kh}(d)=\mathtt{false}$ and so $\mathsf{K}(d)=\mathtt{false}$.

\textbf{Case} (3)(a): Let $d:=14_{3}5_{3}6_{4}$. Then given $x:=124_{3}5_{3}6$ and $y:=4_{3}5_{3}6$, we seek to show that $\theta_{x}L(d)\cong\theta_{y}L(d)\neq0$. Note that $x=12y$. By computations, we have $\theta_{x}\cong\theta_{y}\theta_{12}$ and 
\[ [\theta_{12}L(d)]=[L(d)]+[L(d23)]+(v+v^{-1})[L(d2)]. \]
Note, $D_{L}(y)=\{4,5\}$ is not a subset of $D_{L}((d23)^{-1})=D_{R}(d23)=\{3,4\}$, and so $[\theta_{y}L(d23)]=0$. Moreover, by computations we have $[\theta_{y}L(d2)]=0$. Hence we have
\[ \theta_{x}L(d)\cong\theta_{y}\theta_{12}L(d)\cong\theta_{y}L(d). \]
Lastly, a computation confirms that $[\theta_{y}L(d)]\neq 0$. Thus $\mathsf{Kh}(d)=\mathtt{false}$ and so $\mathsf{K}(d)=\mathtt{false}$.

\textbf{Case} (4)(a): Let $d:=24_{2}5_{4}$. Then given $x:=2_{1}4_{2}5_{4}6$ and $y:=4_{2}5_{4}6$, we seek to show that $\theta_{x}L(d)\cong\theta_{y}L(d)\neq0$. Note that $x=21y$. By computations, we have $\theta_{x}\cong\theta_{y}\theta_{21}$ and 
\[ [\theta_{21}L(d)]=[L(d)]+(v+v^{-1})[L(d1)]. \]
By computations, we have $[\theta_{y}L(d1)]=0$ and $[\theta_{y}L(d)]\neq0$. So $\theta_{x}L(d)\cong\theta_{y}\theta_{21}L(d)\cong\theta_{y}L(d)\neq0$, and thus $\mathsf{Kh}(d)=\mathtt{false}$ and $\mathsf{K}(d)=\mathtt{false}$.

\textbf{Case} (5)(a): Let $d:=23_{2}4_{2}$. Then given $x:=2_{1}4_{1}5_{2}$ and $y:=2_{1}3_{1}4_{2}$, we seek to show that $\theta_{x}L(d)\cong\theta_{y}L(d)\neq0$. Note that $x=45y$. By computations we have $\theta_{x}\cong\theta_{y}\theta_{45}$ and
\[ [\theta_{45}L(d)]=[L(d)]+[L(d56)]+(v+v^{-1})[L(d5)]. \]
By computations, $[\theta_{y}L(d5)],[\theta_{y}L(d56)]=0$ and $[\theta_{y}L(d)]\neq0$. So $\theta_{x}L(d)\cong\theta_{y}\theta_{45}L(d)\cong\theta_{y}L(d)$, which is non-zero, and hence $\mathsf{Kh}(d)=\mathtt{false}$ and $\mathsf{K}(d)=\mathtt{false}$.

\textbf{Case} (7): Let $d:=\mathsf{12_{1}56_{5}}$. For $x:=\mathsf{2_{1}346_{1}}$ and $y:=\mathsf{12_{1}5}$, we seek to show $\theta_{x}L(d)\cong\theta_{y}L(d)\neq0$. Note that $x=\mathsf{236_{3}}y$. By computations we have $\theta_{x}\cong\theta_{y}\theta_{236_{3}}$ and
\[ [\theta_{236_{3}}L(d)]=[L(d)]+\sum_{u\in U}p_{u}[L(u)], \]
for some $0\neq p_{u}\in\mathbb{Z}_{\geq 0}[v,v^{-1}]$ and $U$ the set consisting of the following permutations:
\begin{align*}
&d\mathsf{3}, \hspace{2mm} d\mathsf{4}, \hspace{2mm} \mathsf{12_{1}3456_{5}}, \hspace{2mm} \mathsf{12_{1}35_{4}6_{5}}, \hspace{2mm} d\mathsf{34}, \hspace{2mm} d\mathsf{43}, \hspace{2mm} \mathsf{156_{1}}, \hspace{2mm} \mathsf{2_{1}56_{2}}, \hspace{2mm} \mathsf{12_{1}3456_{4}}, \hspace{2mm} \mathsf{12_{1}35_{4}6_{4}}, \hspace{2mm} \mathsf{12_{1}356_{3}}, \\
&d\mathsf{432}, \hspace{2mm} \mathsf{1256_{1}}, \hspace{2mm} \mathsf{12_{1}3456_{3}}, \hspace{2mm} \mathsf{12_{1}35_{4}6_{3}}, \hspace{2mm} \mathsf{12_{1}356_{2}}, \hspace{2mm} \mathsf{12356_{1}}, \hspace{2mm} \mathsf{12_{1}3456_{2}}, \hspace{2mm} \mathsf{12_{1}35_{4}6_{2}}, \hspace{2mm} \mathsf{123456_{1}}, \\
&\hspace{45mm} \mathsf{1235_{4}6_{1}}, \hspace{2mm} \mathsf{12345_{1}6_{2}}, \hspace{2mm} \mathsf{2_{1}3_{2}4_{3}5_{4}6_{1}}.
\end{align*}
For each $u\in U$, one can check that $D_{L}(y)=\{1,2,5\}$ is not a subset of $D_{L}(u^{-1})=D_{R}(u)$. From \Cref{Eq:ROrderDomDes} and \Cref{Eq:ThetaLNonZero}, this implies $\theta_{y}L(u)=0$. Hence $\theta_{x}L(d)\cong\theta_{y}\theta_{\mathsf{236_{3}}}L(d)\cong\theta_{y}L(d)$, and this module is non-zero since $y\leq_{R}y\mathsf{6}\leq_{R}d$. Hence $\mathsf{Kh}(d)=\mathtt{false}$ and $\mathsf{K}(d)=\mathtt{false}$.

\textbf{Case} (8): Let $d:=\mathsf{2_{1}35_{2}6_{5}}$. Then given $x:=\mathsf{2_{1}3_{2}45_{2}6}$ and $y:=\mathsf{25_{2}6}$, we seek to show that $\theta_{x}L(d)\cong\theta_{y}L(d)\neq0$. Note that $x=\mathsf{2_{1}34}y$. By computations we have 
\[ \theta_{y}\theta_{2_{1}34}\cong\theta_{x}\oplus\theta_{2_{1}3_{2}4_{2}6}\oplus\theta_{12_{1}34}\oplus\theta_{23_{2}6}, \]
and $[\theta_{2_{1}3_{2}4_{2}6}L(d)]=[\theta_{12_{1}3_{2}6}L(d)]=[\theta_{23_{2}6}L(d)]=0$. Therefore we have that $\theta_{x}L(d)\cong\theta_{y}\theta_{2_{1}36_{3}}L(d)$. By computations, we also have that
\[ [\theta_{2_{1}34}L(d)]=[L(d)]+\sum_{u\in U}p_{u}[L(u)], \]
for some $0\neq p_{u}\in\mathbb{Z}_{\geq 0}[v,v^{-1}]$ and $U$ the set consisting of the following permutations:
\begin{align*}
&\mathsf{235_{3}6_{5}}, \hspace{2mm} \mathsf{2_{1}3_{2}5_{3}6_{5}}, \hspace{2mm} \mathsf{2_{1}3_{2}4_{3}5_{3}6_{5}}, \hspace{2mm} \mathsf{2_{1}35_{2}6_{3}}, \hspace{2mm} \mathsf{2_{1}35_{3}6_{5}}, \\
&\mathsf{235_{3}6_{4}}, \hspace{2mm} \mathsf{2_{1}35_{3}6_{4}}, \hspace{2mm} \mathsf{2_{1}3_{2}5_{3}6_{4}}, \hspace{2mm} \mathsf{2_{1}35_{2}6_{4}}, \hspace{2mm} \mathsf{2_{1}3_{2}4_{3}5_{3}6_{4}}.
\end{align*}
One can check that all the permutations $u\in U\backslash\{2_{1}35_{2}6_{4}\}$ are such that $D_{L}(y)=\{2,5\}$ is not a subset of $D_{R}(u)$, and thus by \Cref{Eq:ROrderDomDes} and \Cref{Eq:ThetaLNonZero} we have that $\theta_{y}L(u)=0$. Moreover, a computation confirms that $\theta_{y}L(2_{1}35_{2}6_{4})=0$. Therefore, 
\[ \theta_{x}L(d)\cong\theta_{y}\theta_{2_{1}34}L(d)\cong\theta_{y}L(d). \]
Lastly, via a computation we have $[\theta_{y}L(d)]\neq0$, and hence $\mathsf{Kh}(d)=\mathtt{false}$ and $\mathsf{K}(d)=\mathtt{false}$.

\textbf{Case} (9): Let $d:=\mathsf{23_{2}4_{2}5_{2}}$. Then given $x:=\mathsf{2_{1}3_{1}5_{1}6_{2}}$ and $y:=\mathsf{2_{1}3_{1}4_{1}5_{2}}$, we seek to show that $\theta_{x}L(d)\cong\theta_{y}L(d)\neq0$. Note that $x=56y$. By computations we have $\theta_{x}\cong\theta_{y}\theta_{56}$ and
\[ [\theta_{56}L(d)]=[L(d)]+(v+v^{-1})[L(d6)]. \]
By computations, $[\theta_{y}L(d6)]=0$ and $[\theta_{y}L(d)]\neq0$. Hence $\theta_{x}L(d)\cong\theta_{y}\theta_{56}L(d)\cong\theta_{y}L(d)\neq0$, and thus $\mathsf{Kh}(d)=\mathtt{false}$ and $\mathsf{K}(d)=\mathtt{false}$.

We now show that the remaining three cases are Kostant positive. In fact, we show they also satisfy K{\aa}hrstr\"om's Conjecture, i.e. that $\mathsf{K}(d)=\mathsf{Kh}(d)=[\mathsf{Kh}](d)=[\mathsf{Kh}^{\mathsf{ev}}]=\mathtt{true}$.

\textbf{Case} (6)(a): Let $d=\mathsf{234_{2}}$, we have $\mathsf{sh}(d)=(5,1,1)$ and $D(d)=\{2,4\}$. We first seek to show that $\mathsf{KM}(\star,d)=\mathtt{true}$, that is, to show $\theta_{x}L(d)$ is either zero or indecomposable for all $x\in\mathfrak{S}_{7}$. From \Cref{Eq:ThetaLNonZero}, we only need to consider when $x\leq_{R}d$, which implies $\mathsf{sh}(d)\preceq\mathsf{sh}(x)$ and $D_{L}(x)\subset D(d)$ by \Cref{Eq:ROrderDomDes}. Also, by (a) of \Cref{Prop:IndecompConjLRCellInv}, we only need to consider $x$ an involution. From this, it suffices to only consider $x$ from the following list:
\[ \mathsf{234_{2}}, \mathsf{2_{1}3_{2}}, \mathsf{4_{3}5_{4}}, \mathsf{24}, \mathsf{2}, \mathsf{4}, e. \]
Thus $\mathsf{KM}(\star,d)=\mathtt{true}$ by \Cref{Cor:KMxSmallSup}. By \Cref{Eq:KiffKhKM}, we must have that $\mathsf{K}(d)=\mathsf{Kh}(d)$. Note, the equality $[\mathsf{Kh}^{\mathsf{ev}}](d)=\mathtt{true}$ naturally implies both $[\mathsf{Kh}](d)=\mathtt{true}$ and $\mathsf{Kh}(d)=\mathtt{true}$, and thus also $\mathsf{K}(d)=\mathtt{true}$. Therefore, it suffices to show $[\mathsf{Kh}^{\mathsf{ev}}](d)=\mathtt{true}$. To this end, it will be helpful to know all $x\in\mathfrak{S}_{7}$ such that $x\leq_{R}d$. Such a set is a union of right cells which contain an involution from the above list. One can check which involutions above belong to such a set by checking if $[\theta_{x}L(d)]$ is non-zero or not (Equation \eqref{Eq:ThetaLNonZero}). By GAP computations we get
\[ \mathfrak{I}_{7}(\leq_{R}d):=\{x\in\mathfrak{I}_{7} \ | \ x\leq_{R}d\}=\{\mathsf{234_{2}}, \mathsf{24}, \mathsf{2}, \mathsf{4}, e\}. \]
Thus, if $x\leq_{R}d$, then $x$ belongs to the right cell of some involution in $\mathfrak{I}_{7}(\leq_{R}d)$. We want to show $[\mathsf{Kh}^{\mathsf{ev}}](d)=\mathtt{true}$, i.e. show that for any distinct pair $x,y\leq_{R}d$ we have the inequality
\begin{equation}\label{Eq:Case(6)(a)[Kh]Condition}
\mathsf{ev}(\underline{\hat{H}}_{d}\underline{H}_{x})\neq\mathsf{ev}(\underline{\hat{H}}_{d}\underline{H}_{y}).
\end{equation}
By \Cref{Eq:ThetaxL=ThetayLImpliesxly}, if we instead had an equality, then $x\sim_{L}y$, in particular, $\mathsf{sh}(x)=\mathsf{sh}(y)$. Therefore, without loss of generality, we need to confirm that Inequality~(\ref{Eq:Case(6)(a)[Kh]Condition}) holds for all pairs $(x,y)$ such that  $x\sim_{R}\mathsf{2}$, $y\sim_{R}\mathsf{4}$, and $x\sim_{L}y$, i.e. all pairs from the following list: $(\mathsf{23456},\mathsf{456})$, $(\mathsf{2345},\mathsf{45})$, $(\mathsf{234},\mathsf{4})$, $(\mathsf{23},\mathsf{43})$, $(\mathsf{2},\mathsf{432})$, and $(\mathsf{21},\mathsf{4321})$. One can confirm such via GAP computations.

\textbf{Case} (10): Let $d=\mathsf{2345_{2}}$. By similar considerations to the previous case, one can deduce that
\[ \mathfrak{I}_{7}(\leq_{R}d):=\{x\in\mathfrak{I}_{7} \ | \ x\leq_{R}d\}=\{\mathsf{2345_{2}}, \mathsf{25}, \mathsf{2}, \mathsf{5}, e\}. \]
Thus $\mathsf{KM}(\star,d)=\mathtt{true}$ by \Cref{Cor:KMxSmallSup}. As previously, it suffices to show that $[\mathsf{Kh}^{\mathsf{ev}}](d)=\mathtt{true}$, which in turn comes down to checking that \eqref{Eq:Case(6)(a)[Kh]Condition} holds for all pairs $(x,y)$ such that $x\sim_{R}\mathsf{2}$, $y\sim_{R}\mathsf{5}$, and $x\sim_{L}y$. In particular, for all pairs in the following list: $(\mathsf{23456},\mathsf{56})$, $(\mathsf{2345},\mathsf{5})$, $(\mathsf{234},\mathsf{54})$, $(\mathsf{23},\mathsf{543})$, $(\mathsf{2},\mathsf{5432})$, and $(\mathsf{21},\mathsf{54321})$. Again, GAP computations confirm such inequalities.

\textbf{Case} (11): Let $d=\mathsf{3_{2}4_{2}5_{3}}$. By similar considerations to the previous two cases, we have
\[ \mathfrak{I}_{7}(\leq_{R}d):=\{x\in\mathfrak{I}_{7} \ | \ x\leq_{R}d\}=\{\mathsf{3_{2}4_{2}5_{3}}, \mathsf{3_{1}4_{2}5_{3}}, \mathsf{4_{2}5_{3}6_{4}}, \mathsf{34_{3}}, \mathsf{3_{2}4_{3}}, \mathsf{4_{3}5_{4}}, \mathsf{3}, \mathsf{4}, e\}. \]
Thus $\mathsf{KM}(\star,d)=\mathtt{true}$ by \Cref{Cor:KMxSmallSup}. Again, it suffices to show that $[\mathsf{Kh}^{\mathsf{ev}}](d)=\mathtt{true}$, which in turn comes down to checking that \eqref{Eq:Case(6)(a)[Kh]Condition} holds for all pairs $(x,y)$ such that $x,y\leq_{R}d$ and $x\sim_{L}y$. There are three different shapes for which such pairs can take, which give the three cases: (i) $x\sim_{R}\mathsf{3_{1}4_{2}5_{3}}$, $y\sim_{R}\mathsf{4_{2}5_{3}6_{4}}$, and $x\sim_{L}y$; (ii) $x\sim_{R}\mathsf{3_{2}4_{3}}$, $y\sim_{R}\mathsf{4_{3}5_{4}}$, and $x\sim_{L}y$; and (iii) $x\sim_{R}3$, $y\sim_{R}4$, and $x\sim_{L}y$. For case (i), we are dealing with the following list of pairs:
\begin{align*}
&(\mathsf{3_{1}4_{2}5_{3}6_{4}},\mathsf{4_{2}5_{3}6_{4}}), (\mathsf{3_{1}4_{2}5_{3}},\mathsf{4_{1}5_{2}6_{3}}), (\mathsf{3_{1}4_{2}5_{3}6},\mathsf{4_{2}5_{3}6}), (\mathsf{3_{1}4_{2}5_{3}6_{5}},\mathsf{4_{2}5_{3}6_{5}}), (\mathsf{3_{1}4_{2}56},\mathsf{4_{1}5_{2}6}), \\
&(\mathsf{3_{1}4_{2}5},\mathsf{4_{1}5_{2}6_{5}}), (\mathsf{3_{1}4_{2}5_{4}6_{5}},\mathsf{4_{2}5_{4}6_{5}}), (\mathsf{3_{1}4_{2}5_{4}},\mathsf{4_{1}5_{2}6_{4}}), (\mathsf{3_{1}4_{2}5_{4}6},\mathsf{4_{2}5_{4}6}), (\mathsf{3_{1}4_{3}5_{4}6_{5}},\mathsf{4_{1}5_{4}6_{5}}), \\
& \hspace{18mm} (\mathsf{3_{1}4_{3}5_{4}},\mathsf{4_{1}5_{3}6_{4}}), (\mathsf{3_{1}4_{3}5_{4}6},\mathsf{4_{1}5_{4}6}), (\mathsf{3_{1}4_{3}56},\mathsf{4_{1}5_{3}6}), (\mathsf{3_{1}4_{3}5},\mathsf{4_{1}5_{3}6_{5}}).
\end{align*}
For case (ii) we are dealing with the following list if pairs:
\begin{align*}
& \hspace{2mm} (\mathsf{3_{1}4_{2}},\mathsf{4_{1}5_{2}}), (\mathsf{3_{2}4_{3}5_{4}6_{5}},\mathsf{4_{3}5_{4}6_{5}}), (\mathsf{3_{2}4_{3}5_{4}},\mathsf{4_{3}5_{4}}), (\mathsf{3_{2}4_{3}5_{4}6},\mathsf{4_{3}5_{4}6}), (\mathsf{3_{2}4_{3}},\mathsf{4_{2}5_{3}}), \\
&(\mathsf{3_{2}4_{3}56},\mathsf{4_{3}56}), (\mathsf{3_{2}4_{3}5},\mathsf{4_{3}5}), (\mathsf{3_{2}456},\mathsf{4_{2}56}), (\mathsf{3_{2}45},\mathsf{4_{2}5}), (\mathsf{3_{2}4},\mathsf{4_{2}5_{4}}), (\mathsf{3_{1}456},\mathsf{4_{1}56}), \\ 
& \hspace{35mm} (\mathsf{3_{1}45},\mathsf{4_{1}5}), (\mathsf{3_{1}4},\mathsf{4_{1}5_{4}}), (\mathsf{3_{1}4_{3}},\mathsf{4_{1}5_{3}}).
\end{align*}
Lastly, for case (iii), we are dealing with the following list of pairs:
\[ (\mathsf{3456},\mathsf{456}), (\mathsf{345},\mathsf{45}), (\mathsf{34},\mathsf{4}), (\mathsf{3},\mathsf{43}), (\mathsf{32},\mathsf{432}), (\mathsf{321},\mathsf{4321}). \]
Once again, GAP computations confirm all $34$ of these inequalities.

We are now able to provide an answer to Kostant's problem for $\mathfrak{sl}_{7}$:

\begin{thm} \label{Thm:KostantsProblemA6}
A given $w\in\mathfrak{S}_{7}$ is Kostant negative if and only if it belongs to the same left cell of any of the involutions in the following table:
{\footnotesize
\begin{longtblr}[
  caption = {Kostant negative involutions $d\in\mathfrak{I}_{7}$},
  label = {Table:4},
]{colspec = {cccccccc}}
 $\mathsf{12_{1}3_{1}4_{1}6}$ & $\mathsf{12_{1}3_{1}45_{4}6_{1}}$ & $\mathsf{12_{1}3_{1}5}$ & $\mathsf{12_{1}3_{1}56_{5}}$ & $\mathsf{12_{1}3_{2}4_{1}5_{1}}$ & $\mathsf{12_{1}34_{1}6}$ & $\mathsf{12_{1}34_{3}5_{1}}$ & $\mathsf{12_{1}34_{3}5_{3}6_{1}}$ \\
 $\mathsf{12_{1}34_{3}56_{1}}$ & $\mathsf{12_{1}4}$ & $\mathsf{12_{1}4_{1}56_{4}}$ & $\mathsf{12_{1}4_{3}5_{1}6_{4}}$ & $\mathsf{12_{1}45_{4}}$ & $\mathsf{12_{1}45_{4}6_{4}}$ & $\mathsf{12_{1}456_{4}}$ & $\mathsf{12_{1}46}$ \\
 $\mathsf{12_{1}5_{4}6_{5}}$ & $\mathsf{12_{1}56_{5}}$ & $\mathsf{123_{1}5}$ & $\mathsf{123_{1}56_{5}}$ & $\mathsf{123_{2}4_{1}}$ & $\mathsf{123_{2}4_{1}6}$ & $\mathsf{123_{2}4_{2}5_{1}}$ & $\mathsf{123_{2}4_{2}56_{1}}$ \\
 $\mathsf{123_{2}45_{4}6_{1}}$ & $\mathsf{1234_{1}6}$ & $\mathsf{1234_{3}5_{1}}$ & $\mathsf{124_{1}56_{4}}$ & $\mathsf{13}$ & $\mathsf{13_{1}4_{1}56_{3}}$ & $\mathsf{13_{1}4_{3}6}$ & $\mathsf{13_{1}45_{3}}$ \\
 $\mathsf{13_{1}45_{4}6_{3}}$ & $\mathsf{13_{1}456_{3}}$ & $\mathsf{13_{1}5_{3}6_{5}}$ & $\mathsf{13_{2}4_{1}5_{3}}$ & $\mathsf{13_{2}4_{1}56_{3}}$ & $\mathsf{13_{2}4_{2}5_{1}6_{3}}$ & $\mathsf{13_{2}4_{3}5_{1}6_{3}}$ & $\mathsf{134_{3}}$ \\
 $\mathsf{134_{1}5_{3}6_{3}}$ & $\mathsf{134_{3}5_{3}}$ & $\mathsf{134_{3}5_{3}6_{3}}$ & $\mathsf{134_{3}56_{3}}$ & $\mathsf{134_{3}6}$ & $\mathsf{1345_{3}}$ & $\mathsf{1345_{4}6_{3}}$ & $\mathsf{13456_{3}}$ \\
 $\mathsf{135}$ & $\mathsf{135_{3}6_{5}}$ & $\mathsf{1356_{5}}$ & $\mathsf{136}$ & $\mathsf{14_{3}5_{4}}$ & $\mathsf{14_{3}5_{3}6_{4}}$ & $\mathsf{14_{3}56_{4}}$ & $\mathsf{145_{4}}$ \\
 $\mathsf{146}$ & $\mathsf{2_{1}3_{1}4_{1}56_{2}}$ & $\mathsf{2_{1}3_{1}4_{2}6}$ & $\mathsf{2_{1}3_{1}45_{2}}$ & $\mathsf{2_{1}3_{1}45_{4}6_{2}}$ & $\mathsf{2_{1}3_{1}5_{2}6_{5}}$ & $\mathsf{2_{1}3_{2}4_{1}56_{2}}$ & $\mathsf{2_{1}3_{2}5}$ \\
 $\mathsf{2_{1}3_{2}56_{5}}$ & $\mathsf{2_{1}34_{2}6}$ & $\mathsf{2_{1}34_{3}5_{2}}$ & $\mathsf{2_{1}35_{2}6_{5}}$ & $\mathsf{2_{1}4_{2}5_{4}}$ & $\mathsf{2_{1}4_{2}56_{4}}$ & $\mathsf{2_{1}4_{3}5_{2}6_{4}}$ & $\mathsf{23_{2}}$ \\
 $\mathsf{23_{2}4_{2}}$ & $\mathsf{23_{2}4_{2}5_{2}}$ & $\mathsf{23_{2}4_{2}56_{2}}$ & $\mathsf{23_{2}4_{2}6}$ & $\mathsf{23_{2}4_{3}5_{2}6_{2}}$ & $\mathsf{23_{2}45_{2}}$ & $\mathsf{23_{2}45_{4}6_{2}}$ & $\mathsf{23_{2}456_{2}}$ \\
 $\mathsf{23_{2}5}$ & $\mathsf{23_{2}5_{2}6_{5}}$ & $\mathsf{23_{2}56_{5}}$ & $\mathsf{23_{2}6}$ & $\mathsf{234_{2}6}$ & $\mathsf{234_{3}5_{2}}$ & $\mathsf{24}$ & $\mathsf{24_{2}5_{4}}$ \\
 $\mathsf{24_{2}56_{4}}$ & $\mathsf{2435_{2}6_{4}}$ & $\mathsf{245_{4}}$ & $\mathsf{245_{4}6_{4}}$ & $\mathsf{2456_{4}}$ & $\mathsf{246}$ & $\mathsf{25_{4}6_{5}}$ & $\mathsf{3_{2}4_{2}56_{3}}$ \\
 $\mathsf{3_{2}4_{3}6}$ & $\mathsf{3_{2}45_{3}}$ & $\mathsf{3_{2}5_{3}6_{5}}$ & $\mathsf{34_{3}}$ & $\mathsf{34_{3}5_{3}}$ & $\mathsf{34_{3}56_{3}}$ & $\mathsf{34_{3}6}$ & $\mathsf{35}$ \\
 $\mathsf{356_{5}}$ & $\mathsf{45_{4}}$ & $\mathsf{46}$ & & & & &
\end{longtblr}}
\end{thm}
For any $n\in\mathbb{Z}_{\geq 1}$, consider the non-negative integers
\[ \mathbf{p}_{n}:=|\{w\in\mathfrak{S}_{n} \ | \ \mathsf{K}(w)=\mathtt{true}\}| \hspace{2mm} \text{ and } \hspace{2mm} \mathbf{pi}_{n}:=|\{d\in\mathfrak{I}_{n} \ | \ \mathsf{K}(d)=\mathtt{true}\}|,  \]
We now know the first seven terms for both of these sequences:
\[ (\mathbf{p}_{n})_{n\geq 1}=(1,2,6,22,94,480,2631,\dots) \hspace{6mm} (\mathbf{pi}_{n})_{n\geq 1}=(1,2,4,9,21,51,125,\dots). \]
Neither of these sequences (nor their complements $\mathbf{n}_{n}:=n!-\mathbf{p}_{n}$ and $\mathbf{ni}_{n}:=|\mathfrak{I}_{n}|-\mathbf{pi}_{n}$) appear in the \cite{OEIS}. However, the first six terms of $\mathbf{pi}_{n}$ match numerous sequences, and hence knowing the seventh term has allowed us to rule these out. Notably, $\mathbf{pi}_{n}$ agrees with the sequence of \emph{Motzkin numbers} (A001006 in \cite{OEIS}) for the first six terms, but the seventh term differs by two.

For any $n\in\mathbb{Z}_{\geq 1}$ and integer partition $\lambda\in\Lambda_{n}$, consider the non-negative integers
\[ \mathbf{pi}_{n}^{\lambda}:=|\{d\in\mathfrak{I}_{n} \ | \ \mathsf{sh}(d)=\lambda \ \text{and} \ \mathsf{K}(d)=\mathtt{true}\}|. \]
Let $\lambda'$ denote the transpose of $\lambda$. The first time we obtain the inequality $\mathbf{pi}_{n}^{\lambda}\neq\mathbf{pi}_{n}^{\lambda'}$ is for the case $n=6$ and partitions $\lambda=(4,1^2)$ and $\lambda'=(3,1^3)$. In particular, we have
\[ \mathbf{pi}_{6}^{(4,1^2)}=8\neq7=\mathbf{pi}_{6}^{(3,1^3)}. \]
When $n=7$, we now know $\mathbf{pi}_{7}^{\lambda}=\mathbf{pi}_{7}^{\lambda'}$ for $\lambda\in\Lambda_{7}\backslash\{(5,1^2),(3,1^4),(4,2,1),(3,2,1^2)\}$, while 
\[ \mathbf{pi}_{7}^{(5,1^2)}=12\neq9=\mathbf{pi}_{7}^{(3,1^4)} \hspace{2mm} \text{and } \hspace{2mm} \mathbf{pi}_{7}^{(4,2,1)}=19\neq21=\mathbf{pi}_{7}^{(3,2,1^2)}. \]

\section{Indecomposability Conjecture for $\mathfrak{sl}_{7}$}
\label{Sec:IndecompConjA6}

In this section we show that the Indecomposability Conjecture \eqref{Conj:IndecompConj} holds in $\mathfrak{sl}_{7}$. By (b) of \Cref{Prop:IndecompConjLRCellInv}, it suffices to prove $\mathsf{KM}(\star,d)=\mathtt{true}$ for all $d\in\mathfrak{I}_{7}$. We begin with the following:

\begin{lem}\label{Lem:ParaLifts7Ind}
For any $I\subsetneq S_{7}$ and $w\in\mathfrak{S}_{7}(I)$, we have that $\mathsf{KM}(\star,ww_{0}^{I}w_{0})=\mathtt{true}$.
\end{lem}

\begin{proof}
For any $x\in\mathfrak{S}_{n}$, by (c) of \Cref{Prop:IndecompConjLRCellInv} we have
\[ \mathsf{KM}(x,ww_{0}^{I}w_{0})=\mathsf{KM}(w_{0}w_{0}^{I}w^{-1}w_{0},x^{-1}w_{0}). \]
We know that $|\mathsf{Sup}(w_{0}^{I}w^{-1})|\leq 5$ and so $|\mathsf{Sup}(w_{0}w_{0}^{I}w^{-1}w_{0})|\leq 5$ since conjugation by $w_{0}$ permutes the simple transpositions. Hence, by \Cref{Cor:KMxSmallSup} we have that $\mathsf{KM}(w_{0}w_{0}^{I}w^{-1}w_{0},x^{-1}w_{0})=\mathtt{true}$ for all $x\in\mathfrak{S}_{n}$, which implies $\mathsf{KM}(\star,ww_{0}^{I}w_{0})=\mathtt{true}$.
\end{proof}

For $d$ belonging to \Cref{Table:1}, we have $\mathsf{KM}(\star,d)=\mathtt{true}$ by (b) of \Cref{Prop:IndecompConjLRCellInv} and \Cref{Lem:ParaLifts7Ind}. Consulting Section 4.3 of \cite{MMM24}, we have $\mathsf{KM}(\star,d)=\mathtt{true}$ for all $d$ belonging to \Cref{Table:2}. Also, cases (6), (10), and (11) from \Cref{SubSec:RemainingCases} are Kostant positive, hence they also satisfy the indecomposability conjecture by \Cref{Eq:KiffKhKM}. 

Therefore, we only need to prove $\mathsf{KM}(\star,d)=\mathtt{true}$ for $d$ a Kostant negative case in \Cref{SubSec:RemainingCases} and for $d$ belonging to \Cref{Table:3}. We list all such involutions here:

\begin{multicols}{4}
\begin{itemize}
\item[(1)]
\begin{itemize}
\item[(a)] $\mathsf{124_{1}56_{4}}$
\item[(b)] $\mathsf{13_{1}456_{3}}$
\end{itemize}
\item[(2)]
\begin{itemize}
\item[(a)] $\mathsf{13_{1}45_{3}}$
\item[(b)] $\mathsf{24_{2}56_{4}}$
\end{itemize}
\item[(3)]
\begin{itemize}
\item[(a)] $\mathsf{14_{3}5_{3}6_{4}}$
\item[(b)] $\mathsf{2_{1}3_{1}4_{2}6}$
\end{itemize}
\item[(4)]
\begin{itemize}
\item[(a)] $\mathsf{24_{2}5_{4}}$
\item[(b)] $\mathsf{3_{2}45_{3}}$
\end{itemize}
\end{itemize}
\end{multicols}

\begin{multicols}{4}
\begin{itemize}
\item[(5)]
\begin{itemize}
\item[(a)] $\mathsf{23_{2}4_{2}}$
\item[(b)] $\mathsf{34_{3}5_{3}}$
\end{itemize}
\item[(6)]
\begin{itemize}
\item[(a)] $\mathsf{12_{1}45_{4}}$
\item[(b)] $\mathsf{23_{2}56_{5}}$
\end{itemize}
\item[(7)]
\begin{itemize}
\item[(a)] $\mathsf{123_{1}56_{5}}$
\item[(b)] $\mathsf{12_{1}456_{4}}$
\end{itemize}
\item[(8)]
\begin{itemize}
\item[(a)] $\mathsf{134_{3}}$
\item[(b)] $\mathsf{34_{3}6}$
\end{itemize}
\end{itemize}
\end{multicols}

\begin{multicols}{4}
\begin{itemize}
\item[(9)]
\begin{itemize}
\item[(a)] $\mathsf{135_{3}6_{5}}$
\item[(b)] $\mathsf{2_{1}34_{2}6}$
\end{itemize}
\item[(10)]
\begin{itemize}
\item[(a)] $\mathsf{2_{1}3_{1}5_{2}6_{5}}$
\item[(b)] $\mathsf{2_{1}4_{3}5_{2}6_{4}}$
\end{itemize}
\item[(11)]
\begin{itemize}
\item[(a)] $\mathsf{23_{2}4_{2}6}$
\item[(b)] $\mathsf{134_{3}5_{3}}$
\end{itemize}
\item[(12)]
\begin{itemize}
\item[(a)] $\mathsf{23_{2}5}$
\item[(b)] $\mathsf{245_{4}}$
\end{itemize}
\end{itemize}
\end{multicols}

\begin{multicols}{4}
\begin{itemize}
\item[(13)]
\begin{itemize}
\item[(a)] $\mathsf{23_{2}6}$
\item[(b)] $\mathsf{145_{4}}$
\end{itemize}
\item[(14)]
\begin{itemize}
\item[(a)] $\mathsf{234_{3}5_{2}}$
\item[(b)] $\mathsf{23_{2}45_{2}}$
\end{itemize}
\item[(15)]
\begin{itemize}
\item[(a)] $\mathsf{45_{4}}$
\item[(b)] $\mathsf{23_{2}}$
\end{itemize}
\item[(16)]
\begin{itemize}
\item[(a)] $\mathsf{1345_{3}}$
\item[(b)] $\mathsf{234_{2}6}$
\end{itemize}
\end{itemize}
\end{multicols}
\vspace{-7mm}
\[
\text{(17)} \hspace{2mm} \mathsf{12_{1}56_{5}} \hspace{10mm}
\text{(18)} \hspace{2mm} \mathsf{2_{1}35_{2}6_{5}} \hspace{10mm}
\text{(19)} \hspace{2mm} \mathsf{23_{2}4_{2}5_{2}} \hspace{10mm}
\text{(20)} \hspace{2mm} \mathsf{134_{3}6}
\]
\vspace{-2mm}
\[
\text{(21)} \hspace{2mm} \mathsf{34_{3}} \hspace{10mm} 
\text{(22)} \hspace{2mm} \mathsf{13_{2}4_{1}56_{3}}
\]
We have paired up these cases according to the natural symmetry of the root system, which pairs $d$ with $w_{0}dw_{0}$. So, we only need to consider one involution for each pair. For any $d$ above, we want to show $\mathsf{KM}(x,d)=\mathtt{true}$ for $x\in\mathfrak{S}_{7}$. We now reduce the elements $x$ needed to be considered:

\begin{lem}\label{Lem:ICSimp}
To prove that the Indecomposability conjecture holds in $\mathfrak{sl}_{7}$, it suffices to confirm that $\mathsf{KM}(x,d)=\mathtt{true}$ for all $d$ in the above $22$ cases, and for all $x\in\mathfrak{I}_{7}$ such that $\mathsf{sh}(d)\prec\mathsf{sh}(x)$ and $D(x)\subset D(d)$. Moreover, only the following list of elements $x$ need to be considered:
{\footnotesize
\begin{longtblr}[
  caption = {Suffices to check $\mathsf{KM}(x,d)$ for $d$ in the above $22$ cases and $x$ presented here.},
  label = {Table:5},
]{
  colspec = {|ccc|ccc|},
  rowhead = 1,
  row{even} = {gray9},
  row{1} = {teal9},
}
\hline
$x\in\mathfrak{I}_{7}$ & Shape $\mathsf{sh}(x)$ & Descent set $D(x)$ & $x\in\mathfrak{I}_{7}$ & Shape $\mathsf{sh}(x)$ & Descent set $D(x)$ \\ \hline

$\mathsf{123456_{1}}$ & $(5,1,1)$ & $\{1,6\}$ & $\mathsf{2_{1}3456_{2}}$ & $(4,2,1)$ & $\{2,6\}$  \\ 

$\mathsf{1235_{1}6_{5}}$ & $(4,2,1)$ & $\{1,5\}$ & $\mathsf{3_{1}4_{2}56_{3}}$ & $(3,3,1)$ & $\{3,6\}$  \\ 

$\mathsf{3_{1}4_{1}5_{2}6_{3}}$ & $(3,3,1)$ & $\{3,4\}$ & $\mathsf{2_{1}4_{1}5_{2}6_{4}}$ & $(3,3,1)$ & $\{2,4\}$  \\ 

$\mathsf{2_{1}4_{2}56_{4}}$ & $(3,3,1)$ & $\{2,4,6\}$ & $\mathsf{3_{1}4_{3}5_{2}6_{3}}$ & $(3,3,1)$ & $\{3,5\}$  \\ 

$\mathsf{2_{1}3_{1}5_{2}6_{5}}$ & $(3,3,1)$ & $\{2,3,5\}$ & $\mathsf{2_{1}35_{2}6_{5}}$ & $(3,3,1)$ & $\{2,5\}$  \\ 

$\mathsf{2_{1}4_{3}5_{2}6_{4}}$ & $(3,3,1)$ & $\{2,4,5\}$ & $\mathsf{13_{1}5_{3}6_{5}}$ & $(3,3,1)$ & $\{1,3,5\}$  \\ 

$\mathsf{14_{1}5_{3}6_{4}}$ & $(3,3,1)$ & $\{1,4\}$ & $\mathsf{2_{1}3_{2}4_{3}5_{1}6_{2}}$ & $(3,2,2)$ & $\{2,5\}$  \\ 

$\mathsf{2_{1}3_{2}4_{1}56_{2}}$ & $(3,2,2)$ & $\{2,4,6\}$ & $\mathsf{2_{1}3_{1}456_{2}}$ & $(3,2,2)$ & $\{2,3,6\}$  \\ 
\hline
\end{longtblr}
}
\end{lem} 

\begin{proof}
From above, we know it suffices to confirm $\mathsf{KM}(x,d)=\mathtt{true}$ for $d$ in the above $22$ cases and $x\in\mathfrak{S}_{7}$. But, by (a) of \Cref{Prop:IndecompConjLRCellInv}, \Cref{Eq:ThetaLNonZero}, and \Cref{Eq:ROrderDomDes}, we only need to consider $x\in\mathfrak{I}_{7}$, such that $\mathsf{sh}(d)\preceq\mathsf{sh}(x)$ and $D(x)\subset D(d)$. Also, the inequality $\mathsf{sh}(d)\preceq\mathsf{sh}(x)$ can be improved to a strict inequality $\mathsf{sh}(d)\prec\mathsf{sh}(x)$ by \cite[Lemma~6]{KiM16} (see also Section 5.2 therein).

As for \Cref{Table:5}, all cases $d$ out of the $22$ cases presented above, except $(11)$ and $(19)$, are such that $(3,2,2)\preceq\mathsf{sh}(d)$. For these $20$ cases, it thus suffices to confirm $\mathsf{KM}(x,d)=\mathtt{true}$ for all $x\in\mathfrak{I}_{7}$ such that $(3,2,2)\prec\mathsf{sh}(x)$, or equivalently, all $x\in\mathfrak{I}_{7}$ whose shape is one of the following seven:
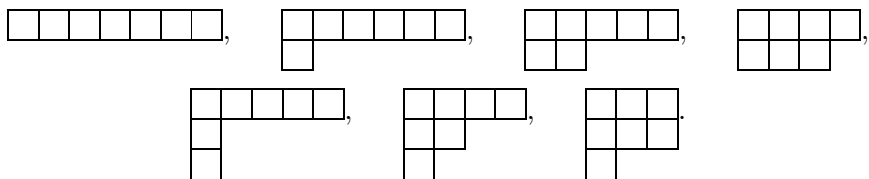
\[ {\tiny\ydiagram{7}}, \hspace{6mm} {\tiny\ydiagram{6,1}}, \hspace{6mm} {\tiny\ydiagram{5,2}}, \hspace{6mm} {\tiny\ydiagram{4,3}}, \]
\[ {\tiny\ydiagram{5,1,1}}, \hspace{6mm} {\tiny\ydiagram{4,2,1}}, \hspace{6mm} {\tiny\ydiagram{3,3,1}}. \]
For $\lambda\in\{(7), (6,1), (5,2), (4,3)\}$, one can check that any involution $x\in\mathfrak{I}_{7}$ such that $\mathsf{sh}(x)=\lambda$ satisfies $|\mathsf{Sup}(x)|\leq 5$. Thus by Corollary~\ref{Cor:KMxSmallSup} we have $\mathsf{KM}(x,\star)=\mathtt{true}$, hence such involutions need not be considered. As for shapes $(5,1,1)$, $(4,2,1)$, and $(3,3,1)$, one can check that any involution of such a shape which has maximal support (i.e. a support of size $6$) is present in \Cref{Table:5}, with the others which do not have maximal support not needed to be considered due to \Cref{Cor:KMxSmallSup}. For example, the following is the collection of all $35$ involutions of shape $(4,2,1)$:
\begin{center}
\begin{tabular}[h]{ccccccc}
$\mathsf{356_{5}}$ & $\mathsf{4_{3}56_{4}}$ & $\mathsf{34_{3}6}$ & $\mathsf{4_{3}5_{3}6_{4}}$ & $\mathsf{35_{3}6_{5}}$ & $\mathsf{2_{1}3_{1}4_{2}}$ & $\mathsf{256_{5}}$  \\
$\mathsf{245_{4}}$ & $\mathsf{2456_{4}}$ & $\mathsf{2_{1}34_{2}}$ & $\mathsf{3_{2}45_{3}}$ & $\mathsf{3_{2}456_{3}}$ & $\mathsf{23_{2}6}$ & $\mathsf{23_{2}5}$ \\
$\mathsf{2_{1}345_{2}}$ & $\mathsf{3_{2}4_{2}5_{3}}$ & $\mathsf{234_{2}6}$ & $\mathsf{24_{2}5_{4}}$ & $\mathsf{2_{1}3456_{2}}$ & $\mathsf{235_{2}6_{5}}$ & $\mathsf{156_{5}}$ \\
$\mathsf{145_{4}}$ & $\mathsf{1456_{4}}$ & $\mathsf{134_{3}}$ & $\mathsf{1345_{3}}$ & $\mathsf{13456_{3}}$ & $\mathsf{12_{1}6}$ & $\mathsf{12_{1}5}$ \\
$\mathsf{12_{1}4}$ & $\mathsf{123_{1}6}$ & $\mathsf{123_{1}5}$ & $\mathsf{13_{1}4_{3}}$ & $\mathsf{1234_{1}6}$ & $\mathsf{124_{1}5_{4}}$ & $\mathsf{1235_{1}6_{5}}$
\end{tabular}
\end{center}
Only $\mathsf{2_{1}3456_{2}}$ and $\mathsf{1235_{1}65}$ have maximal support and these are the only involutions in \Cref{Table:5} with shape $(4,2,1)$. So it remains to explain why \Cref{Table:5} is sufficient for cases $(11)$ and $(19)$.

For case $(11)$, we have $d=23_{2}4_{2}6$ and $\mathsf{sh}(d)=(3,2,1,1)$. We only need to consider $x\in\mathfrak{I}_{7}$ with maximal support and $\mathsf{sh}(d)\prec\mathsf{sh}(x)$. Such an $x$ has shape given by one of the seven mentioned above (which are already accounted for in \Cref{Table:5}), or either $(4,1,1,1)$ or $(3,2,2)$. It can be checked that all such $x$ of shape $(4,1,1,1)$, along with their descent sets $D(x)$, are as follows:
\begin{center}
\begin{tabular}[h]{cc}
$\mathsf{12_{1}3456_{1}}$ & $\{\mathsf{1,2,6}\}$ \\
$\mathsf{12345_{4}6_{1}}$ & $\{\mathsf{1,5,6}\}$ \\
$\mathsf{1234_{3}56_{1}}$ & $\{\mathsf{1,4,6}\}$ \\
$\mathsf{123_{2}456_{1}}$ & $\{\mathsf{1,3,6}\}$
\end{tabular}
\end{center}
Similarly, all maximal support involutions $x$ of shape $(3,2,2)$, and their descent sets $D(x)$, are:
\begin{center}
\begin{tabular}[h]{cc||cc}
$\mathsf{2_{1}3_{2}4_{3}5_{1}6_{2}}$ & $\{\mathsf{2,5}\}$ & $\mathsf{2_{1}3_{2}4_{1}56_{2}}$ & $\{\mathsf{2,4,6}\}$ \\
$\mathsf{2_{1}3_{1}456_{2}}$ & $\{\mathsf{2,3,6}\}$ & $\mathsf{13_{1}456_{3}}$ & $\{\mathsf{1,3,6}\}$ \\
$\mathsf{124_{1}56_{4}}$ & $\{\mathsf{1,4,6}\}$ & $\mathsf{13_{2}4_{1}56_{3}}$ & $\{\mathsf{1,3,4,6}\}$ \\
$\mathsf{124_{3}5_{1}6_{4}}$ & $\{\mathsf{1,4,5}\}$ & $\mathsf{13_{2}4_{3}5_{1}6_{3}}$ & $\{\mathsf{1,3,5}\}$
\end{tabular}
\end{center}
We have $D(d)=\{2,3,4,6\}$. Hence, by the condition of inclusion of descent sets, the only maximal support involutions $x$ of shape $(4,1,1,1)$ and $(3,2,2)$ which we need to consider are $\mathsf{2_{1}3_{2}4_{1}56_{2}}$ and $\mathsf{2_{1}3_{1}456_{2}}$, both of shape $(3,2,2)$. These two cases are included in \Cref{Table:5}. So far, this accounts for all entries in \Cref{Table:5} with the exception of $2_{1}3_{2}4_{3}5_{1}6_{2}$. 

For case $(19)$, we have $d=23_{2}4_{2}5_{2}$, $\mathsf{sh}(d)=(3,1,1,1,1)$, and $D(d)=\{2,3,4,5\}$.  Thus we only need to consider $x\in\mathfrak{I}_{7}$ of maximal support and whose shape is one of the nine mentioned above or $(3,2,1,1)$. For shape $(4,1,1,1)$, the descent set condition rules out the four involutions listed above, while for shape $(3,2,2)$, the only involutions which we need to consider are the same two $\mathsf{2_{1}3_{2}4_{1}56_{2}}$ and $\mathsf{2_{1}3_{1}456_{2}}$, for case $(11)$, and the addition case $\mathsf{2_{1}3_{2}4_{3}5_{1}6_{2}}$. Hence the lemma follows as long as no maximal support $x\in\mathfrak{I}_{7}$ of shape $(3,2,1,1)$ needs to be considered with regard to this $d$. One can check that there are precisely $11$ involutions $x$ with maximal support and of shape $(3,2,1,1)$, and for all of these either $\mathsf{1}$ or $\mathsf{6}$ belongs to its descent set. Thus $D(x)$ is not a subset of $D(d)$, meaning no such involutions need to be considered.  
\end{proof}

Therefore, to prove that the Indecomposability Conjecture holds in $\mathfrak{sl}_{7}$, we only need to confirm that $\mathsf{KM}(x,d)=\mathtt{true}$ for $d$ in cases $(1)$ to $(22)$ above, and for all $x$ in \Cref{Table:5}. Most of these cases will be shown by checking that $[\theta_{x}L(d)]=0$, and hence $\theta_{x}L(d)=0$. For the non-zero situations, the following lemma will be helpful: Firstly, for any $x,y,z\in\mathfrak{S}_{n}$, let $[\theta_{x}L(y): L(z)]$ denote the graded composition multiplicity of $L(z)$ within $\theta_{x}L(y)$. Explicitly, by \cite[Proposition~3.3]{KMM23}, 
\[ [\theta_{x}L(y): L(z)]=[\underline{H}_{y}](\underline{H}_{z}\underline{H}_{x^{-1}})\in\mathbb{Z}_{\geq 0}[v,v^{-1}], \]
i.e. the coefficient of $\underline{H}_{y}$ in the product $\underline{H}_{z}\underline{H}_{x^{-1}}$ when expressed in terms of the Kazhdan-Lusztig basis. For $i\in\mathbb{Z}_{\geq 0}$, let $[\theta_{x}L(y): L(z)\langle i\rangle]$ denote the composition multiplicity of $L(z)\langle i\rangle$ in $\theta_{x}L(y)$ which is the coefficient of $v^{-i}$ in $[\theta_{x}L(y): L(z)]$. Lastly, we let $\mathbf{a}:\mathfrak{S}_{n}\rightarrow\mathbb{Z}_{\geq 0}$ be Lusztig's $\mathbf{a}$-function (see \cite{Lu87}). This function is uniquely defined by the two properties of being constant on elements sharing a shape, and $\mathbf{a}(w_{0}^{I})=\ell(w_{0}^{I})$ for any $I\subset S_{n}$.

\begin{lem}\label{Lem:IndCoeff1AtaFunc}
For $x\in\mathfrak{I}_{n}$ and $d\in\mathfrak{S}_{n}$, the module $\theta_{x}L(d)$ is indecomposable whenever
\[ [\theta_{x}L(d):L(d)\langle\mathbf{a}(x)\rangle]=[v^{-\mathbf{a}(x)}][\theta_{x}L(d): L(d)]=1. \] 
\end{lem}

\begin{proof}
This result is implicit from the work of \cite{KMM23}, we simply collect all of the details together. We prove this by showing that the endomorphism space of $\theta_{x}L(d)$ is positively graded with $1$-dimensional component of degree zero. Firstly, for any $y\in\mathfrak{S}_{n}$, the projective functor $\theta_{y}$ is adjoint to $\theta_{y^{-1}}$. Also, by \cite[Lemma~5.3]{KMM23}, we have
\[ \theta_{x^{-1}}\theta_{x}\cong\theta_{x}\langle-\mathbf{a}(x)\rangle\oplus\bigoplus_{w\in\mathfrak{S}_{n}}\bigoplus_{i<\mathbf{a}(w)}\theta_{w}\langle -i\rangle^{\oplus m(w,i)} \]
with $m(w,i)\geq 0$. Thus $\mathsf{Hom}_{\mathcal{O}_{0}^{\mathbb{Z}}}(\theta_{x}L(d),\theta_{x}L(d))=\mathsf{Hom}_{\mathcal{O}_{0}^{\mathbb{Z}}}(\theta_{x^{-1}}\theta_{x}L(d),L(d))$, which becomes
\[ \mathsf{Hom}_{\mathcal{O}_{0}^{\mathbb{Z}}}(\theta_{x}L(d),L(d)\langle\mathbf{a}(x)\rangle)\oplus\bigoplus_{w\in\mathfrak{S}_{n}}\bigoplus_{i<\mathbf{a}(w)}\mathsf{Hom}_{\mathcal{O}_{0}^{\mathbb{Z}}}\left(\theta_{w}L(d),L(d)\langle i\rangle\right)^{\oplus m(w,i)},\]
where we have moved shifts $\langle i\rangle$ in the domain to $\langle-i\rangle$ in the codomain. By \cite[Theorem~2.2]{KMM23}, only the first hom-space above can be non-zero (and all other summands require strictly positive shifts to lead  to a potentially non-zero space), hence
\begin{equation}\label{Eq:EndoSpace}
\mathsf{End}_{\mathcal{O}_{0}^{\mathbb{Z}}}(\theta_{x}L(d))=\mathsf{Hom}_{\mathcal{O}_{0}^{\mathbb{Z}}}(\theta_{x}L(d),L(d)\langle\mathbf{a}(x)\rangle).
\end{equation}
As $x$ is an involution, by \cite[Section 7]{MM16}, there is a non-zero natural transformation
\[ \eta:\theta_{x}\rightarrow\theta_{e}\langle\mathbf{a}(x)\rangle. \]
By \cite[Proposition~17]{MM11}, the cokernel of $\eta$ is killed by $\theta_{x}$. Thus $\eta_{L(d)}$ is non-zero since $\theta_{x}L(d)\neq 0$ (as $[\theta_{x}L(d):L(d)\langle\mathbf{a}(x)\rangle]\neq 0$). So the dimension of the space of morphisms $\theta_{x}L(d)\rightarrow L(d)\langle\mathbf{a}(x)\rangle$ is at least $1$. But, as $[\theta_{x}L(d):L(d)\langle\mathbf{a}(x)\rangle]=1$, the dimension cannot be any larger that $1$, and so it is precisely $1$. Hence, the endomorphism space of \Cref{Eq:EndoSpace} has dimension $1$.
\end{proof}

We now tackle the $22$ cases presented at the start of this section.

\textbf{Case} (2)(a): We have $d:=\mathsf{13_{1}45_{3}}$, $\mathsf{sh}(d)=(3,2,2)$, and $D(d)=\{1,3,5\}$. By \Cref{Lem:ICSimp} we only need to check $\mathsf{KM}(x_{i},d)=\mathtt{true}$ for each $i\in[3]$ where
\[ x_{1}:=\mathsf{1235_{1}6_{5}}, \hspace{2mm} x_{2}:=\mathsf{3_{1}4_{3}5_{2}6_{3}}, \hspace{2mm} x_{3}:=\mathsf{13_{1}5_{3}6_{5}}. \]
By GAP3 computations, we have $[\theta_{x_{1}}L(d)]=0$ and
\[ [\theta_{x_{2}}L(d): L(d)]=[\theta_{x_{3}}L(d): L(d)]=v^{5}+5v^{3}+10v+10v^{-1}+5v^{-3}+v^{-5}. \]
We have that $\mathbf{a}(x_{2})=\mathbf{a}(x_{3})=5$, hence this case holds by \Cref{Lem:IndCoeff1AtaFunc}. 

\textbf{Case} (3)(a): We have $d:=\mathsf{14_{3}5_{3}6_{4}}$, $\mathsf{sh}(d)=(3,3,1)$, and $D(d)=\{1,4,5\}$. From \Cref{Lem:ICSimp}, we only need to check $\mathsf{KM}(x,d)=\mathtt{true}$ for $x:=\mathsf{1235_{1}6_{5}}$. By GAP3 computations, $[\theta_{x}L(d)]=0$.

\textbf{Case} (4)(a): We have $d:=\mathsf{14_{3}5_{3}6_{4}}$, $\mathsf{sh}(d)=(4,2,1)$, and $D(d)=\{2,4\}$. From \Cref{Lem:ICSimp}, this case is solved, that is, there exists no $x$ in \Cref{Table:5} such that $\mathsf{sh}(d)\prec\mathsf{sh}(x)$ and $D(x)\subseteq D(d)$.

\textbf{Case} (5)(a): We have $d:=\mathsf{23_{2}4_{2}}$, $\mathsf{sh}(d)=(4,1,1,1)$, and $D(d)=\{2,3,4\}$. From \Cref{Lem:ICSimp}, we only need to check $\mathsf{KM}(x_{i},d)=\mathtt{true}$ for $i\in[2]$ where $x_{1}:=\mathsf{3_{1}4_{1}5_{2}6_{3}}$ and $x_{2}:=\mathsf{2_{1}4_{1}5_{2}6_{4}}$. By GAP3 computations, we have that $[\theta_{x_{1}}L(d)]=[\theta_{x_{2}}L(d)]=0$.

\textbf{Case} (6)(a): We have $d:=\mathsf{12_{1}45_{4}}$, $\mathsf{sh}(d)=(3,2,2)$, and $D(d)=\{1,2,4,5\}$. By \Cref{Lem:ICSimp} we only need to check $\mathsf{KM}(x_{i},d)=\mathtt{true}$ for each $i\in[5]$ where
\[ x_{1}:=\mathsf{1235_{1}6_{5}}, \hspace{2mm} x_{2}:=\mathsf{2_{1}4_{1}5_{2}6_{4}}, \hspace{2mm} x_{3}:=\mathsf{2_{1}35_{2}6_{5}}, \hspace{2mm} x_{4}:=\mathsf{2_{1}4_{3}5_{2}6_{4}}, \hspace{2mm} x_{5}:=\mathsf{14_{1}5_{3}6_{4}}. \]
By GAP3 computations, we have $[\theta_{x_{2}}L(d)]=[\theta_{x_{3}}L(d)]=0$ and
\begin{align*}
[\theta_{x_{1}}L(d): L(d)]&=v^{4}+4v^{2}+6+4v^{-2}+v^{-4}, \\
[\theta_{x_{4}}L(d): L(d)]&=v^{5}+4v^{3}+7v+7v^{-1}+4v^{-3}+v^{-5}, \\
[\theta_{x_{5}}L(d): L(d)]&=v^{5}+5v^{3}+10v+10v^{-1}+5v^{-3}+v^{-5}.
\end{align*}
Since $\mathbf{a}(x_{1})=4$ and $\mathbf{a}(x_{4})=\mathbf{a}(x_{5})=5$, such case holds by \Cref{Lem:IndCoeff1AtaFunc}. 

\textbf{Case} (7)(a): We have $d:=\mathsf{123_{1}56_{5}}$, $\mathsf{sh}(d)=(3,2,2)$, and $D(d)=\{1,3,5,6\}$. By \Cref{Lem:ICSimp} we only need to check $\mathsf{KM}(x_{i},d)=\mathtt{true}$ for each $i\in[5]$ where
\[ x_{1}:=\mathsf{123456_{1}}, \hspace{2mm} x_{2}:=\mathsf{1235_{1}6_{5}}, \hspace{2mm} x_{3}:=\mathsf{3_{1}4_{2}56_{3}}, \hspace{2mm} x_{4}:=\mathsf{3_{1}4_{3}5_{2}6_{3}}, \hspace{2mm} x_{5}:=\mathsf{13_{1}5_{3}6_{5}}. \]
By GAP3 computations, we have $[\theta_{x_{3}}L(d)]=[\theta_{x_{4}}L(d)]=0$ and
\begin{align*}
[\theta_{x_{2}}L(d): L(d)]&=v^{4}+4v^{2}+6+4v^{-2}+v^{-4}, \\
[\theta_{x_{5}}L(d): L(d)]&=v^{5}+5v^{3}+10v+10v^{-1}+5v^{-3}+v^{-5}.
\end{align*}
As $\mathbf{a}(x_{2})=4$ and $\mathbf{a}(x_{5})=5$, such case hold by \Cref{Lem:IndCoeff1AtaFunc}. By (a) of \Cref{Prop:IndecompConjLRCellInv},
to prove that $\mathsf{KM}(x_{1},d)=\mathtt{true}$ 
is equivalent to showing that $\mathsf{KM}(123456_{5},d)=\mathtt{true}$. 
This is the most complicated assertion in this section which we write as
a separate statement.

\begin{lem}
$\mathsf{KM}(\mathsf{123456_{5}},d)=\mathtt{true}$, for $d=\mathsf{123_156_5}$.  
\end{lem}

\begin{proof}
By GAP3 computations, the module $\theta_{\mathsf{123456_{5}}}L(d)$ 
is zero outside degrees $\{0,\pm 1, \pm 2,\pm 3,\pm 4\}$ and it has
one simple module in degree $-4$, being $L(\mathsf{123_{2}56_{5}})$. There are many 
simple modules in degree $-3$, but by adjunction, only the unique
appearance of $L(d)$ in this degree  can give a component of the top. 
So, we need to show that this unique $L(d)$ in degree $-3$ belongs to the 
radical of $\theta_{\mathsf{123456_{5}}}L(d)$, equivalently, that 
$\theta_{\mathsf{123456_{5}}}L(d)$ has an indecomposable quotient of
length two with top $L(\mathsf{123_{2}56_{5}})$ and socle $L(d)$.
Note that $\mathrm{Ext}^1(L(\mathsf{123_{2}56_{5}}),L(d))$ is one-dimensional
as $\mathsf{123_{2}56_{5}}<d$ and the two elements are Bruhat neighbours. 

To this end, we use GAP to compute all modules $\theta_{\mathsf{1}}L(d)$,
$\theta_{\mathsf{12}}L(d)$, $\theta_{\mathsf{123}}L(d)$, 
$\theta_{\mathsf{1234}}L(d)$, $\theta_{\mathsf{12345}}L(d)$,
$\theta_{\mathsf{123456}}L(d)$ and $\theta_{\mathsf{123456_{5}}}L(d)$ and look how 
$L(\mathsf{123_{2}56_{5}})$ reaches degree $-4$ and how $L(d)$ reaches degree
$-3$ in $\theta_{\mathsf{123456_{5}}}L(d)$. We see that this started in
$\theta_{\mathsf{1234}}L(d)$ where $L(\mathsf{123_{2}56_{5}})$ was in degree $-1$
and $L(d)$ was in degree $0$ and they together got shifted by one degree
at each next step. If we could show that the non-split extension between
$L(\mathsf{123_{2}56_{5}})$ in degree $-1$ and $L(d)$ in degree $0$ is realizable
as a subquotient of $\theta_{\mathsf{1234}}L(d)$, then, by adjunction, we obtain
that the appropriate shift of it is realizable as a quotient of 
$\theta_{\mathsf{123456_{5}}}L(d)$, which is exactly what we need.

The simple subquotient $L(\mathsf{123_{2}56_{5}})$ in degree $-1$ of 
$\theta_{\mathsf{1234}}L(d)$ appears inside the Jantzen middle of
$\theta_4 L(\mathsf{123_{2}56_{5}4})$, for a unique subquotient 
$L(\mathsf{123_{2}56_{5}4})$ of $\theta_{\mathsf{123}}L(d)$ in degree $-1$.
The simple subquotient $L(\mathsf{123_{1}56_{5}})$ in degree $0$ of 
$\theta_{\mathsf{1234}}L(d)$ appears inside the Jantzen middle of
$\theta_4 L(\mathsf{123_{1}56_{5}4})$, for a unique subquotient 
$L(\mathsf{123_{1}56_{5}4})$ of $\theta_{\mathsf{123}}L(d)$ in degree $0$.

In turn, 
the simple subquotient $L(\mathsf{123_{2}56_{5}4})$ in degree $-1$ of 
$\theta_{\mathsf{123}}L(d)$ appears inside the Jantzen middle of
$\theta_{\mathsf{3}} L(\mathsf{123_{2}56_{5}})$, for a unique subquotient 
$L(\mathsf{123_{2}56_{5}})$ of $\theta_{\mathsf{12}}L(d)$ in degree $-1$.
The simple subquotient $L(\mathsf{123_{1}56_{5}4})$ in degree $0$ of 
$\theta_{\mathsf{123}}L(d)$ appears inside the Jantzen middle of
$\theta_{\mathsf{3}} L(\mathsf{123_{1}56_{5}})$, for a unique subquotient 
$L(\mathsf{123_{1}56_{5}})$ of $\theta_{\mathsf{12}}L(d)$ in degree $0$.

The module $\theta_{\mathsf{12}}L(d)$ is isomorphic to 
$\theta_{\mathsf{2}} L(\mathsf{123_{2}56_{5}})$, in particular, it has simple
top $L(\mathsf{123_{2}56_{5}})$ in degree $-1$ and hence surjects
onto the unique up to isomorphism
indecomposable module $N$ of length two with
top $L(\mathsf{123_{2}56_{5}})$ in degree $-1$ and socle
$L(\mathsf{123_{1}56_{5}})$ in degree $0$. It remains to show that
$\theta_{\mathsf{4}}\theta_{\mathsf{3}} N$ contains $N$ as a subquotient.

Denote by $M$  the unique up to isomorphism indecomposable module
of length two with top $L(\mathsf{123_{2}56_{5}4})$ in degree $-1$ and socle
$L(\mathsf{123_{1}56_{5}4})$ in degree $0$. GAP3 computations show that 
the Jantzen middle of $\theta_{\mathsf{4}} L(\mathsf{123_{2}56_{5}4})$ is
just $L(\mathsf{123_{2}56_{5}})$ and the Jantzen middle of
$\theta_{\mathsf{4}} L(\mathsf{123_{1}56_{5}4})$ is just 
$L(\mathsf{123_{1}56_{5}})$.
If we assume that $N$ is not a subquotient of $\theta_{\mathsf{4}} M$,
then the latter module has a uniserial quotient of length three
with top
$L(\mathsf{123_{2}56_{5}4})$, middle $L(\mathsf{123_{1}56_{5}4})$ and socle
$L(\mathsf{123_{1}56_{5}})$. Note that $\theta_{\mathsf{3}}$ kills the first
two but not the socle. By adjunction it follows that 
this length three module must be a submodule of 
$\theta_{\mathsf{3}}\theta_{\mathsf{3}} L(123_{1}56_{5})
\cong \theta_{\mathsf{3}} L(\mathsf{123_{1}56_{5}})\oplus 
\theta_{\mathsf{3}} L(\mathsf{123_{1}56_{5}})$,
which is, clearly, false. Hence $N$ is a subquotient of
$\theta_{\mathsf{4}} M$. 

It remains to show that $M$ is a subquotient of $\theta_{\mathsf{3}} N$.
The Jantzen middles of both $\theta_{\mathsf{3}} L(\mathsf{123_{1}56_{5}})$
and $\theta_{\mathsf{3}} L(\mathsf{123_{2}56_{5}})$ contain many simples. However,
one can list all of them and check, using GAP3 or SageMath,
that the only simple in the Jantzen middle for
$\theta_{\mathsf{3}} L(\mathsf{123_{1}56_{5}})$ that has a non-trivial 
first extension with $L(\mathsf{123_{2}56_{5}4})$ is the module
$L(\mathsf{123_{1}56_{5}4})$. Therefore, if we assume
$M$ is not a subquotient of $\theta_{\mathsf{3}} N$, the latter module
has a uniserial quotient of length three with top $L(\mathsf{123_{2}56_{5}})$,
middle $L(\mathsf{123_{1}56_{5}})$ and socle $L(\mathsf{123_{1}56_{5}4})$.
We note that $\theta_{\mathsf{4}}$ kills the first two but not the
last one. Hence, by adjunction, this length three module
must be a submodule of $\theta_{\mathsf{4}}\theta_{\mathsf{4}}L(\mathsf{123_{1}56_{5}4})
\cong \theta_{\mathsf{4}}L(\mathsf{123_{1}56_{5}4})\oplus 
\theta_{\mathsf{4}}L(\mathsf{123_{1}56_{5}4})$ which
is, clearly, false. This completes the proof.
\end{proof}

\textbf{Case} (8)(a): We have $d:=\mathsf{134_{3}}$, $\mathsf{sh}(d)=(4,2,1)$, and $D(d)=\{1,3,4\}$. From \Cref{Lem:ICSimp}, this case is solved, that is, there exists no $x$ in \Cref{Table:5} such that $\mathsf{sh}(d)\prec\mathsf{sh}(x)$ and $D(x)\subseteq D(d)$.

\textbf{Case} (9)(a): We have $d:=\mathsf{135_{3}6_{5}}$, $\mathsf{sh}(d)=(3,3,1)$, and $D(d)=\{1,3,5\}$. By \Cref{Lem:ICSimp} we only need to check $\mathsf{KM}(x,d)=\mathtt{true}$ for $x:=\mathsf{1235_{1}6_{5}}$. By GAP3 computations,
\[ [\theta_{x}L(d): L(d)]=v^{4}+5v^{2}+8+5v^{-2}+v^{-4}. \]
Since $\mathbf{a}(x)=4$, this case holds by \Cref{Lem:IndCoeff1AtaFunc}. 

\textbf{Case} (10)(a): We have $d:=\mathsf{2_{1}3_{1}5_{2}6_{5}}$, $\mathsf{sh}(d)=(3,3,1)$, and $D(d)=\{2,3,5\}$. From \Cref{Lem:ICSimp}, this case is solved, that is, there exists no $x$ in \Cref{Table:5} with $\mathsf{sh}(d)\prec\mathsf{sh}(x)$ and $D(x)\subseteq D(d)$.

\textbf{Case} (11)(a): We have $d:=\mathsf{23_{2}4_{2}6}$, $\mathsf{sh}(d)=(3,2,1,1)$, and $D(d)=\{2,3,4,6\}$. By \Cref{Lem:ICSimp}, we only need to check $\mathsf{KM}(x_{i},d)=\mathtt{true}$ for each $i\in[7]$ where
\[ x_{1}:=\mathsf{2_{1}3456_{2}}, \hspace{2mm} x_{2}:=\mathsf{3_{1}4_{2}56_{3}}, \hspace{2mm} x_{3}:=\mathsf{3_{1}4_{1}5_{2}6_{3}}, \hspace{2mm} x_{4}:=\mathsf{2_{1}4_{1}5_{2}6_{4}}, \]
\[ x_{5}:=\mathsf{2_{1}4_{2}56_{4}}, \hspace{2mm} x_{6}:=\mathsf{2_{1}3_{2}4_{1}56_{2}}, \hspace{2mm} x_{7}:=\mathsf{2_{1}3_{1}456_{2}}. \]
By GAP3 computations, we have $[\theta_{x_{5}}L(d)]=[\theta_{x_{7}}L(d)]=0$ and
\begin{align*}
[\theta_{x_{1}}L(d): L(d)]&=v^{4}+4v^{2}+6+4v^{-2}+v^{-4}, \\
[\theta_{x_{2}}L(d): L(d)]&=v^{5}+5v^{3}+10v+10v^{-1}+5v^{-3}+v^{-5}, \\
[\theta_{x_{3}}L(d): L(d)]=[\theta_{x_{4}}L(d): L(d)]&=v^{5}+4v^{3}+7v+7v^{-1}+4v^{-3}+v^{-5}, \\
[\theta_{x_{6}}L(d): L(d)]&=v^{6}+5v^{4}+11v^{2}+14+11v^{-2}+5v^{-4}+v^{-6}.
\end{align*}
As $\mathbf{a}(x_{1})=4$, $\mathbf{a}(x_{2})=\mathbf{a}(x_{3})=5$, and $\mathbf{a}(x_{6})=6$, this case holds by  \Cref{Lem:IndCoeff1AtaFunc}.

\textbf{Case} (12)(a): We have $d:=\mathsf{23_{2}5}$, $\mathsf{sh}(d)=(4,2,1)$, and $D(d)=\{2,3,5\}$. From \Cref{Lem:ICSimp}, this case is solved, that is, there exists no $x$ in \Cref{Table:5} such that $\mathsf{sh}(d)\prec\mathsf{sh}(x)$ and $D(x)\subseteq D(d)$.

\textbf{Case} (13)(a): We have $d:=\mathsf{23_{2}6}$, $\mathsf{sh}(d)=(4,2,1)$, and $D(d)=\{2,3,6\}$. From \Cref{Lem:ICSimp}, this case is solved, that is, there exists no $x$ in \Cref{Table:5} such that $\mathsf{sh}(d)\prec\mathsf{sh}(x)$ and $D(x)\subseteq D(d)$.

\textbf{Case} (14)(a): We have $d:=\mathsf{234_{3}5_{2}}$, $\mathsf{sh}(d)=(4,1,1,1)$, and $D(d)=\{2,4,5\}$. From \Cref{Lem:ICSimp}, this case is solved, that is, there exists no $x$ in \Cref{Table:5} with $\mathsf{sh}(d)\prec\mathsf{sh}(x)$ and $D(x)\subseteq D(d)$.

\textbf{Case} (15)(a): We have $d:=\mathsf{45_{4}}$, $\mathsf{sh}(d)=(5,1,1)$, and $D(d)=\{4,5\}$. From \Cref{Lem:ICSimp}, this case is solved, that is, there exists no $x$ in \Cref{Table:5} such that $\mathsf{sh}(d)\prec\mathsf{sh}(x)$ and $D(x)\subseteq D(d)$.

\textbf{Case} (16)(a): We have $d:=\mathsf{1345_{3}}$, $\mathsf{sh}(d)=(4,2,1)$, and $D(d)=\{1,3,5\}$. From \Cref{Lem:ICSimp}, this case is solved, that is, there exists no $x$ in \Cref{Table:5} such that $\mathsf{sh}(d)\prec\mathsf{sh}(x)$ and $D(x)\subseteq D(d)$.

\textbf{Case} (17): We have $d:=\mathsf{12_{1}56_{5}}$, $\mathsf{sh}(d)=(3,2,2)$, and $D(d)=\{1,2,5,6\}$. From \Cref{Lem:ICSimp}, we only need to check $\mathsf{KM}(x_{i},d)=\mathtt{true}$ for each $i\in[4]$ where
\[ x_{1}:=\mathsf{123456_{1}}, \hspace{2mm} x_{2}:=\mathsf{2_{1}3456_{2}}, \hspace{2mm} x_{3}:=\mathsf{1235_{1}6_{5}}, \hspace{2mm} x_{4}:=\mathsf{2_{1}35_{2}6_{5}}. \]
By GAP3 computations, we have $[\theta_{x_{4}}L(d)]=0$ and
\[ [\theta_{x_{2}}L(d): L(d)]=[\theta_{x_{3}}L(d): L(d)]=v^{4}+4v^{2}+10+4v^{-2}+v^{-4}. \]
We have that $\mathbf{a}(x_{2})=\mathbf{a}(x_{3})=4$, hence these holds by \Cref{Lem:IndCoeff1AtaFunc}. Lastly, we need to check that $\mathsf{KM}(x_{1},d)=\mathtt{true}$. By (a) of \Cref{Prop:IndecompConjLRCellInv}, it suffices to prove $\mathsf{KM}(123456_{5},d)=\mathtt{true}$. We show this by proving that $\theta_{123456_{5}}L(d)$ has simple top, and hence is indecomposable. Firstly, GAP3 computations tell us that $\theta_{123456_{5}}L(d)$ 
is zero in degrees outside $\{0,\pm 1,\pm 2,\pm 3,\pm 4\}$ and has
one simple in degree   $-4$. Since $\mathbf{a}(123456_{5})=3$, it suffices to prove that no simple module in degree $-3$ belongs to the top of $\theta_{123456_{5}}L(d)$. If some $L(w)\langle3\rangle$ appears in the top, then, by adjunction of $\theta_{123456_{5}}$ and $\theta_{56_{1}}$, $L(d)\langle3\rangle$ will appear in the top of $\theta_{56_{1}}L(w)$. One can confirm via GAP3 computations, that for every $w\in\mathfrak{S}_{n}$ such that $L(w)\langle3\rangle$ appears in $\theta_{123456_{5}}L(d)$, either $\theta_{56_{1}}L(d)=0$ or there is no $L(d)\langle3\rangle$ appearing in $\theta_{56_{1}}L(d)$, thus we are done.

\textbf{Case} (18): We have $d:=\mathsf{2_{1}35_{2}6_{5}}$, $\mathsf{sh}(d)=(3,3,1)$, and $D(d)=\{2,5\}$. From \Cref{Lem:ICSimp}, this case is solved, that is, there exists no $x$ in \Cref{Table:5} such that $\mathsf{sh}(d)\prec\mathsf{sh}(x)$ and $D(x)\subseteq D(d)$.

\textbf{Case} (20): We have $d:=\mathsf{134_{3}6}$, $\mathsf{sh}(d)=(3,3,1)$, and $D(d)=\{1,3,4,6\}$. By \Cref{Lem:ICSimp} we only need to check $\mathsf{KM}(x,d)=\mathtt{true}$ for $x:=\mathsf{123456_{1}}$. By GAP3 computations,
\[ [\theta_{x}L(d): L(d)]=v^{3}+3v+3v^{-1}+v^{-3}. \]
Since $\mathbf{a}(x)=3$, this case holds by \Cref{Lem:IndCoeff1AtaFunc}. 

\textbf{Case} (21): We have $d:=\mathsf{34_{3}}$, $\mathsf{sh}(d)=(5,1,1)$, and $D(d)=\{3,4\}$. From \Cref{Lem:ICSimp}, this case is solved, that is, there exists no $x$ in \Cref{Table:5} such that $\mathsf{sh}(d)\prec\mathsf{sh}(x)$ and $D(x)\subseteq D(d)$.

\textbf{Case} (22): We have $d:=\mathsf{13_{2}4_{1}56_{3}}$, $\mathsf{sh}(d)=(3,2,2)$, and $D(d)=\{1,3,4,6\}$. By \Cref{Lem:ICSimp} we only need to check $\mathsf{KM}(x_{i},d)=\mathtt{true}$ for each $i\in[5]$ where
\[ x_{1}:=\mathsf{123456_{1}}, \hspace{2mm} x_{2}:=\mathsf{3_{1}4_{2}56_{3}}, \hspace{2mm} x_{3}:=\mathsf{3_{1}4_{1}5_{2}6_{3}}, \hspace{2mm} x_{4}:=\mathsf{14_{1}5_{3}6_{4}}. \]
By GAP3 computations, we have $[\theta_{x_{2}}L(d)]=[\theta_{x_{3}}L(d)]=[\theta_{x_{4}}L(d)]=0$ and
\[ [\theta_{x_{1}}L(d): L(d)]=v^{3}+3v+3v^{-1}+v^{-3}. \]
As $\mathbf{a}(x_{1})=3$, this case hold by \Cref{Lem:IndCoeff1AtaFunc}. 

\textbf{Case} (19): We have $d:=\mathsf{23_{2}4_{2}5_{2}}$, $\mathsf{sh}(d)=(3,1,1,1,1)$, and $D(d)=\{2,3,4,5\}$. By \Cref{Lem:ICSimp}, we only need to check $\mathsf{KM}(x_{i},d)=\mathtt{true}$ for each $i\in[7]$ where
\[ x_{1}:=\mathsf{3_{1}4_{1}5_{2}6_{3}}, \hspace{2mm} x_{2}:=\mathsf{2_{1}4_{1}5_{2}6_{4}}, \hspace{2mm} x_{3}:=\mathsf{3_{1}4_{3}5_{2}6_{3}}, \hspace{2mm} x_{4}:=\mathsf{2_{1}3_{1}5_{2}6_{3}}, \]
\[ x_{5}:=\mathsf{2_{1}35_{2}6_{5}}, \hspace{2mm} x_{6}:=\mathsf{2_{1}4_{3}5_{2}6_{4}}, \hspace{2mm} x_{7}:=\mathsf{2_{1}3_{2}4_{3}5_{1}6_{2}}. \]
By GAP3 computations, we have that
\[ [\theta_{x_{4}}L(d): L(d)]=[\theta_{x_{5}}L(d): L(d)]=[\theta_{x_{6}}L(d): L(d)]=v^{5}+4v^{3}+7v+7v^{-1}+4v^{-3}+v^{-5}. \]
We have $\mathbf{a}(x_{4})=\mathbf{a}(x_{5})=\mathbf{a}(x_{6})=5$, so these hold by \Cref{Lem:IndCoeff1AtaFunc}. By (c) of \Cref{Prop:IndecompConjLRCellInv},
\begin{align*}
\mathsf{KM}(x_{1},d)&=\mathsf{KM}(dw_{0},w_{0}x_{1})=\mathsf{KM}(\mathsf{123456_{1}},\mathsf{12_{1}56_{5}}) \ \text{(Case (17))}, \\
\mathsf{KM}(x_{2},d)&=\mathsf{KM}(dw_{0},w_{0}x_{2})=\mathsf{KM}(123456_{1},\mathsf{12_{1}4_{3}56_{5}}) \ \text{(Case (7)(a))}, \\
\mathsf{KM}(x_{3},d)&=\mathsf{KM}(dw_{0},w_{0}x_{3})=\mathsf{KM}(123456_{1},\mathsf{13_{1}56_{4}}) \ \text{(Case (7)(b))}, \\
\mathsf{KM}(x_{7},d)&=\mathsf{KM}(dw_{0},w_{0}x_{7})=\mathsf{KM}(123456_{1},\mathsf{134_{3}6}) \ \text{(Case (20))},
\end{align*}
The first equality holds from Case (17), the second from Case (7)(a) since $\mathsf{12_{1}4_{3}56_{5}}\sim_{L}\mathsf{2_{1}4_{1}5_{2}6_{4}}$, the third from Case (7)(b) since $\mathsf{13_{1}56_{4}}\sim_{L}\mathsf{3_{1}4_{3}5_{2}6_{3}}$, while the forth from Case (20).

\textbf{Case} (1)(a): We have $d:=\mathsf{124_{1}56_{4}}$, $\mathsf{sh}(d)=(3,2,2)$, and $D(d)=\{1,4,6\}$. By \Cref{Lem:ICSimp} we only need to check $\mathsf{KM}(x_{i},d)=\mathtt{true}$ for each $i\in[2]$ where
\[ x_{1}:=\mathsf{123456_{1}}, \hspace{2mm} x_{2}:=\mathsf{14_{1}5_{3}6_{4}}. \]
By GAP3 computations, $[\theta_{x_{2}}L(d): L(d)]=v^{5}+7v^{3}+16v+16v^{-1}+7v^{-3}+v^{-5}$. Lastly, by (c) of \Cref{Prop:IndecompConjLRCellInv}, we have $\mathsf{KM}(x_{1},d)=\mathsf{KM}(dw_{0},w_{0}x_{1})=\mathsf{KM}(\mathsf{23_{1}5_{2}6_{4}},\mathsf{23_{2}4_{2}5_{2}})$. This is solved in Case (19) above since $23_{1}5_{2}6_{4}\sim_{L}2_{1}4_{3}5_{2}6_{4}$ (which equals $x_{6}$ from Case (19)).

All cases have now be solved, which results in the following:
\begin{thm} \label{Thm:IC7}
The Indecomposability Conjecture 
(\Cref{Conj:IndecompConj}) holds for $\mathfrak{sl}_{7}$.
\end{thm}
With an analogous proof to that of \Cref{Cor:KMxSmallSup}, we immediately have the following:
\begin{cor}\label{Cor:KMxSmallSup6}
Let $x\in\mathfrak{S}_{n}$ be such that $|\mathsf{Sup}(x)|\leq6$, then $\mathsf{KM}(x,\star)=\mathtt{true}$.
\end{cor}
We end with a brief discussion regarding K{\aa}hrstr{\"o}m's Conjecture. By \Cref{Thm:IC7} and \Cref{Eq:KiffKhKM}, we now know that $\mathsf{K}(w)=\mathsf{Kh}(w)$ for all $w\in\mathfrak{S}_{7}$. Therefore, to establish K{\aa}hrstr{\"o}m's Conjecture for $A_{6}$, it suffices to show that for all $d\in\mathfrak{I}_{7}$,
\begin{equation}\label{Eq:KhSuffice}
\mathsf{Kh}(d)=[\mathsf{Kh}](d)=[\mathsf{Kh}^{\mathsf{ev}}](d).
\end{equation}
These equalities are seen to immediately hold whenever $d$ is Kostant negative, i.e. $\mathsf{Kh}(w)=\mathtt{false}$. Moreover, K{\aa}hrstr{\"o}m's Conjecture was shown to hold for all fully commutative involutions in \cite[Theorem 5.1]{MMM24}, and we proved computationally that such equalities hold for the cases $(6)$, $(10)$, and $(11)$ in \Cref{SubSec:RemainingCases}. Thus, to prove that K{\aa}hrstr{\"o}m's Conjecture holds in $A_{6}$, it suffices to show that the above equalities hold for all Kostant positive involutions $d$ belonging to \Cref{Table:1}, of which there are 106 (99 if we remove 6 fully commutative involutions and the longest element $w_{0}$, which are known to uphold the above equalities). 

By running code in GAP3, we have confirmed the graded part of \Cref{Eq:KhSuffice} in almost all cases (so far) as follows.

\begin{prop}\label{Prop:GrKhGAP3}
Let $Q\subset \mathfrak{I}_{7}$ be the set consisting of the involutions
\[ \mathsf{12_{1}3_{2}4_{3}5_{1}6_{1}}, \hspace{1mm} \mathsf{2_{1}3_{1}4_{1}5_{1}6_{2}} \hspace{1mm} \text{ and } \hspace{1mm} \mathsf{123_{1}4_{1}5_{2}6_{1}}. \]
Then, for all $d\in\mathfrak{I}_{7}\setminus Q$, we have the equality $\mathsf{Kh}(d)=[\mathsf{Kh}](d)$.
\end{prop}

Computations to resolve the three remaining involutions continue.


\section*{Appendix: GAP3 Computations}
\label{Sec:Appendix}

In this brief appendix we give some examples of how we used the CHEVIE package in GAP3 (version of 2024, Jan. 7) to do the computations mentioned throughout the paper. From such examples, it will be easy to verify almost all of the computations done in this paper. We also discuss the only non-trivial piece of code which allowed for the verification of \Cref{Prop:GrKhGAP3}. 

\subsection*{Set up}

Once the CHEVIE package is installed, input the following into the GAP3 terminal:

\mybox{
\begin{lstlisting}[language=GAP]
gap> v:=X(Rationals);; v.name:="v";;
gap> W:=CoxeterGroup("A",6);; H:=Hecke(W,v^2,v);;
gap> D:=Basis(H,"D'");; C:=Basis(H,"C'");;
\end{lstlisting}
}

The first line sets up $\mathsf{v}$, which corresponds to the variable $v$ in the ring $\mathbb{A}=\mathbb{Z}[v,v^{-1}]$. The second line sets $\mathsf{W}$ as the symmetric group $\mathfrak{S}_{7}$, and $\mathsf{H}$ as the corresponding Hecke algebra $\mathcal{H}_{7}$. Lastly, the third line sets $\mathsf{D}$ and $\mathsf{C}$ as the dual and ordinary Kazhdan-Lusztig basis, respectively. Comparing to the notation used in the paper, we can think of $\mathsf{D}$ as the symbol $\underline{\hat{H}}$, and $\mathsf{C}$ as the symbol $\underline{H}$.

\subsection*{Example 1}

As given in Case (1)(a) in \Cref{SubSec:RemainingCases}, consider the permutations of $\mathfrak{S}_{7}$ given by 
\[ x:=12456_{2},  \hspace{1mm} y:=124_{2}56, \hspace{1mm} \text{ and } \hspace{1mm} d:=124_{1}56_{4}. \]
During that case it was claimed that, from computations, we have the decomposition
\begin{equation}\label{Eq:AppEx1-1}
\theta_{y}\theta_{65}\cong\theta_{123_{2}6}\oplus\theta_{123_{2}56_{5}}\oplus\theta_{x}, \hspace{1mm} \text{ equivalently } \hspace{1mm} \underline{H}_{65}\underline{H}_{y}=\underline{H}_{123_{2}6}+\underline{H}_{123_{2}56_{5}}+\underline{H}_{x}.
\end{equation}
The following is an example of how one could double check this computation in GAP3:

\mybox{
\begin{lstlisting}[language=GAP]
gap> C(C(6,5)*C(1,2,4,3,2,5,6));
C'(1,2,3,2,6)+C'(1,2,3,2,5,6,5)+C'(1,2,4,5,6,5,4,3,2)
\end{lstlisting}
}

The first line tells GAP3 to compute $\underline{H}_{65}\underline{H}_{y}$ in terms of the Kazhdan-Lusztig basis, and the second line is the output which agrees with the right hand side of the equality in \eqref{Eq:AppEx1-1}. 

Another claim made in this case is that we have the equality
\begin{equation}\label{Eq:AppEx1-2}
[\theta_{65}L(d)]=[L(d)]+(v+v^{-1})[L(d5)], \hspace{1mm} \text{ equivalently } \hspace{1mm} \underline{\hat{H}}_{d}\underline{H}_{65}=\underline{\hat{H}}_{d}+(v+v^{-1})\underline{\hat{H}}_{d5}.
\end{equation}
The following is an example of how one could double check this computation in GAP3:

\mybox{
\begin{lstlisting}[language=GAP]
gap> D(D(1,2,4,3,2,1,5,6,5,4)*C(6,5));
D'(1,2,4,3,2,1,5,6,5,4)+(v+v^-1)D'(1,2,4,3,2,1,5,4,6,5,4)
\end{lstlisting}
}

The first line tells GAP3 to compute $\underline{\hat{H}}_{d}\underline{H}_{65}$ in terms of the dual Kazhdan-Lusztig basis, and the second line is the output which agrees with the right hand side of the second equality in \eqref{Eq:AppEx1-2}.

\subsection*{Example 2}

As given in Case (20) in \Cref{Sec:IndecompConjA6}, consider the permutations of $\mathfrak{S}_{7}$ given by 
\[ x:=123456_{1} \hspace{1mm} \text{ and } \hspace{1mm} d:=134_{3}6. \]
In this case it was claimed, by computations, that 
\begin{equation}\label{Eq:AppEx2-1}
[\theta_{x}L(d):L(d)]=v^{3}+3v+3v^{-1}+v^{-3}, \hspace{1mm} \text{ equivalently } \hspace{1mm} [\underline{H}_{d}](\underline{H}_{d}\underline{H}_{x})=v^{3}+3v+3v^{-1}+v^{-3}.
\end{equation}
Recall that the expression $[\underline{H}_{d}](\underline{H}_{d}\underline{H}_{x})$ denotes the coefficient of $\underline{H}_{d}$ in the product $\underline{H}_{d}\underline{H}_{x}$ when written in terms of the Kazhdan-Lusztig basis, and the equivalence of the equations in \eqref{Eq:AppEx2-1} is from \cite[Proposition 3.3]{KMM23}, see also the discussion preceding \Cref{Lem:IndCoeff1AtaFunc}. 

The following is an example of how one could double check this computation in GAP3:

\mybox{
\begin{lstlisting}[language=GAP]
gap> x:=[1,2,3,4,5,6,5,4,3,2,1];; d:=[1,3,4,3,6];;
gap> Coefficient(C(d)*C(x),d);
v^3 + 3*v + 3*v^(-1) + v^(-3)
\end{lstlisting}
}

The first line sets $x$ and $d$ to be the Coxeter words described above, the second line tells GAP3 to calculate $[\underline{H}_{d}](\underline{H}_{d}\underline{H}_{x})$, and the third line is the output which agrees with \eqref{Eq:AppEx2-1}.

\subsection*{Table 1} Out of the three tables present in this paper, \Cref{Table:1} was the only one which was not produced by hand. The following is an example of how one could recover \Cref{Table:1} in GAP3:

\mybox{
\begin{lstlisting}[language=GAP]
gap> w0:=LongestCoxeterElement(W);;
gap> Invs:=[];;
gap> Table1:=[];;
gap> for i in [0,1,2,3,4,5] do
gap> for I in Combinations([1,2,3,4,5,6],i) do
gap> w0I:=LongestCoxeterElement(W,I);
gap> for w in Elements(ReflectionSubgroup(W,I)) do
gap> for d in Filtered(Elements(LeftCell(W,w*w0I*w0)), x -> x*x = ()) do
gap> if (CoxeterWord(W,d) in Invs)=false then
gap> Add(Invs,CoxeterWord(W,d));
gap> Add(Table1,[CoxeterWord(W,d),I,CoxeterWord(W,w)]);
gap> fi; od; od; od; od;
\end{lstlisting}
}

In summary, this code first sets $\mathsf{Invs}$ and $\mathsf{Table1}$ to be empty lists, it then runs over every subset $I\subset S_{7}$ and every permutation $w\in\langle I\rangle$, finds the unique involution $d$ such that $d\sim_{L}ww_{0}^{P}w_{0}$, adds it to the list $\mathsf{Invs}$, and adds the triple $(d,I,w)$ to the list $\mathsf{Table1}$ (as long as $d$ was not already present in $\mathsf{Invs}$). As a result, the list $\mathsf{Invs}$ now contains every involution $d$ from \Cref{Table:1}, and the list $\mathsf{Table1}$ contains the first three columns from \Cref{Table:1} in the form of triples $(d,I,w)$. 

For example, the first three columns in the sixth row of \Cref{Table:1} tell us that
\[ 12_{1}3_{1}45_{4}6_{1}\sim_{L}(24)w_{0}^{\{2,3,4\}}w_{0}. \]
This can be verified in GAP3 by checking that the triple $(12_{1}3_{1}45_{4}6_{1},\{2,3,4\},24)$ belongs to the list $\mathsf{Table1}$. To do this, for example, one can do the following:  

\mybox{
\begin{lstlisting}[language=GAP]
gap> [[1,2,1,3,2,1,4,5,4,6,5,4,3,2,1],[2,3,4],[2,4]] in Table1;
true
\end{lstlisting}
}

As for the fourth and eighth columns in \Cref{Table:1}, which record the truth value of Kostant's problem for the corresponding $w$ (equivalently the corresponding $d$), this information was input by hand since it is easily deduced for most of the permutations $w$, and from \Cref{SubSec:Kostant'sProblemforSmallerCases}.

\subsection*{Checking \Cref{Prop:GrKhGAP3}}

All of the computations mentioned above (which account for almost all the computations done in the paper) only take seconds (in real time) to compute. 
The only exception is the code used to confirm \Cref{Prop:GrKhGAP3}. For this,  we created an executable file for GAP3 to run, 
and due to its length, we do not include it here. However, we are happy to provide this file upon request.  

In summary, this code does the following:
\begin{itemize}
\item it goes over all  Kostant positive involution $d\in\mathfrak{I}_{7}$,
\item for each such $d$, it creates a list of all involution $d'\in\mathfrak{I}_{7}$
such that $d'\leq_R d$;
\item then, for any pair $d_1,d_2$ of two different elements on this list satisfying
$d_1\sim_J d_2$, it lists all  pairs
$(x,y)$ of elements such that $x\sim_R d_1$, $y\sim_R d_2$ and $x\sim_L y$;
\item finally,  for each such pair $(x,y)$, it confirmed the inequality 
$\underline{\hat{H}}_{d}\underline{H}_{x}\neq\underline{\hat{H}}_{d}\underline{H}_{y}$.
\end{itemize}
By \cite[Proposition~33]{CM25}, this implies $\mathsf{Kh}(d)=[\mathsf{Kh}](d)$. After running our code on two standard desktop computers for about two months of real time, three involutions remain to be resolved.

It is also worth remarking that we have ran another code which does what is described above in the four points except for the third point, where we only consider one such pair $(x,y)$ instead of all such pairs. This code has completed and confirms the inequality $\underline{\hat{H}}_{d}\underline{H}_{x}\neq\underline{\hat{H}}_{d}\underline{H}_{y}$ for all Kostant positive involution $d\in\mathfrak{I}_{7}$ and such pairs $(x,y)$. This on its own is not enough to confirm $\mathsf{Kh}(d)=[\mathsf{Kh}](d)$, since we have not considered all pairs $(x,y)$, but it is conjectured to be enough by \cite[Conjecture 41]{CM25}.


\vspace{2mm}

\noindent
Department of Mathematics, Uppsala University, Box. 480,
SE-75106, Uppsala, \\ SWEDEN, 
emails:
{\tt samuel.creedon\symbol{64}math.uu.se}\hspace{5mm}
{\tt mazor\symbol{64}math.uu.se}

\end{document}